\documentclass[10pt,reqno]{amsart}

\usepackage{float}
\usepackage{hyperref} 
\usepackage{amssymb,amsmath,amsfonts,caption,subcaption,graphicx,cite,latexsym,mathdots,enumerate,lscape,comment,mathrsfs,mathtools}

\graphicspath{{images/}}

\theoremstyle{definition} 
\newtheorem{Example}{Example}

\newtheorem*{Definition}{Definition} 
 
\newtheorem{Question}{Question} 
\newtheorem{Proposition}{Proposition}

\numberwithin{Theorem}{section}
\numberwithin{equation}{section}

\DeclareMathOperator{\stab}{Stab}
\DeclareMathOperator{\Irr}{Irr}

\DeclareMathOperator{\Aut}{Aut}
\DeclareMathOperator{\tr}{tr}

\newcommand{\C}{\ensuremath{\mathbb{C}}}

\newcommand{\T}{\ensuremath{\mathbb{T}}} 
\newcommand{\Z}{\ensuremath{\mathbb{Z}}}

\newcommand{\Y}{\mathcal{Y}} 
\newcommand{\X}{\mathcal{X}} 

\newcommand{\minimatrix}[4]{\begin{bmatrix}#1&#2\\#3&#4\end{bmatrix}}
\newcommand{\megamatrix}[9]{\begin{bmatrix} #1 & #2 & #3\\ #4 & #5 & #6 \\ #7 & #8 & #9 \end{bmatrix}}

\renewcommand{\vec}[1]{{\boldsymbol{\bf #1}}}

\allowdisplaybreaks

\title{The graphic nature of the symmetric group} 

\author[\tiny Brumbaugh]{J. L. Brumbaugh} 

\author[Bulkow]{Madeleine Bulkow} 
	\address{Department of Mathematics, Scripps College, 1030 Columbia Ave., Claremont, CA 91711} 

\author[Garcia]{Luis Alberto Garcia} 

\author[Garcia]{Stephan Ramon Garcia}
	\address{Department of Mathematics, Pomona College, 610 N. College Ave., Claremont, CA 91711} 
	\email{Stephan.Garcia@pomona.edu}
	\urladdr{\url{http://pages.pomona.edu/~sg064747}}
	
\author[Michal]{Matt Michal} \address{School of Mathematical Sciences, Claremont Graduate University, 150 E. 10th St., Claremont, CA 91711} 

\author[Turner]{Andrew P. Turner} \address{Department of Mathematics, Harvey Mudd College, 301 Platt Blvd., Claremont, CA 91711}

\thanks{Partially supported by NSF Grants DMS-1001614 and DMS-1265973.  We also gratefully acknowledge the support of the
Fletcher Jones Foundation and Pomona College's SURP Program.}
\begin{document}

\begin{abstract}
We investigate a remarkable class of exponential sums which are derived from the symmetric groups and which
display a diverse array of visually appealing features.   
Our interest in these expressions stems not only from their astounding visual properties, 
but also from the fact that they represent a novel and intriguing class of \emph{supercharacters}. 
\end{abstract}

\maketitle


\section{Introduction}

	Our aim in this note is to study a remarkable class of exponential sums which are derived from the symmetric groups
	and which display a diverse array of visually appealing features.  
	To be more specific, we are concerned here with the properties of certain complex-valued functions 
	associated to orbits in $(\Z/n\Z)^d$ arising from the natural permutation action 
	of the symmetric group $S_d$.
	The images of these functions, when plotted as subsets of the complex plane, display a wide variety of 
	striking features of great complexity and subtlety (see Figures \ref{FigureHook} and \ref{FigureHook2}).
	
	Our interest in these functions stems not only from their astounding visual properties, 
	but also from the fact that they represent a novel and 
	intriguing class of \emph{supercharacters}.  The theory of supercharacters, of which classical character 
	theory is a special case, was recently introduced in an axiomatic fashion by P.~Diaconis and 
	I.M.~Isaacs in 2008 \cite{DiIs08}, 
	who generalized earlier work of C.~Andr\'e \cite{An95, An01, An02}.  More recently,
	the study of supercharacters on abelian groups has proven surprisingly fruitful
	\cite{HendricksonNew, HendricksonThesis, SESUP, RSS, GNGP}.
	It is in this novel framework that our exponential sums arise.

	We make no attempt to survey the rapidly evolving literature on supercharacters and focus only on the
	essentials necessary for our investigations.
	To get started, we require the following important definition.

	\begin{Definition}[Diaconis-Isaacs \cite{DiIs08}]
		Let $G$ be a finite group, let $\X$ be a partition of the set $\Irr G$ of irreducible characters of $G$, 
		and let $\Y$ be a partition of $G$.  We call the ordered pair $(\X, \Y)$ a \emph{supercharacter theory} if 
		\begin{enumerate}\addtolength{\itemsep}{0.5\baselineskip}
			\item[(i)] $\Y$ contains $\{1\}$, where $1$ denotes the identity element of $G$,
			\item[(ii)] $|\X| = |\Y|$,
			\item[(iii)] For each $X$ in $\X$, the character $\sigma_X = \sum_{\chi \in X} \chi(1)\chi$
				is constant on each $Y$ in $\Y$.
		\end{enumerate}
		The characters $\sigma_X$ are called \emph{supercharacters} and the elements $Y$ of $\Y$
		 are called \emph{superclasses}.
	\end{Definition}

	\begin{figure}[H]
		\begin{subfigure}{0.45\textwidth}
			\centering
			\includegraphics[width=\textwidth]{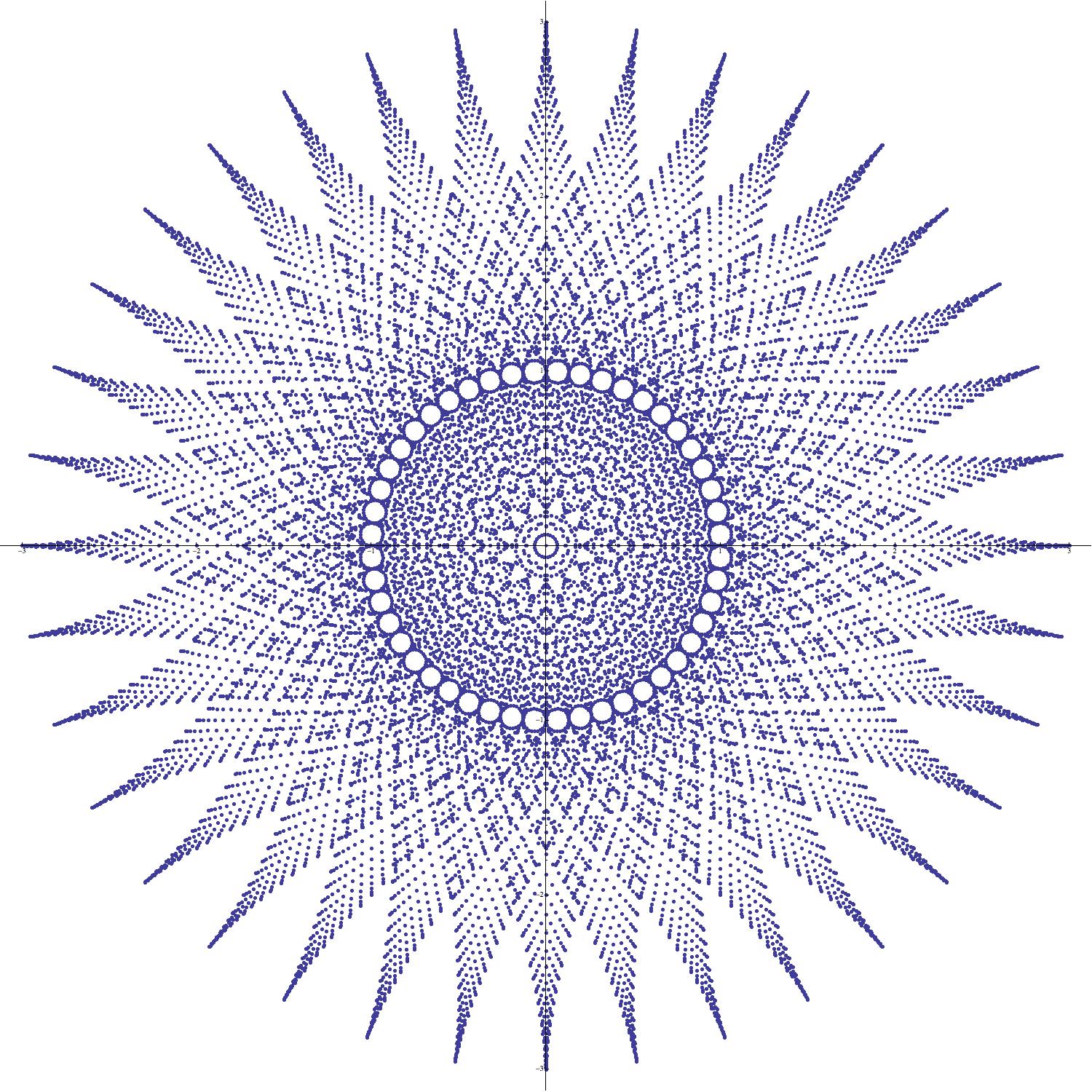}
			\caption{\scriptsize $n=96$, $d=3$, $X=S_3(1,6,1)$}
		\end{subfigure}		
		\qquad
		\begin{subfigure}{0.45\textwidth}
			\centering
			\includegraphics[width=\textwidth]{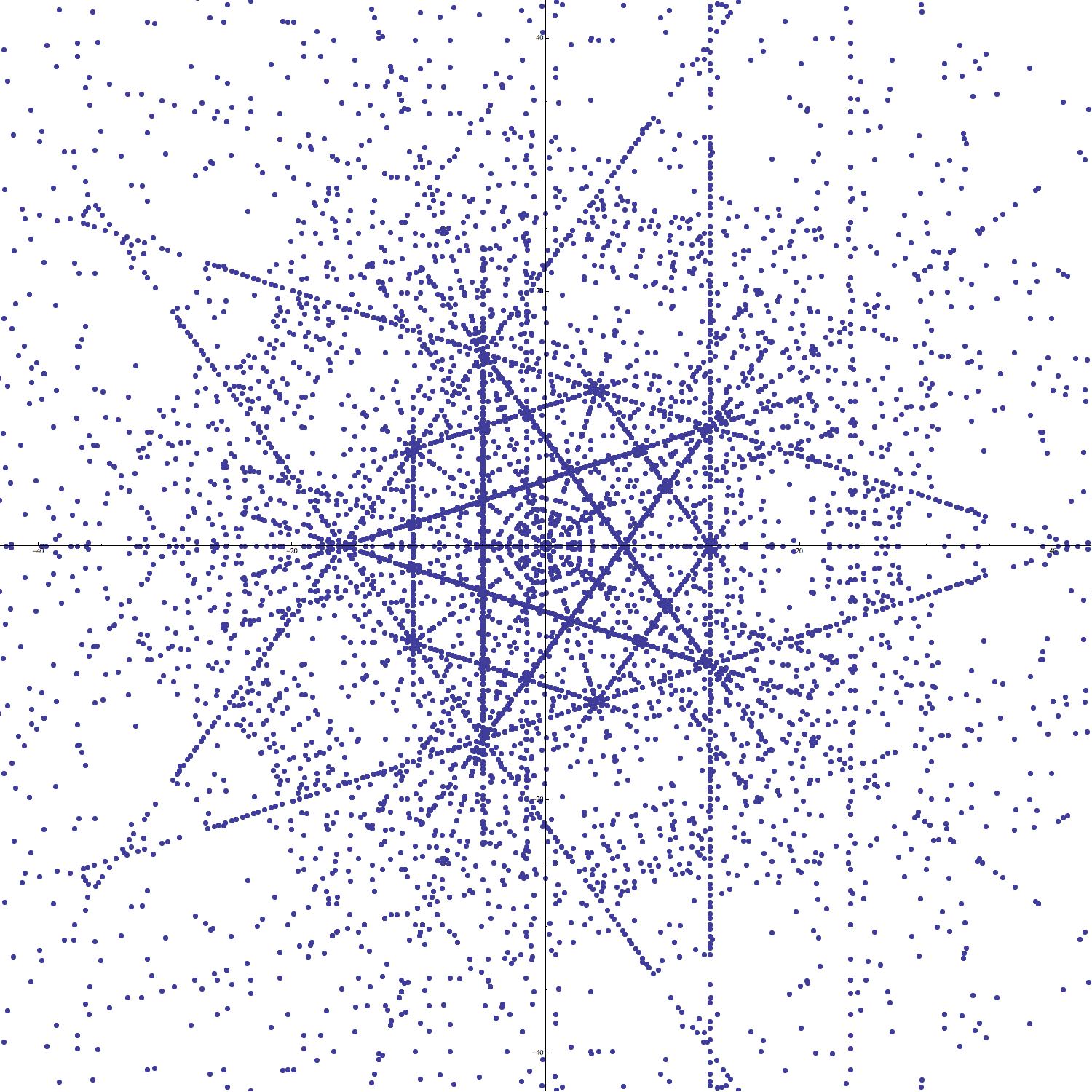}
			\caption{\scriptsize $n=10$, $d=8$, $X=S_8(0,1,3,3,3,3,3,8)$}
		\end{subfigure}	
		\\[5pt]	
		
		\begin{subfigure}{0.45\textwidth}
			\centering
			\includegraphics[width=\textwidth]{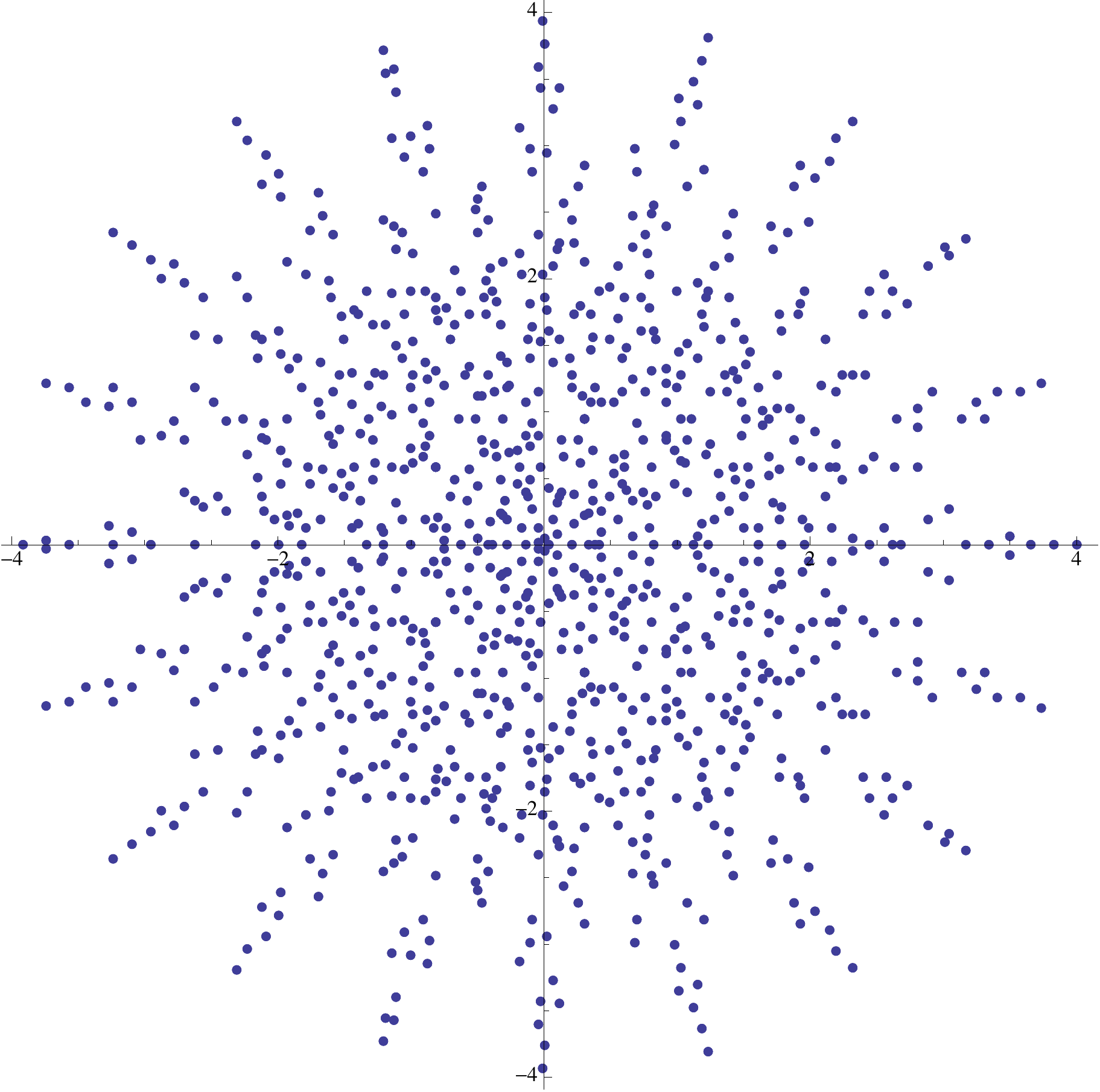}
			\caption{\scriptsize $n=15$, $d=4$, $X=S_4(1,1,1,3)$}
		\end{subfigure}		
		\qquad
		\begin{subfigure}{0.45\textwidth}
			\centering
			\includegraphics[width=\textwidth]{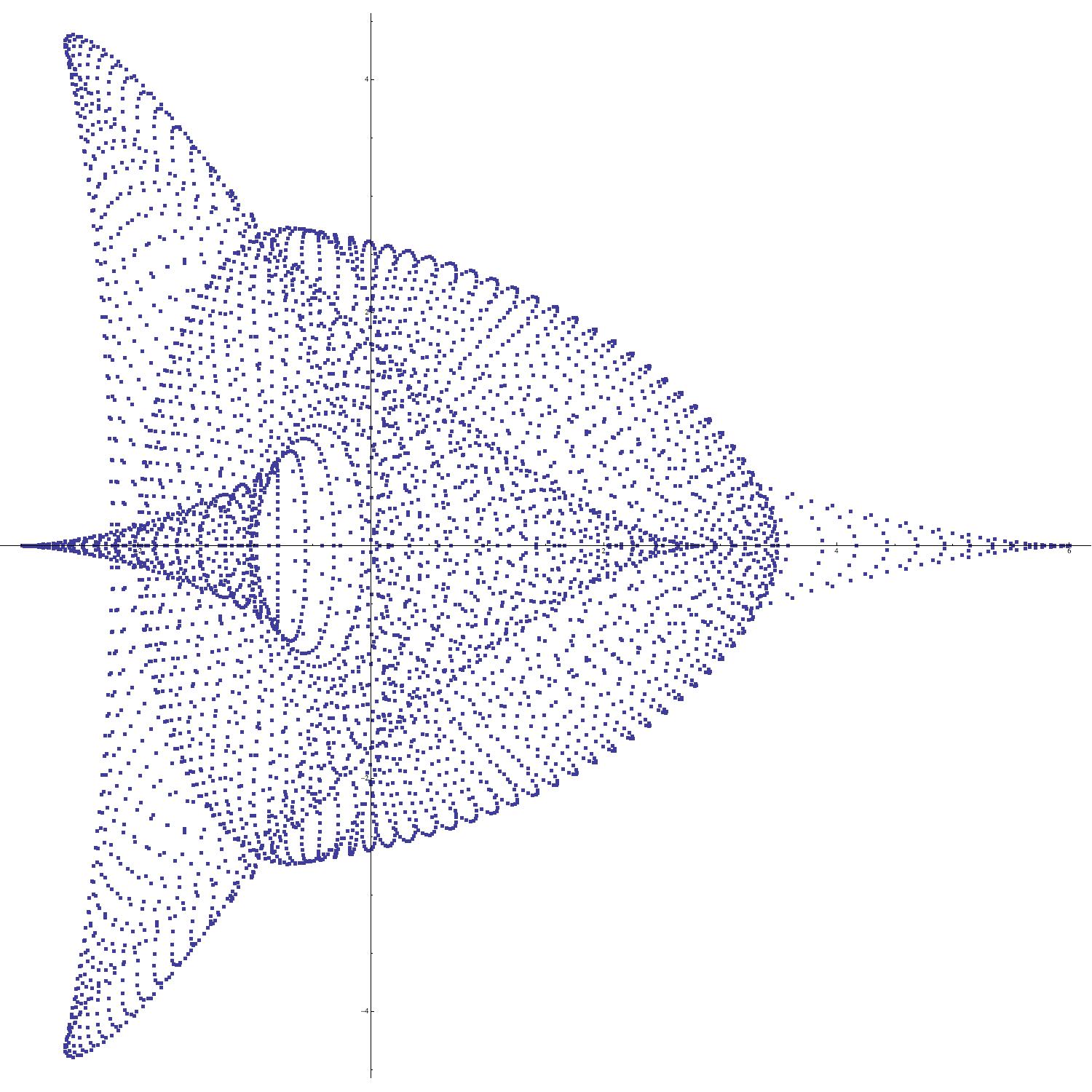}
			\caption{\scriptsize $n=173$, $d=3$, $X=S_3(1,2,170)$}
		\end{subfigure}				
		\\[5pt]

		\begin{subfigure}{0.45\textwidth}
	                \centering
	                \includegraphics[width=\textwidth]{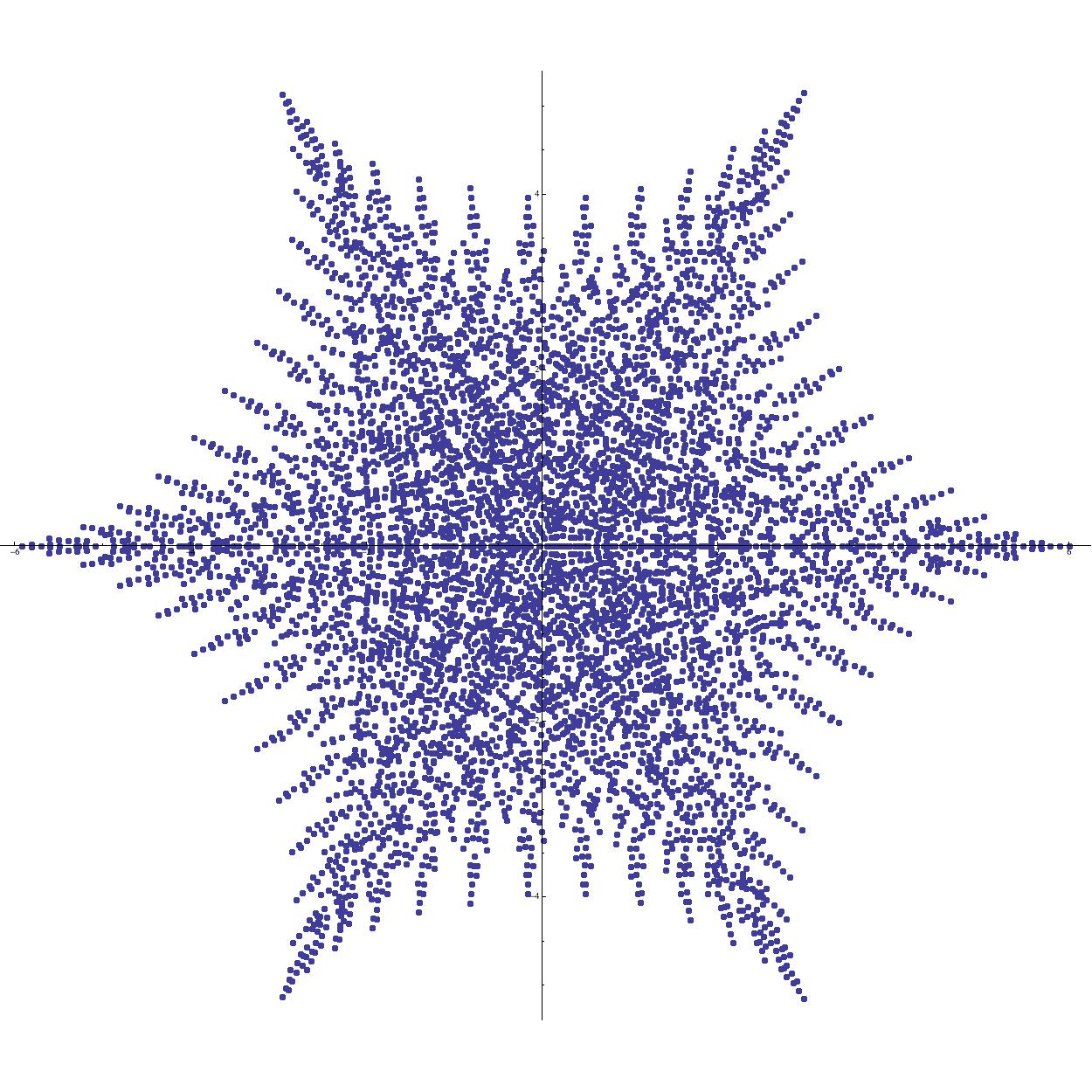}
	                \caption{\scriptsize $n=19$, $d=6$, $X = S_6(1,1,1,1,1,14)$}
	        \end{subfigure}
		\qquad
		\begin{subfigure}{0.45\textwidth}
	                \centering
	                \includegraphics[width=\textwidth]{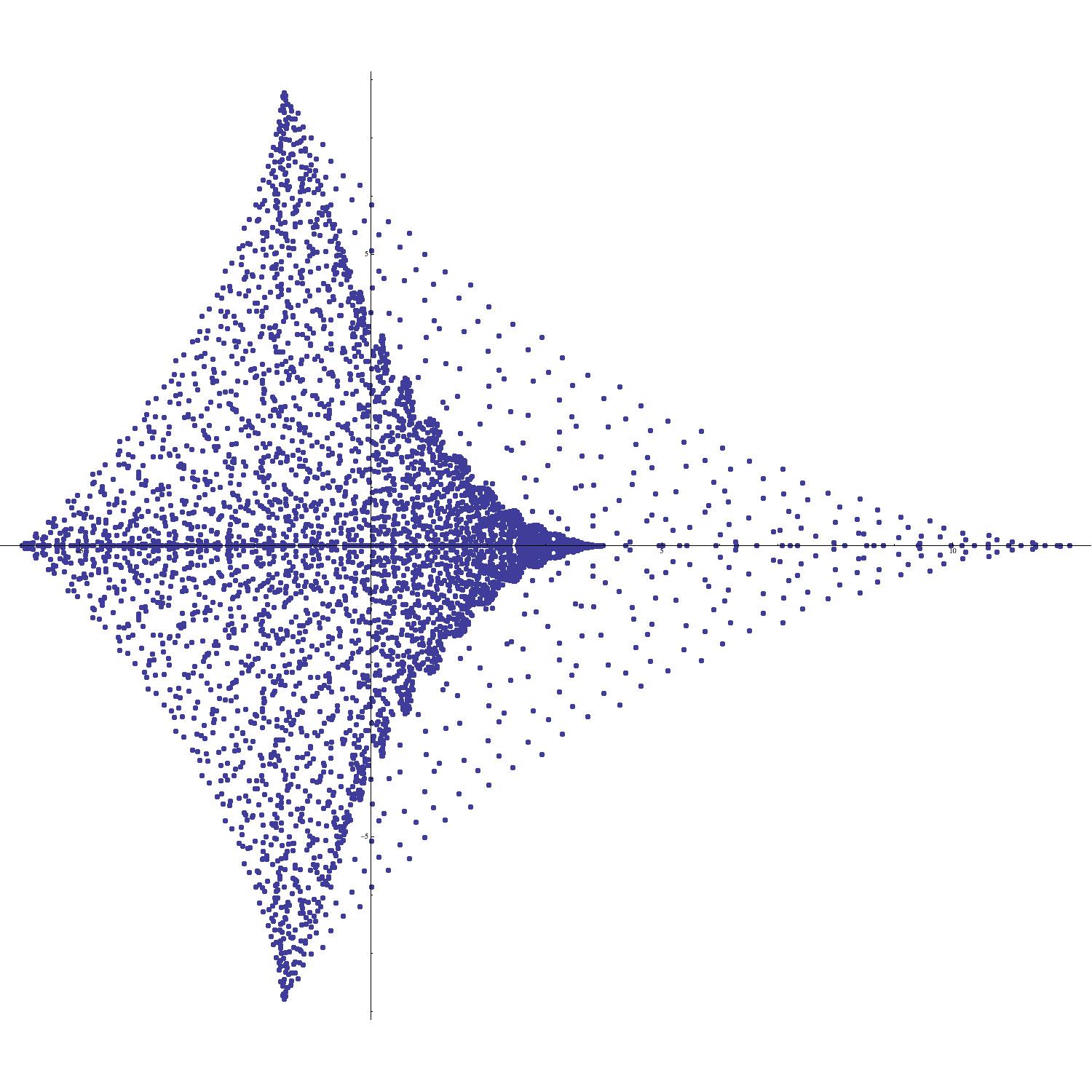}
	                \caption{\scriptsize $n=47$, $d=4$, $X = S_4(0,1,1,45)$}
	        \end{subfigure}
	        		
		\caption{Images of supercharacters $\sigma_X:(\Z/n\Z)^d\to\C$ 
		for various moduli $n$, dimensions $d$, and orbits $X$.}
		\label{FigureHook}
	\end{figure}

	\begin{figure}[H]
		\begin{subfigure}{0.45\textwidth}
			\centering
			\includegraphics[width=\textwidth]{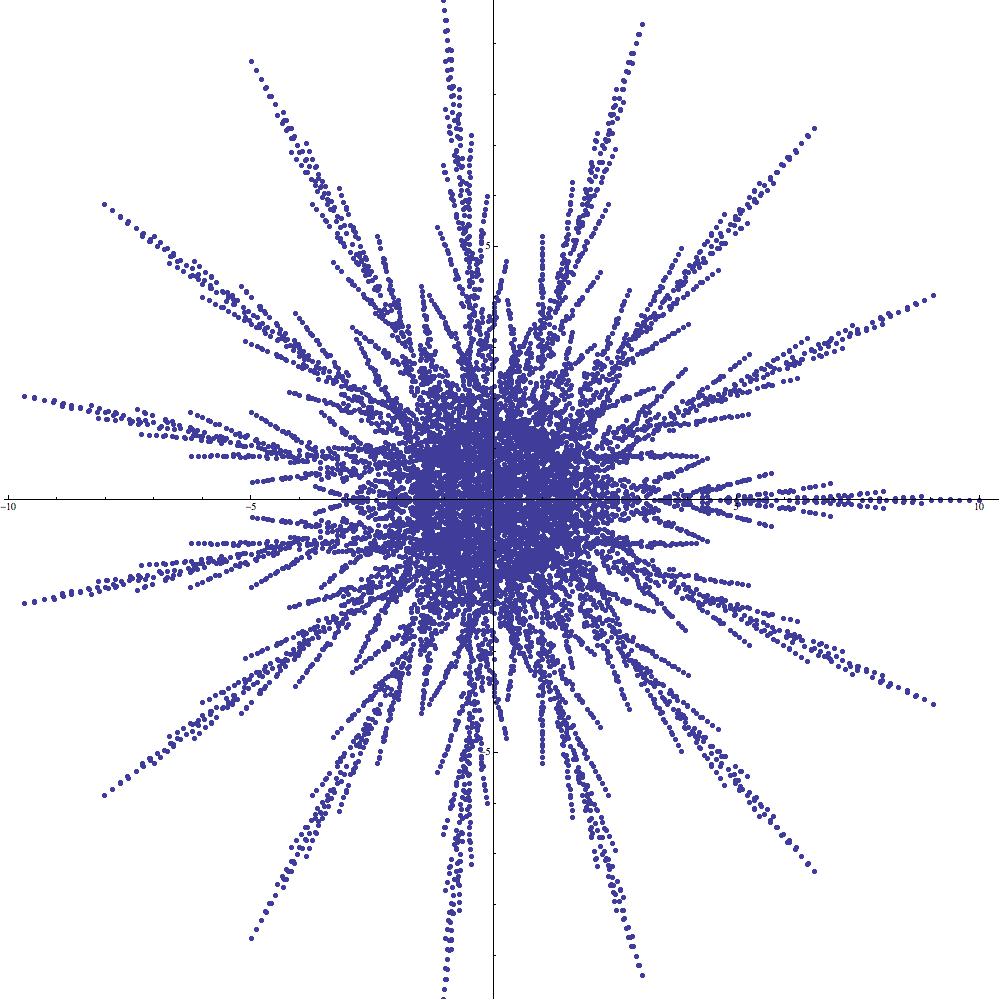}
			\caption{\scriptsize $n=24$, $d=5$, $X=S_5(1,1,2,2,2)$}
		\end{subfigure}		
		\qquad
		\begin{subfigure}{0.45\textwidth}
			\centering
			\includegraphics[width=\textwidth]{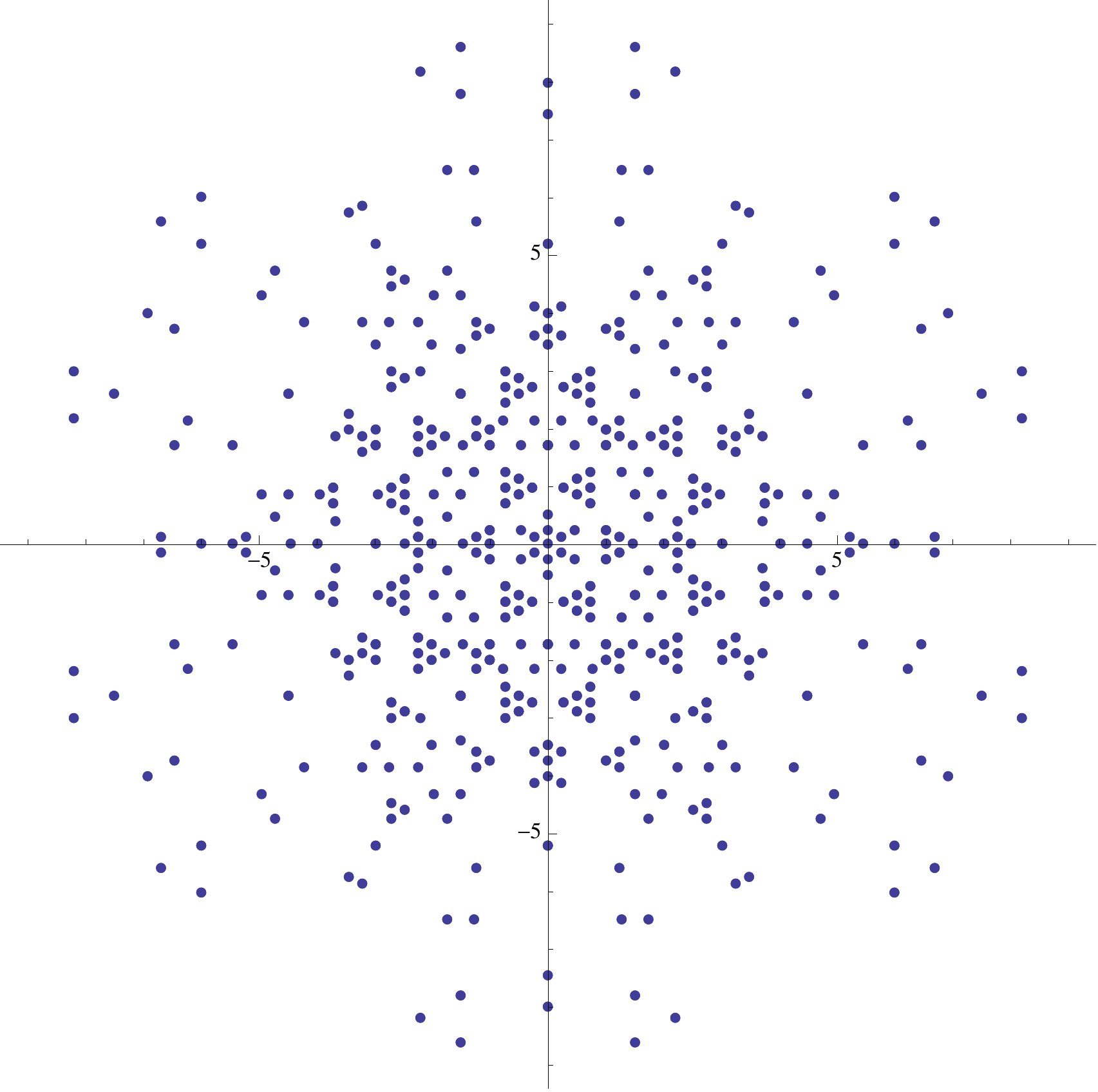}
			\caption{\scriptsize $n=12$, $d=4$, $X=S_4(0,3,3,4)$}
		\end{subfigure}	
		\\[5pt]	
		
		\begin{subfigure}{0.45\textwidth}
			\centering
			\includegraphics[width=\textwidth]{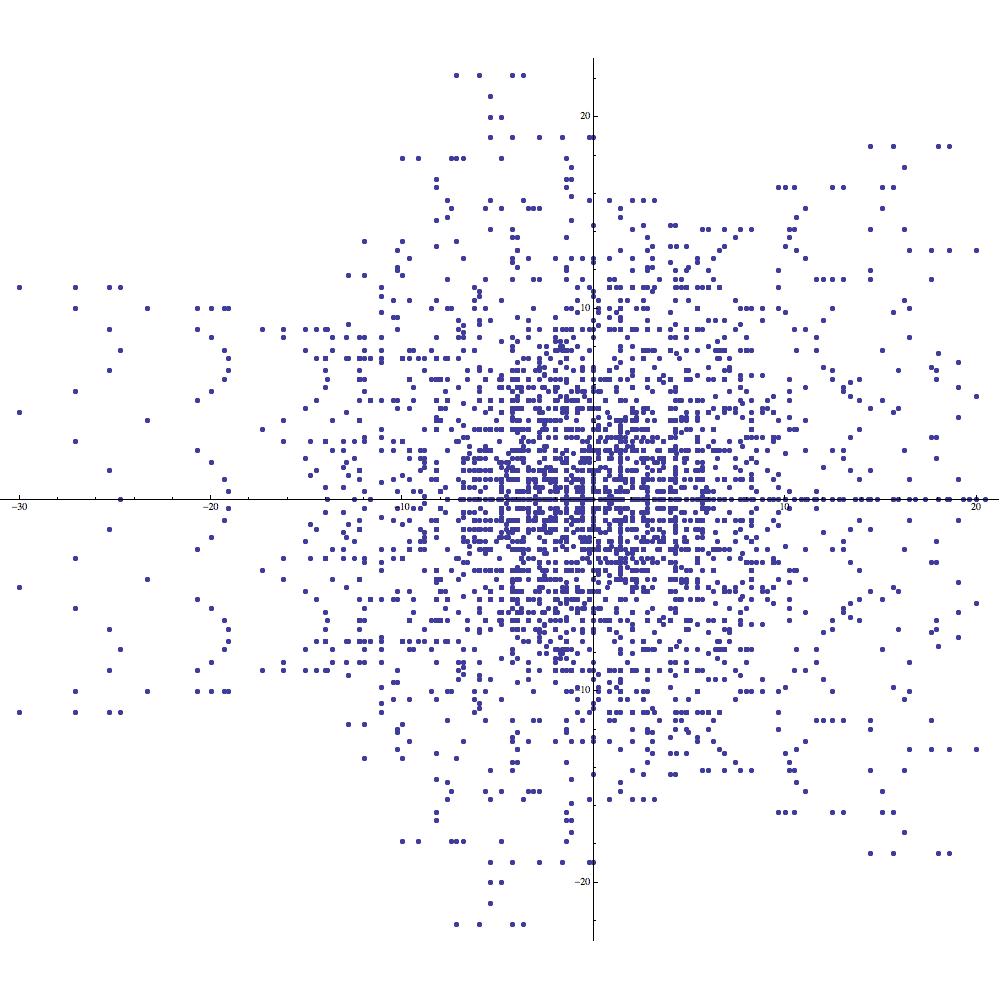}
			\caption{\scriptsize $n=16$, $d=7$, $X=S_7(3,5,8,8,8,8,8)$}
		\end{subfigure}		
		\qquad
		\begin{subfigure}{0.45\textwidth}
			\centering
			\includegraphics[width=\textwidth]{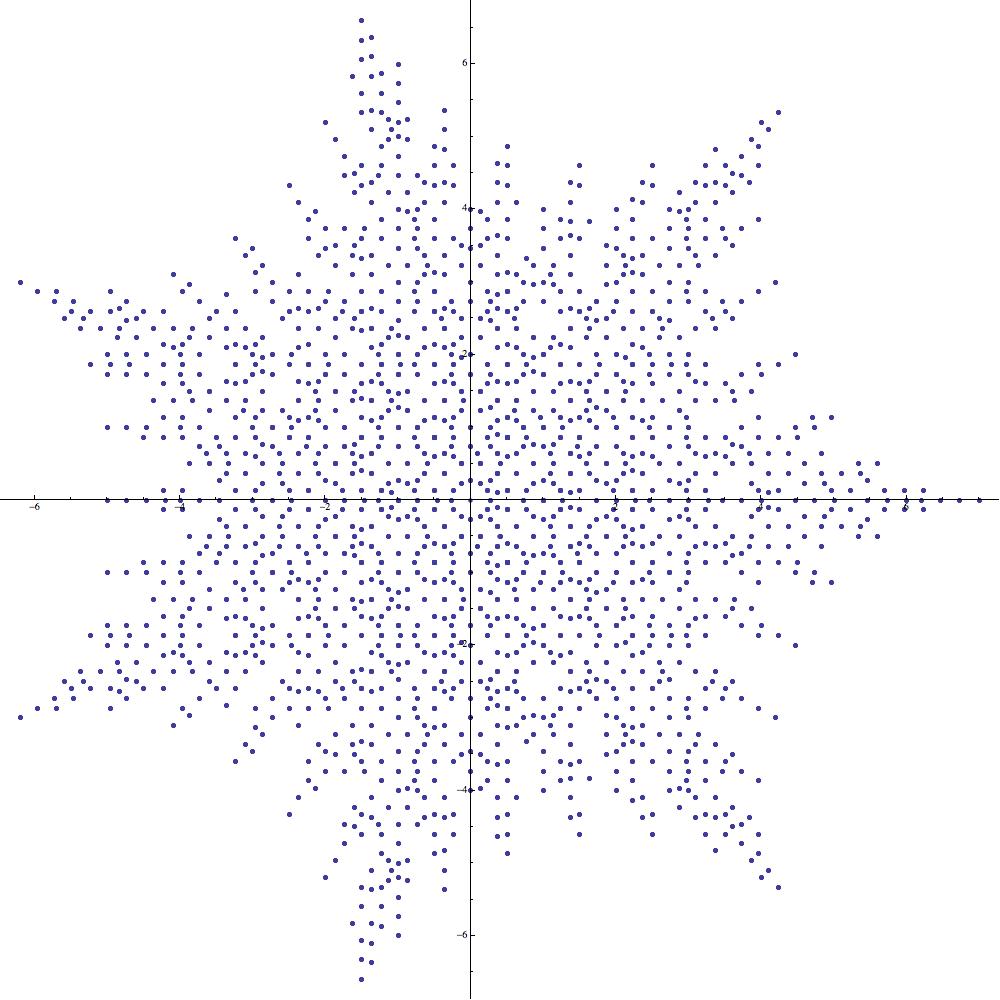}
			\caption{\scriptsize $n=12$, $d=7$, $X=S_7(1,1,1,1,1,1,6)$}
		\end{subfigure}	
		\\[5pt]	
		
		\begin{subfigure}{0.45\textwidth}
			\centering
			\includegraphics[width=\textwidth]{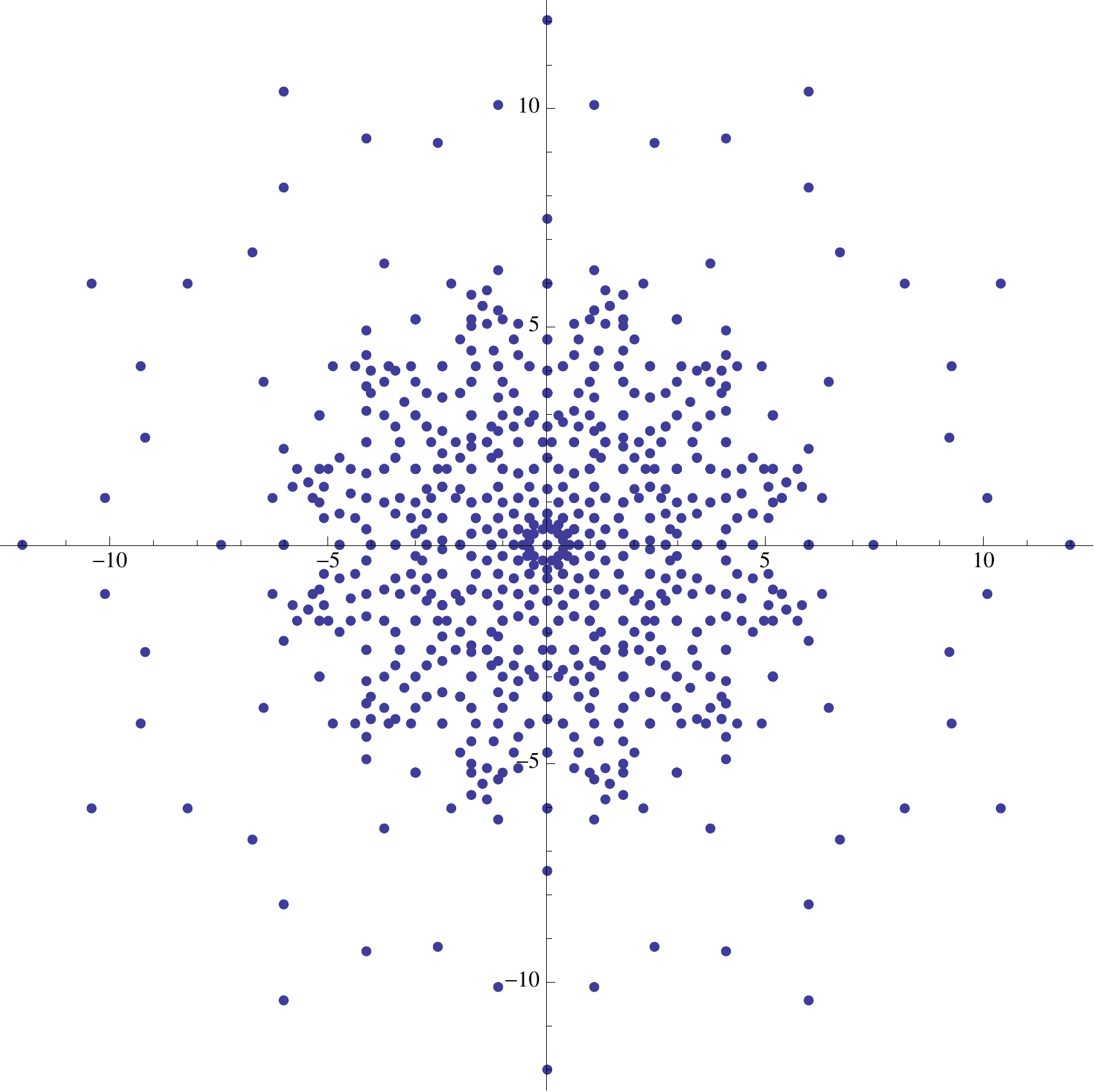}
			\caption{\scriptsize $n=12$, $d=4$, $X=S_4(3,4,4,6)$}
		\end{subfigure}		
		\qquad		
		\begin{subfigure}{0.45\textwidth}
			\centering
			\includegraphics[width=\textwidth]{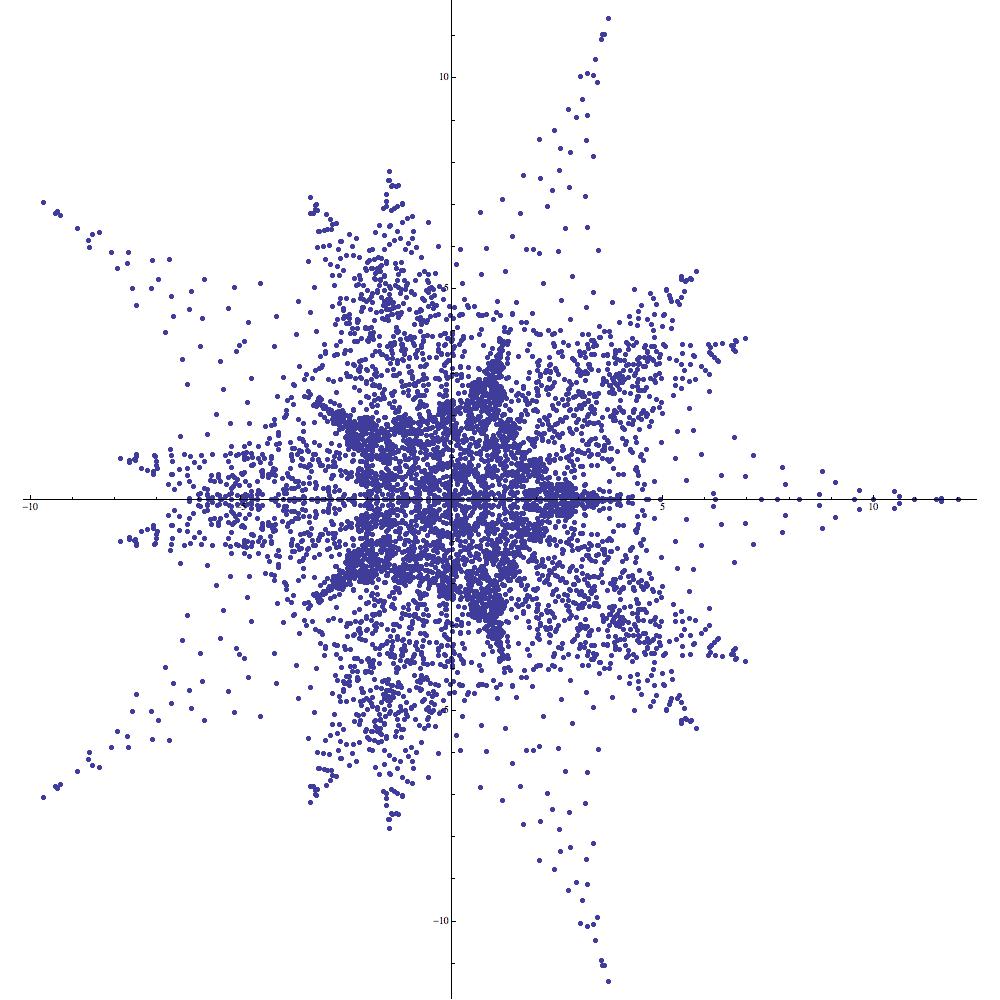}
			\caption{\scriptsize $n=30$, $d=4$, $X = S_4(4,6,7,7)$}
		\end{subfigure}
			        		
		\caption{Images of supercharacters $\sigma_X:(\Z/n\Z)^d\to\C$ 
		for various moduli $n$, dimensions $d$, and orbits $X$.}
		\label{FigureHook2}
	\end{figure}
		
	It turns out that  a variety of exponential sums 
	which are of interest in number theory can be viewed quite profitably
	as supercharacters on certain abelian groups \cite{RSS, SESUP, GNGP}.
	Although the sums we propose to study do not appear to have such celebrated names 
	attached to them, they are produced by the same mechanism as the 
	expressions displayed in Table \ref{FigureNT}.
	\begin{table}[h]
		\begin{equation*}\footnotesize
		\begin{array}{|c|c|c|c|}
			\hline
			\text{Name} & \text{Expression} & G & \Gamma \\
			\hline\hline
			\text{Gauss} & 
			\eta_j = \displaystyle \sum_{\ell=0}^{d-1} e\left( \frac{g^{k\ell+j}}{p}\right)
			& \Z/p\Z & \text{nonzero $k$th powers mod $p$} \\[20pt]
			\text{Ramanujan} & c_n(x)=\displaystyle \sum_{ \substack{ j = 1 \\ (j,n) = 1} }^n \!\!\!\! 
			e\left(\frac{ jx}{n} \right) & \Z/n\Z & (\Z/n\Z)^{\times} \\[20pt]
			\text{Kloosterman} & K_p(a,b)=\displaystyle \sum_{ \ell = 0  }^{p-1} e\left( \frac{a\ell + b \overline{\ell} }{p}\right)
			 & (\Z/p\Z)^2 & \left\{ \minimatrix{u}{0}{0}{u^{-1}}  : u \in (\Z/p\Z)^{\times} \right\} \\[20pt]
			\text{Heilbronn} & \displaystyle H_p(a)=\sum_{\ell=0}^{p-1} e\left(\frac{a \ell^p}{p^2} \right) & \Z/p^2\Z & 
			\footnotesize\text{nonzero $p$th powers mod $p^2$} \\[20pt]
			\hline
		\end{array}
		\end{equation*}	
		\caption{Gaussian periods, Ramanujan sums, Kloosterman sums, and Heilbronn sums 
		appear as supercharacters arising from the action of a subgroup $\Gamma$ of $\Aut G$ for a suitable
		abelian group $G$.  Here $p$ denotes an odd prime number.}
		\label{FigureNT}
	\end{table}
	
	One way to construct a supercharacter theory on a given finite group $G$ 
	is to employ the action of a subgroup
	$\Gamma$ of $\operatorname{Aut} G$ to obtain the partition $\Y$.  In this setting, the elements of $\Y$
	are precisely the orbits in $G$ under the action of $\Gamma$.  One then seeks a compatible action of
	$\Gamma$ on $\Irr G$ and appeals to a result of Brauer \cite[Thm.~6.32, Cor.~6.33]{Isaacs} 
	to conclude that the actions of $\Gamma$ on $G$ and on $\Irr G$ yield the same number of orbits.
	The precise details of the general case do not concern us here, 
	for the groups and automorphisms we consider below will be completely transparent.

	In the following,
	we denote by $\vec{x} = (x_1,x_2,\ldots,x_d)$ and $\vec{y}=(y_1,y_2,\ldots,y_d)$ 
	typical elements of $(\Z/n\Z)^d$ and let $\vec{x} \cdot \vec{y} = \sum_{i=1}^d x_i y_i$.
	We also set $e(x) = \exp(2 \pi i x)$, so that the function $e(x)$ is periodic with period $1$.
	We are interested here in the supercharacter theory on $(\Z/n\Z)^d$ induced by the natural permutation action of the symmetric
	group $S_d$.  In this setting, the superclasses are simply the orbits $S_d\vec{y} = \{ \pi(\vec{y}):\pi \in S_d\}$ 
	in $(\Z/n\Z)^d$.  Recalling that $\Irr G = \{ \psi_{\vec{x}} : \vec{x} \in G \}$, where
	$\psi_{\vec{x}}(\vec{y}) = e ( \frac{\vec{x} \cdot \vec{y}}{n} )$,
	we let $S_d$ act upon $\Irr G$ in the obvious manner by setting $\psi_{\vec{x}}^{\pi} = \psi_{\pi(\vec{x})}$.
	Letting $\X$ denote the set of orbits in $\Irr G$ and $\Y$ denote the set of orbits
	in $G$, it is clear that $|\X| = |\Y|$.  We often denote this common value by $N$.
		
	Although the elements of each orbit $X$ in $\X$ are certain characters $\psi_{\vec{x}}$,
	we identify $\psi_{\vec{x}}$ with the vector $\vec{x}$ so that $X$
	is stable under the action $\vec{x} \mapsto \pi (\vec{x})$ of $S_d$.
	Having established this convention, 
	for each $X$ in $\X$ we define the corresponding character
	\begin{equation}\label{eq-Supercharacter}
		\sigma_X(\vec{y}) = \sum_{\vec{x} \in X} e \left( \frac{ \vec{x} \cdot \vec{y} } {n} \right),
	\end{equation}
	noting that $\sigma_X(\vec{y}) = \sigma_X(\vec{y}')$ whenever $\vec{y}$ and $\vec{y}'$
	belong to the same $S_d$ orbit.\footnote{If $X = S_d\vec{x}$, then 
		$\sigma_X(\vec{y}) = \operatorname{per} \big( e(\frac{x_j y_k}{n}) \big)_{j,k=1}^d$, where $\operatorname{per}$
		denotes the \emph{permanent} of a matrix.}  
	The pair $(\X,\Y)$ 
	constructed above is a supercharacter theory on $(\Z/n\Z)^d$.\footnote{One 
	can also view this endeavor in terms of the classical character theory of the semidirect
	product $(\Z/n\Z)^d \rtimes S_d$ (sometimes referred to as a \emph{generalized symmetric group}).
	However, the supercharacter approach cleaner and more natural since $(\Z/n\Z)^d \rtimes S_d$ 
	is highly nonabelian and possesses a large number of conjugacy classes, whereas $(\Z/n\Z)^d$ is abelian and, by comparison, has
	relatively few superclasses.  Moreover, many of the irreducible
	characters of $(\Z/n\Z)^d \rtimes S_d$ are uninteresting for our purposes
	(e.g., assuming only $0$ or $n$th roots of unity as values).}
	
	We refer to the characters 
	\eqref{eq-Supercharacter} as \emph{symmetric supercharacters} (or often just \emph{supercharacters})
	and the sets $Y$ as \emph{superclasses}.  Expanding upon the notational liberties we have taken above,
	we choose to identify $X$, whose elements are the irreducible characters which comprise
	$\sigma_X$, with the set of vectors $\{ \vec{x} : \psi_{\vec{x}} \in X\}$.  
	Having made this identification, we see
	that $\X = \Y$.  In light of this, we frequently regard the elements $X$ of $\X$ as superclasses.
	Since $\sigma_X$ is constant on each superclass $Y$, if $\vec{y}$ belongs to $Y$ we often write
	$\sigma_X(Y)$ instead of $\sigma_X(\vec{y})$.
	
	In addition to \eqref{eq-Supercharacter}, there is another description of symmetric 
	supercharacters which is more convenient in certain circumstances.  Letting
	$\stab(\vec{x}) = \{ \pi\in S_d : \pi( \vec{x}) = \vec{x} \}$,
	it follows that the orbit $X = S_d \vec{x}$ contains $|\stab(\vec{x})|$ copies of $\vec{x}$ whence
	\begin{equation} \label{eq-Stab}
		\sigma_X(\vec{y}) = \frac{1}{|\stab(\vec{x})|}\sum_{\pi \in S_d} e \left( \frac{  \pi(  \vec{x}) \cdot \vec{y} } {n} \right).
	\end{equation}
	
	Although we do not need the following observation for our work, it is worth mentioning since it indicates that
	symmetric supercharacters also enjoy a variety of nontrivial algebraic properties.
	Fixing an enumeration $X_1,X_2,\ldots,X_N$ of $\X = \Y$, we label the 
	supercharacters corresponding to these sets $\sigma_1,\sigma_2,\ldots,\sigma_N$.  It is possible to show
	that the $N \times N$ matrix	
	\begin{equation}\label{eq-U}
		U = \frac{1}{\sqrt{n^d}} \left(  \frac{   \sigma_i(X_j) \sqrt{  |X_j| }}{ \sqrt{|X_i|}} \right)_{i,j=1}^N
	\end{equation}
	is symmetric (i.e., $U = U^T$) and unitary.  In fact, $U$ encodes a type of Fourier transform on the space
	of superclass functions (i.e., functions $f:(\Z/n\Z)^d\to\C$ which are constant on each superclass).  
	This is a consequence of more general considerations, of which 
	symmetric supercharacters are a special case.
	Complete details can be found in \cite{SESUP}, although some of these ideas are already present in 
	\cite{RSS,CKS}.

	We are now ready to proceed to the heart of the paper (Section \ref{SectionGraphics}), 
	where a variety of features of symmetric supercharacters are surveyed, documented, and explored.
	As one might expect from a paper appearing in this particular venue, most of the phenomena described below 
	were discovered experimentally.  While we are able to explain some of these phenomena, many others
	remain largely mysterious.  In light of this, we conclude 
	with a number of open problems motivated by our numerical investigations (Section \ref{SectionOpen}).
	We invite other mathematicians, perhaps those armed with more sophisticated tools than we here possess, to continue
	our work, for the study of symmetric supercharacters is surely fertile ground.  To this end, we include in
	Appendix \ref{SectionCode} the \texttt{Mathematica} code used to produce our supercharacter plots.

\section{Graphical properties}\label{SectionGraphics}
	Our approach here will be to survey and explain a variety of intriguing qualitative 
	features exhibited by the images of symmetric supercharacters.
	We proceed roughly in order of increasing complexity, starting here with the most elementary properties.

\subsection{Maximum modulus}

	Since there are precisely $|X|$ terms of unit modulus which comprise the sum which defines $\sigma_X$,
	it follows that $|\sigma_X(\vec{y})| \leq |X|$ for all $\vec{y}$ in $(\Z/n\Z)^d$. 
	Setting $\vec{y} = \vec{0}$ reveals that this inequality is sharp.
	If $X = S_d\vec{x}$ where the vector $\vec{x}$ in $(\Z/n\Z)^d$ has precisely $k_1,k_2,\ldots,k_n$ occurrences 
	of the elements $1,2,\ldots,n$ (so that $k_1+k_2+\cdots+k_n = d$), then it follows that
	\begin{equation*}
		|\sigma_X(\vec{y})| \leq \frac{d!}{k_1!k_2!\cdots k_n!}.
	\end{equation*}
	
\subsection{Conjugate symmetry}\label{SubsectionConjugate}
	For each superclass $Y$, the set $-Y$ obtained by negating each element of $Y$ is
	another superclass.  The definition \eqref{eq-Supercharacter} now yields
	\begin{equation}\label{eq-ConjugatePairs}
		\sigma_{X}(-Y) = \overline{\sigma_X(Y)} = \sigma_{-X}(Y).
	\end{equation}	
	Thus the image of any $\sigma_X$ is symmetric with respect to the real axis
	(see Figure \ref{FigureConjugate}).  Furthermore, if $X = -X$, then
	$\sigma_X$ is real-valued (see Figure \ref{FigureReal}).
	
	\begin{figure}[h]
		\begin{subfigure}{0.3\textwidth}
			\centering
			\includegraphics[width=\textwidth]{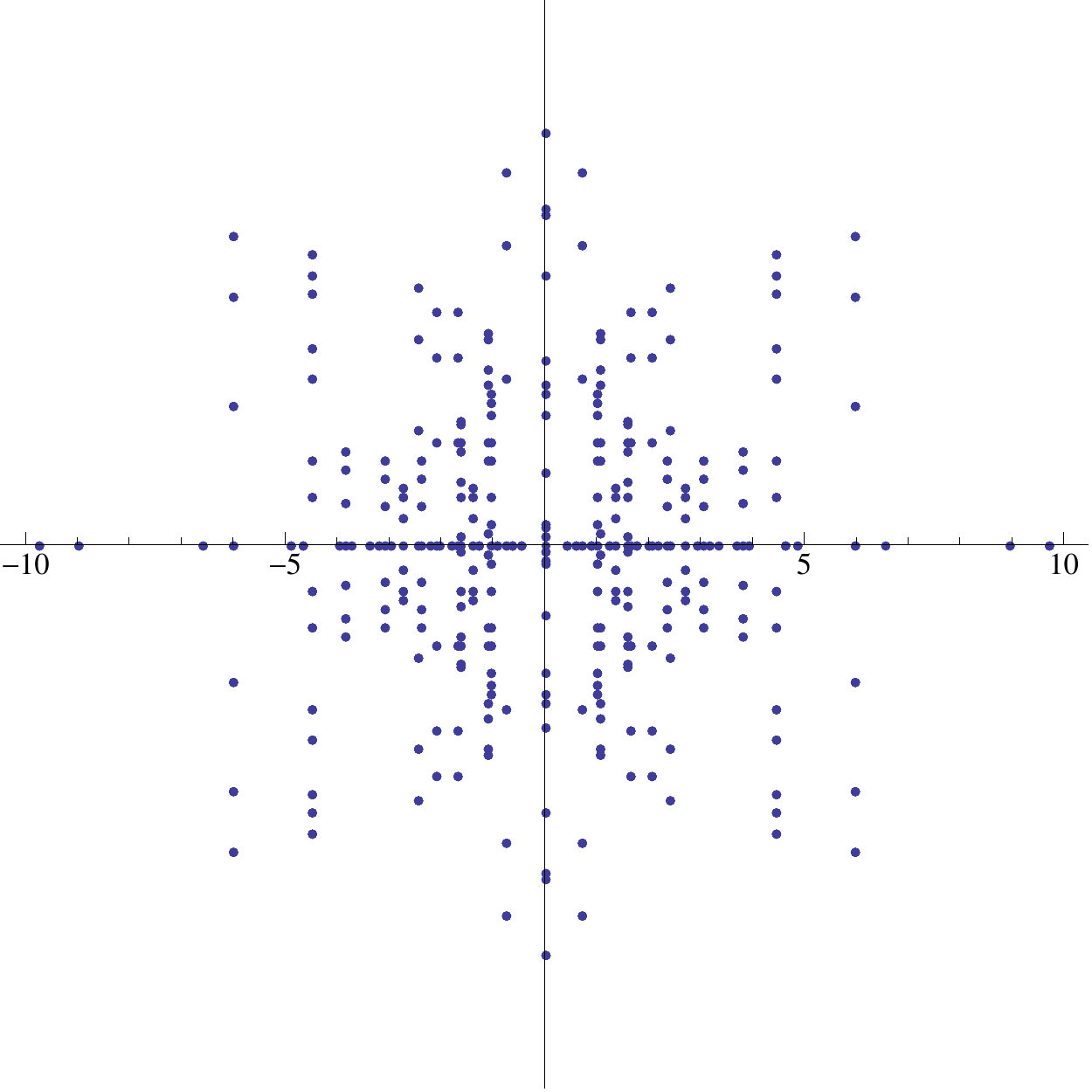}\quad
			\caption{$X = S_4(0,0,1,6)$}
		\end{subfigure}
		\quad
		\begin{subfigure}{0.3\textwidth}
			\centering
			\includegraphics[width=\textwidth]{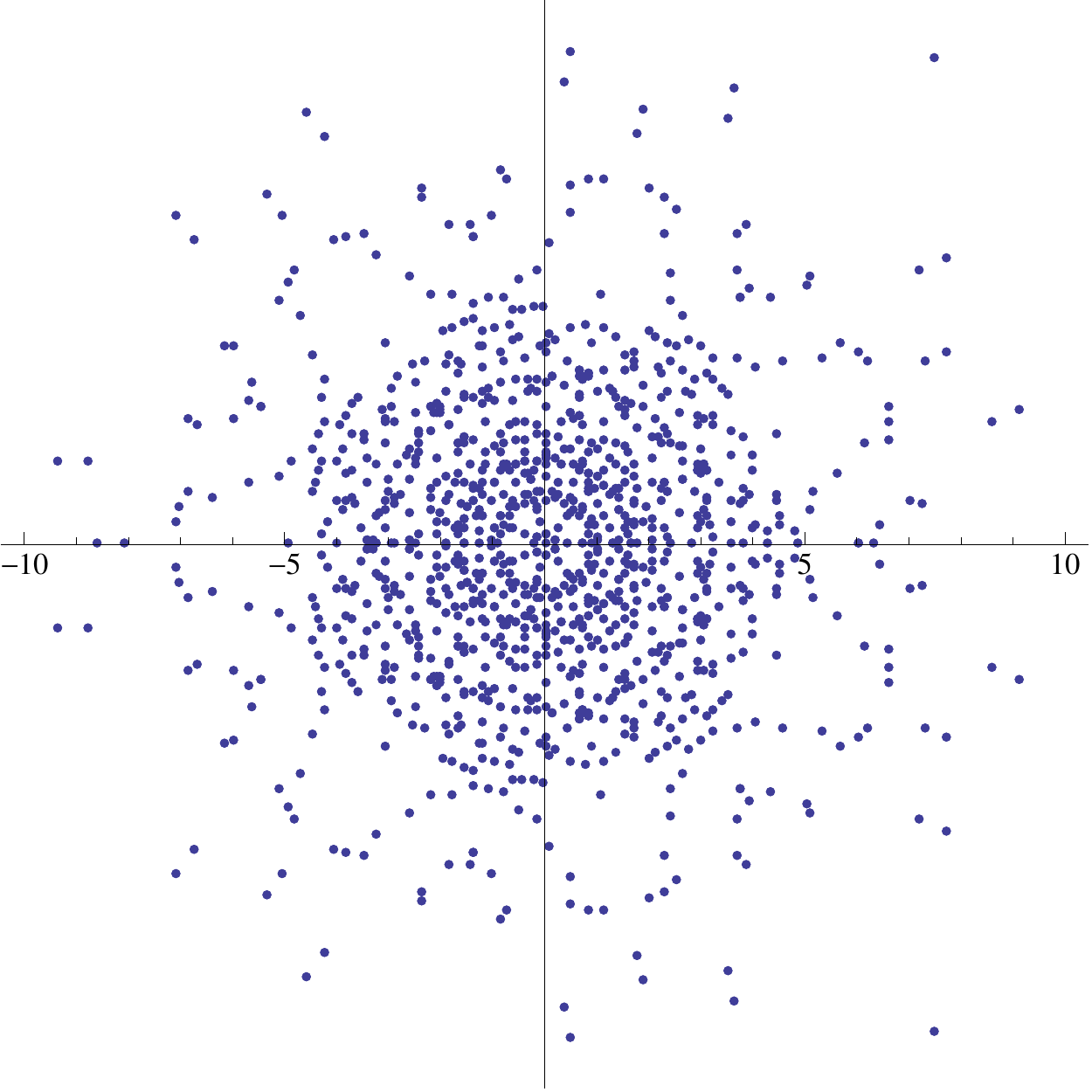}
			\caption{$X = S_4(0,1,1,6)$}
		\end{subfigure}
		\quad
		\begin{subfigure}{0.3\textwidth}
			\centering
			\includegraphics[width=\textwidth]{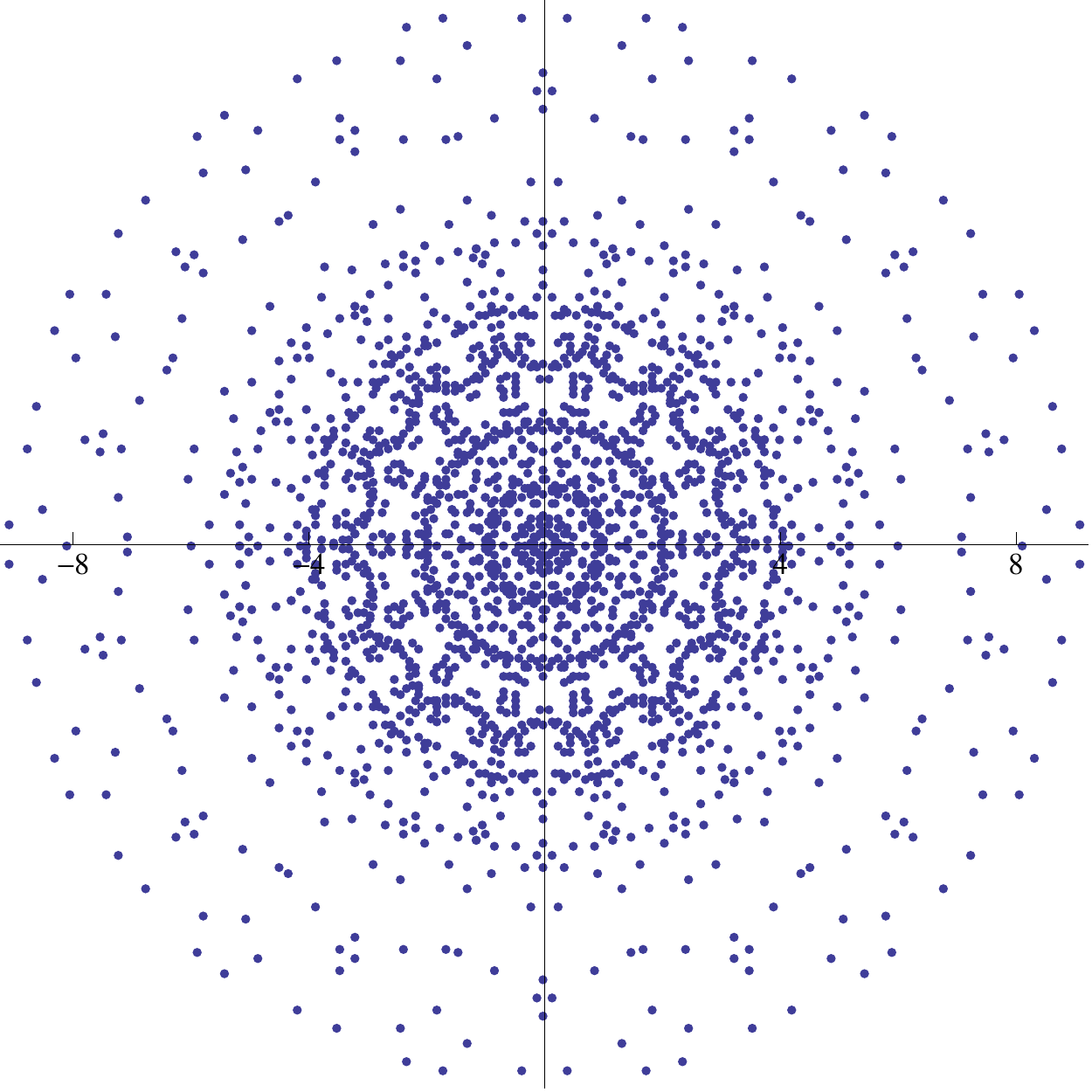}
			\caption{$X = S_4(0,1,6,6)$}
		\end{subfigure}
		\caption{Images of supercharacters $\sigma_X:(\Z/14\Z)^4\to\C$ for various $X$.
		Each plot is symmetric with respect to the real axis.}
		\label{FigureConjugate}
	\end{figure}
	
	\begin{figure}[h]
		\begin{subfigure}{\textwidth}
			\centering
			\includegraphics[width=0.9\textwidth]{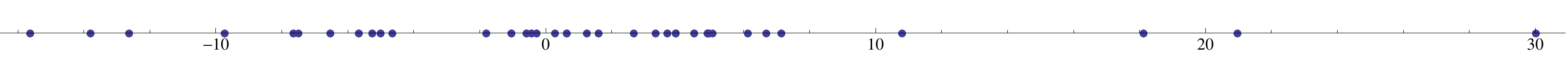}
			\caption{$n=7$, $d=5$, $X=S_5(0,1,1,6,6)$, $|X| = 30$}
		\end{subfigure}
		\begin{subfigure}{\textwidth}
			\centering
			\includegraphics[width=0.9\textwidth]{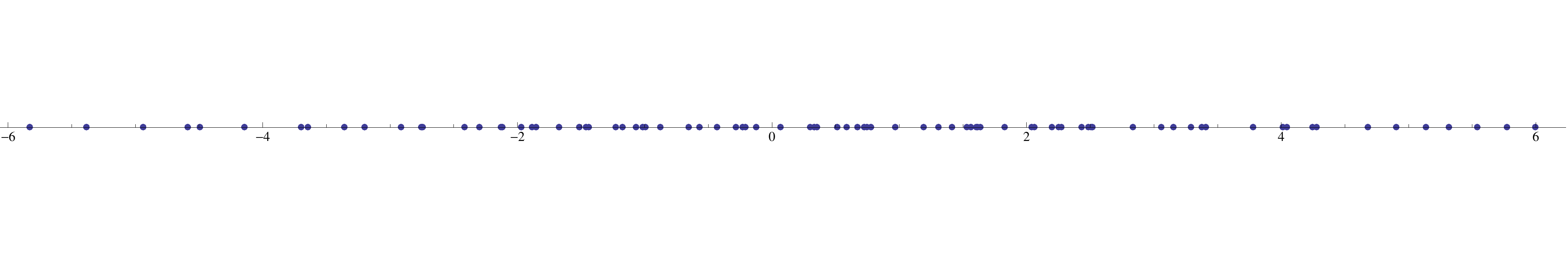}
			\vspace{-30pt}
			\caption{$n=13$, $d=4$, $X=S_4(6,6,7,7)$, $|X|=6$}
		\end{subfigure}
		
		\caption{Images of two real-valued supercharacters.  In each case, the superclass $X$ is closed under negation.}
		\label{FigureReal}
	\end{figure}

\subsection{Dihedral Symmetry}

	Let us introduce several notational conventions which will prove useful in what follows.
	For each $\vec{x}=(x_1,x_2,\ldots,x_d)$ in $(\Z/n\Z)^d$, we let
	\begin{equation*}
		[\vec{x}] := \sum_{\ell =1}^d x_{\ell} \pmod{n},
	\end{equation*}
	observing that $[\vec{x}] = [\pi(\vec{x})]$ for all $\pi$ in $S_d$.  For each superclass $X$
	we may define $[X]$ unambiguously by setting it equal to $[\vec{x}]$ for any representative $\vec{x}$ of $X$.
	We also let $\vec{1} = (1,1,\ldots,1)$ so that
	$X + j \vec{1} =\{\vec{x}+j\vec{1} : \vec{x} \in X\}$ is superclass whenever $X$ is.

	We say that a subset of the complex plane has \emph{$k$-fold dihedral symmetry} if it is invariant
	under the action of the dihedral group of order $2k$.  To establish that a 
	given supercharacter plot enjoys $k$-fold dihedral symmetry it suffices to show that it
	is invariant under rotation through an angle of $\frac{2\pi}{k}$.  The degree of symmetry
	enjoyed by a supercharacter plot is governed by the quantity $[X]$ (see Figure \ref{FigureSymmetry}).
	
	\begin{Proposition}\label{PropositionDihedral}
		The image of a symmetric supercharacter $\sigma_X:(\Z/n\Z)^d\to\C$ has $\frac{n}{(n,[X])}$-fold dihedral symmetry.
	\end{Proposition}
	
	\begin{proof}
		First observe that
		\begin{equation*}
			\sigma_X(\vec{y} + \ell\vec{1}) 
			= \sum_{\vec{x}\in X} e\left( \frac{\vec{x}\cdot(\vec{y}+\ell\vec{1})}{n}\right) 
			= e\left( \frac{[\vec{x}]\ell}{n} \right)\sum_{\vec{x}\in X}  e\left( \frac{\vec{x}\cdot\vec{y}}{n}\right) 
			= e\left( \frac{[X]\ell}{n} \right) \sigma_X(\vec{y}).
		\end{equation*}
		Since the congruence $[X]\ell \equiv b \pmod{n}$ is solvable if and only if $(n,[X])$ divides $b$, it follows that 
		the image of $\sigma_X$ has $\frac{n}{(n,[X])}$-fold dihedral symmetry. 
	\end{proof}

	\begin{figure}[h]
		\begin{subfigure}{0.3\textwidth}
			\centering
			\includegraphics[width=\textwidth]{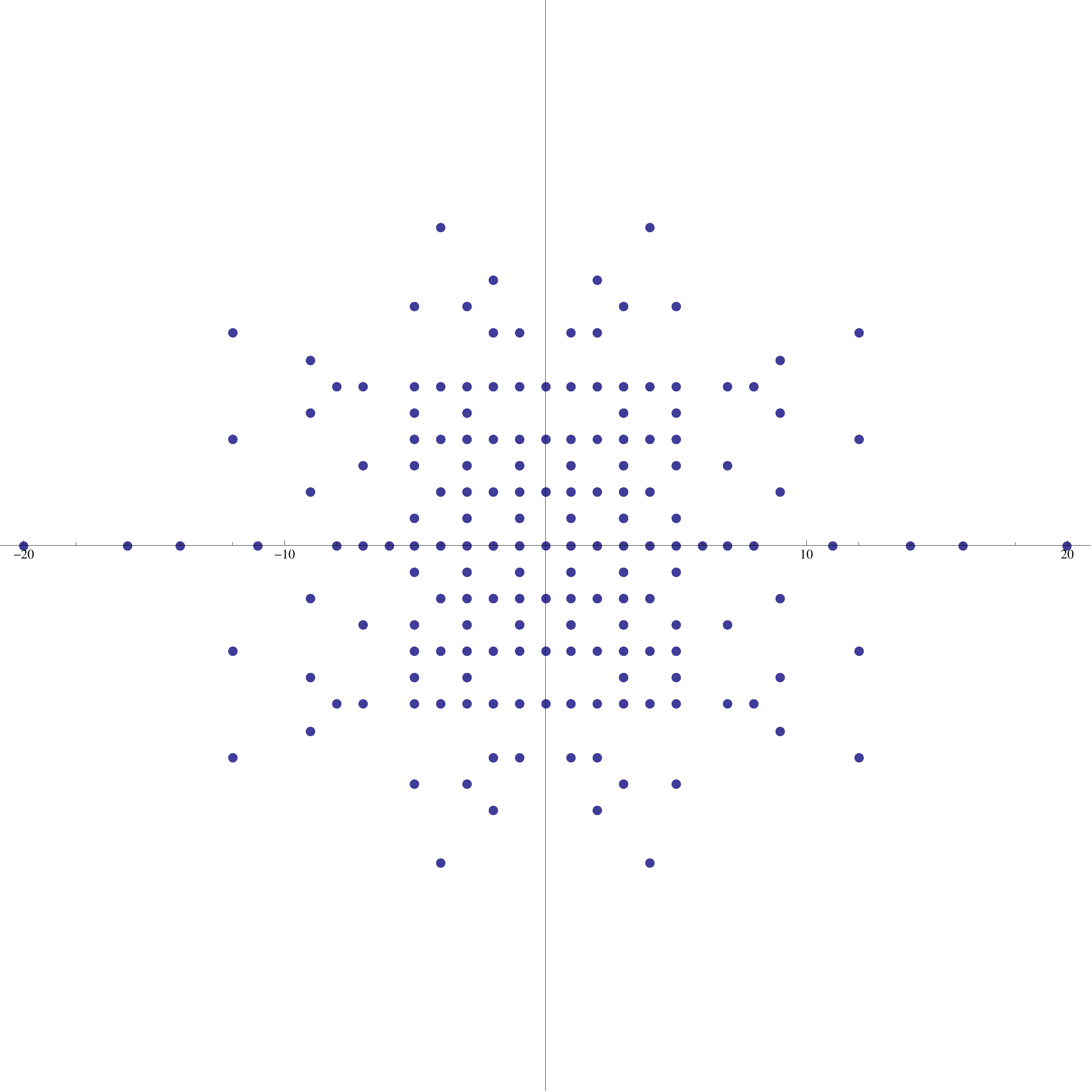}
			\caption{$X=S_5(0,0,0,1,5)$}
		\end{subfigure}
		\quad
		\begin{subfigure}{0.3\textwidth}
			\centering
			\includegraphics[width=\textwidth]{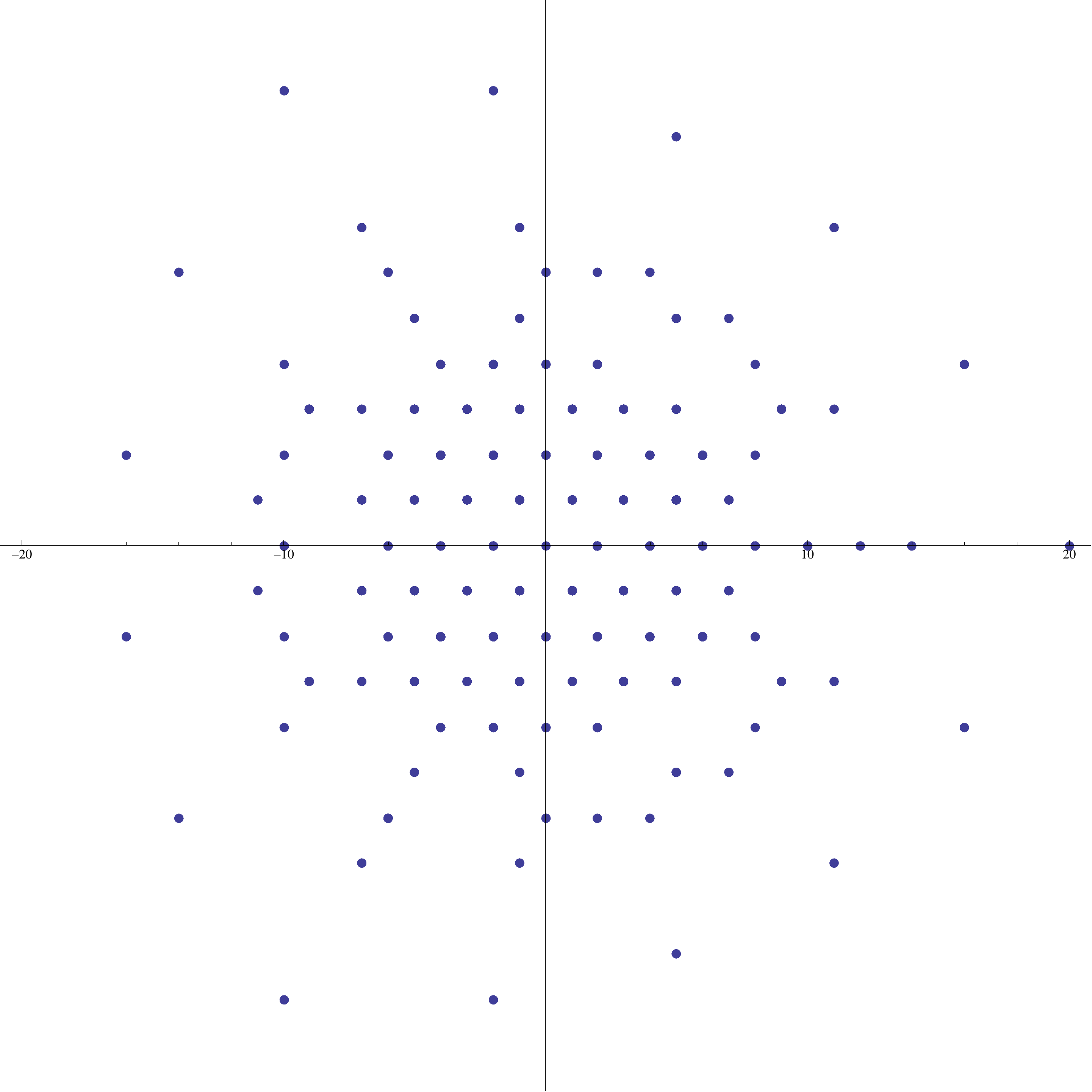}
			\caption{$X=S_5(0,0,0,1,7)$}
		\end{subfigure}
		\quad
		\begin{subfigure}{0.3\textwidth}
			\centering
			\includegraphics[width=\textwidth]{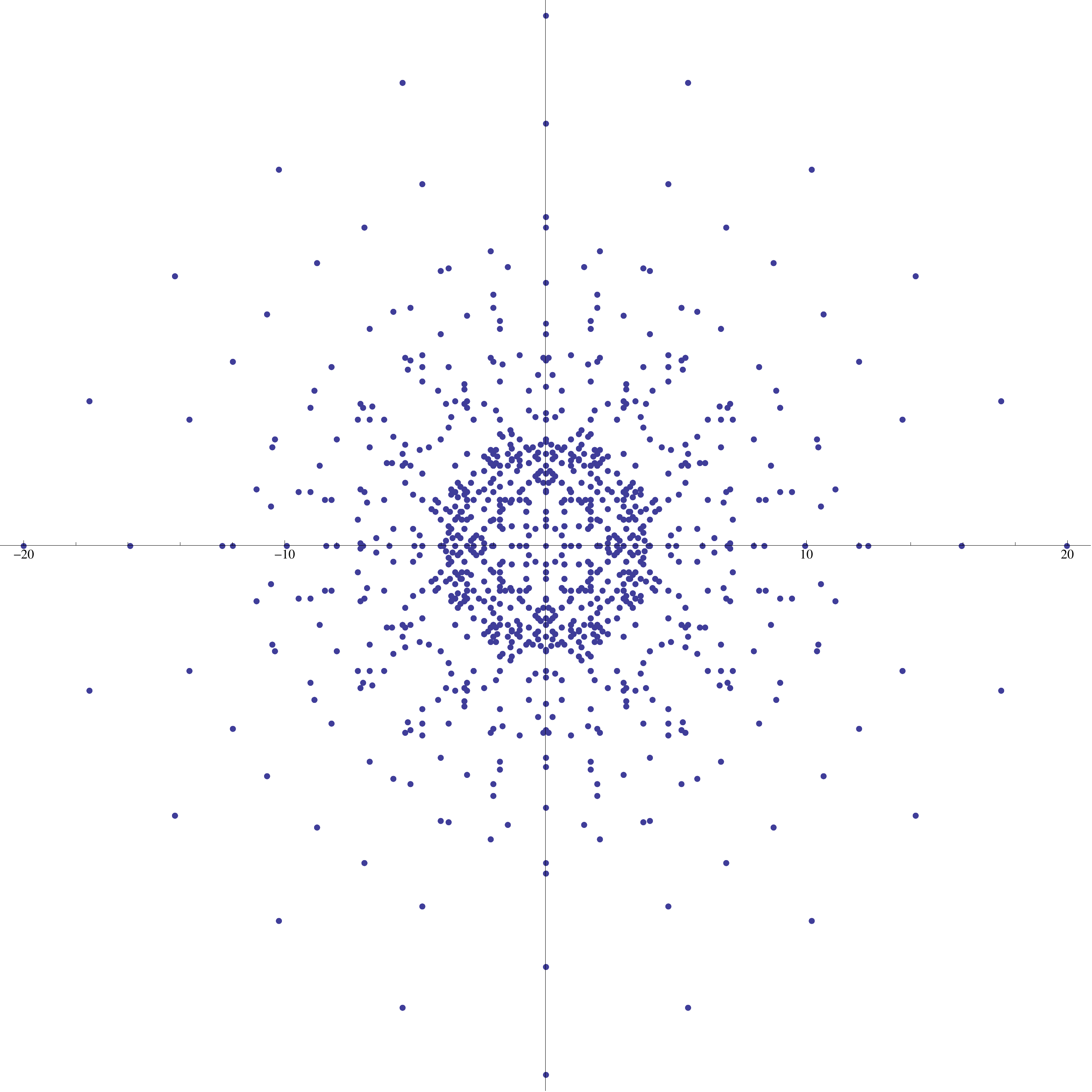}
			\caption{$X=S_5(0,0,0,1,2)$}
		\end{subfigure}
		\bigskip

		\begin{subfigure}{0.3\textwidth}
			\centering
			\includegraphics[width=\textwidth]{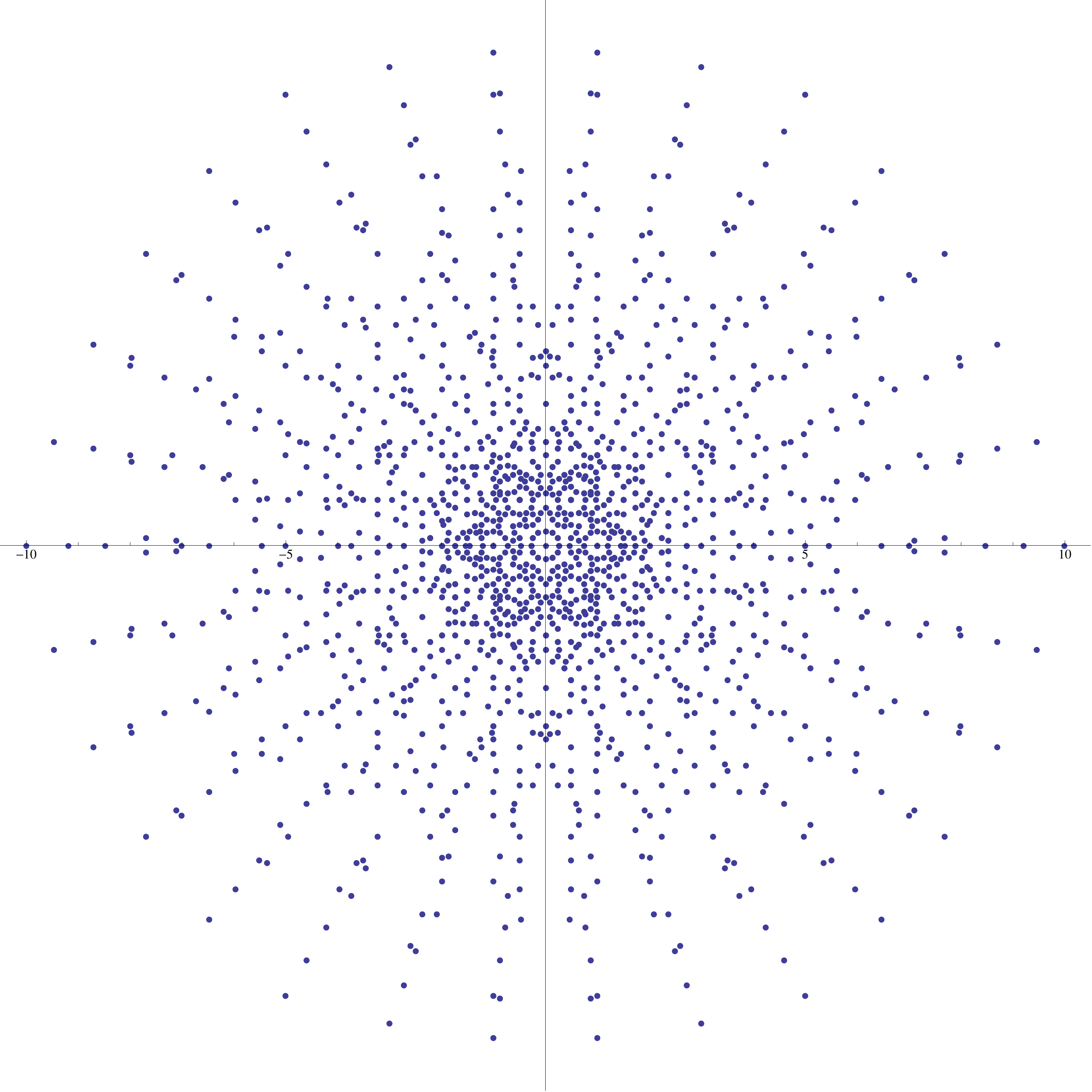}
			\caption{$X=S_5(0,0,0,1,1)$}
		\end{subfigure}
		\quad
		\begin{subfigure}{0.3\textwidth}
			\centering
			\includegraphics[width=\textwidth]{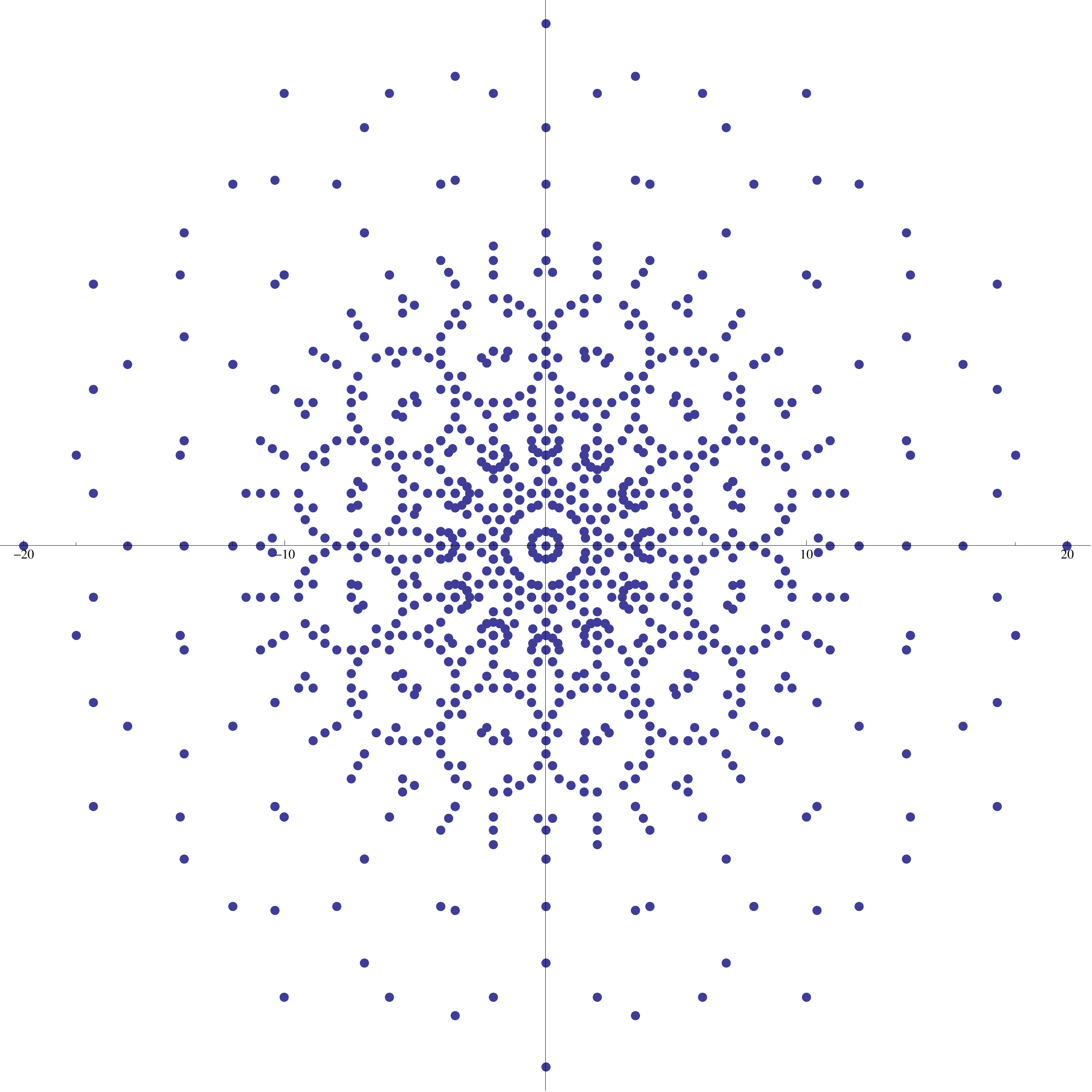}
			\caption{$X=S_5(0,0,0,1,6)$}
		\end{subfigure}
		\quad
		\begin{subfigure}{0.3\textwidth}
			\centering
			\includegraphics[width=\textwidth]{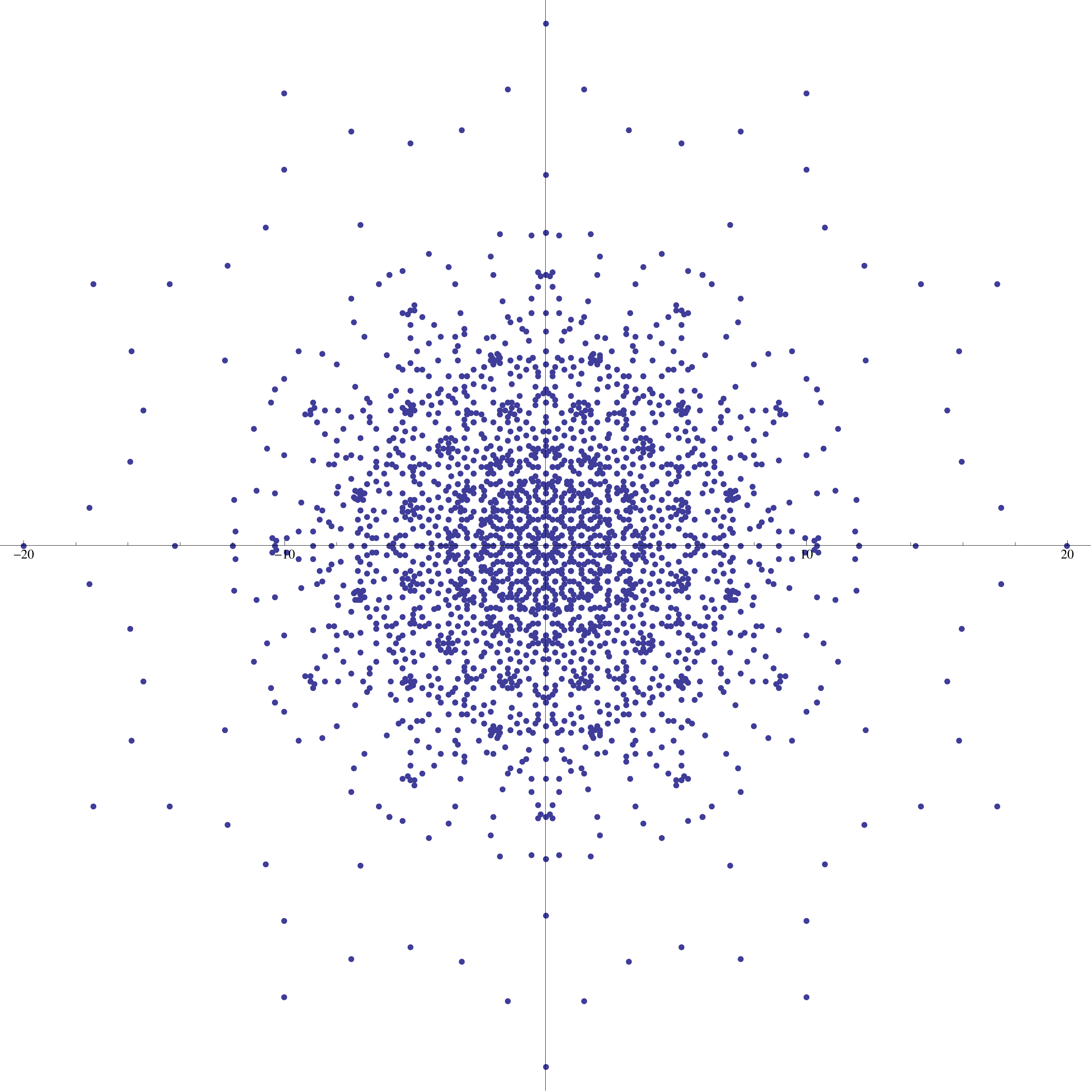}
			\caption{$X\!=\!S_5(0,0,0,1,10)$}
		\end{subfigure}
		
		\caption{Supercharacter plots $\sigma_X:(\Z/n\Z)^d\to\C$ corresponding to $n = 12$ and $d = 5$
		for various $X$.  In accordance with Proposition \ref{PropositionDihedral}, each image enjoys
		$\frac{12}{(12,[X])}$-fold dihedral symmetry.}
		\label{FigureSymmetry}
	\end{figure}
	
	Although individual supercharacter plots may have only $D_2$ symmetry
	(e.g., Figure \ref{FigureReal}), when one considers the images of all symmetric supercharacters on $(\Z/n\Z)^d$ simultaneously,
	one obtains a much higher degree of symmetry (see Figure \ref{FigureFull}).

	\begin{Proposition}\label{PropositionFull}
		For $n,d$ fixed, the union
		\begin{equation*}
			\bigcup_{X \in \X} \sigma_X\big( (\Z/n\Z)^d \big)
		\end{equation*}
		of all symmetric supercharacter plots on $(\Z/n\Z)^d$ has $\frac{n}{(n,d)}$-fold dihedral symmetry.
	\end{Proposition}
	
	\begin{proof}
		Fixing $n$ and $d$, we first note that
		\begin{align}
			\sigma_{(X+j\vec{1})}(Y+k\vec{1}) 
			&= \sum_{\vec{x} \in X} e\left(\frac{ (\vec{x}+j\vec{1}) \cdot (\vec{y} + k\vec{1})}{n}\right) \nonumber \\
			&= \sum_{\vec{x} \in X} e\left(\frac{ \vec{x \cdot y} + [\vec{y}] j + [\vec{x}] k + djk}{n}\right) \nonumber \\
			&= e\left(\frac{[Y] j + [X] k + djk}{n}\right) \sigma_X(Y). \label{eq-xyjk}
		\end{align}
		We therefore wish to identify the smallest positive $m$ for which the congruence
		\begin{equation}\label{eq-abdjk}
			aj+bk+djk \equiv m \pmod{n}
		\end{equation}
		can be solved for $j$ and $k$ given any prescribed values of $a=[X]$ and $b=[Y]$.  Setting $a = b = 0$ reveals that $m \geq (n,d)$.
		In fact, we wish to show that \eqref{eq-abdjk} can always be solved when $m=(n,d)$.

		By the Chinese Remainder Theorem, it suffices to show that
		\begin{equation}\label{eq-abdpl}
			aj+bk+djk \equiv (n,d) \pmod{p^\ell}
		\end{equation}
		can be solved for each prime power $p^{\ell}$ which appears in the canonical factorization of $n$.  
		There are three special cases which are easy to dispatch:		
		\begin{itemize}\addtolength{\itemsep}{0.25\baselineskip}
			\item If $(a,p^{\ell})=1$, then let $j=a^{-1}(n,d)$ and $k=0$.
			\item If $(b,p^{\ell})=1$, then let $j=0$ and $k=b^{-1}(n,d)$.
			\item If $(d,p^{\ell})=1$, then let $j=d^{-1}(1-b)$ and $k=(n,d)-aj$.
		\end{itemize}
		We may therefore assume that 
		$a=p^\alpha a'$, $b=p^\beta b'$, and $d = p^\delta d'$, where $(a',p)=(b',p)=(d',p)=1$ and $\alpha, \beta, \delta > 0$. 
		This leads to two special cases:
		\begin{itemize}\addtolength{\itemsep}{0.25\baselineskip}
			\item If $\delta \geq \ell$, then both $d$ and $(n,d)$ are divisible by $p^\ell$ so that \eqref{eq-abdpl} becomes
				\begin{equation*}
					aj+bk+0 \equiv 0 \pmod{p^\ell},
				\end{equation*}
				which has the solution $j=k=0$. 
			\item If $1 \leq \delta < \ell$, then let $\mu = \text{min}\{\alpha,\beta,\delta\}$ and observe that 
				$1 \leq \mu < \ell$.  It follows from \eqref{eq-abdpl} that
				\begin{equation}\label{eq-scrb}
					\frac{a}{p^{\mu}} j + \frac{b}{p^{\mu}} k + \frac{d}{p^{\mu}}jk
					\equiv \Big( \frac{n}{p^{\mu}}, \frac{d}{p^{\mu}} \Big) \pmod{ p^{\ell - \mu}}
				\end{equation}
				By our choice of $\mu$, at least one of $p^{\alpha-\mu}$, $p^{\beta-\mu}$, and $p^{\delta-\mu}$ must equal $1$.
				In other words, at least one of the coefficients in \eqref{eq-scrb} is relatively prime to the modulus $p^{\ell-\mu}$ of
				the congruence.  In light of the three special cases considered above, we conclude that
				\eqref{eq-scrb}, and hence \eqref{eq-abdpl}, is solvable.
		\end{itemize}
		Putting this all together, we see that when $m = (n,d)$, the congruence \eqref{eq-abdjk} can be solved for $j$ and $k$,
		given any fixed values of $a$ and $b$.  This concludes the proof.
	\end{proof}

	\begin{figure}[h]
		\begin{subfigure}{0.3\textwidth}
			\centering
			\includegraphics[width=\textwidth]{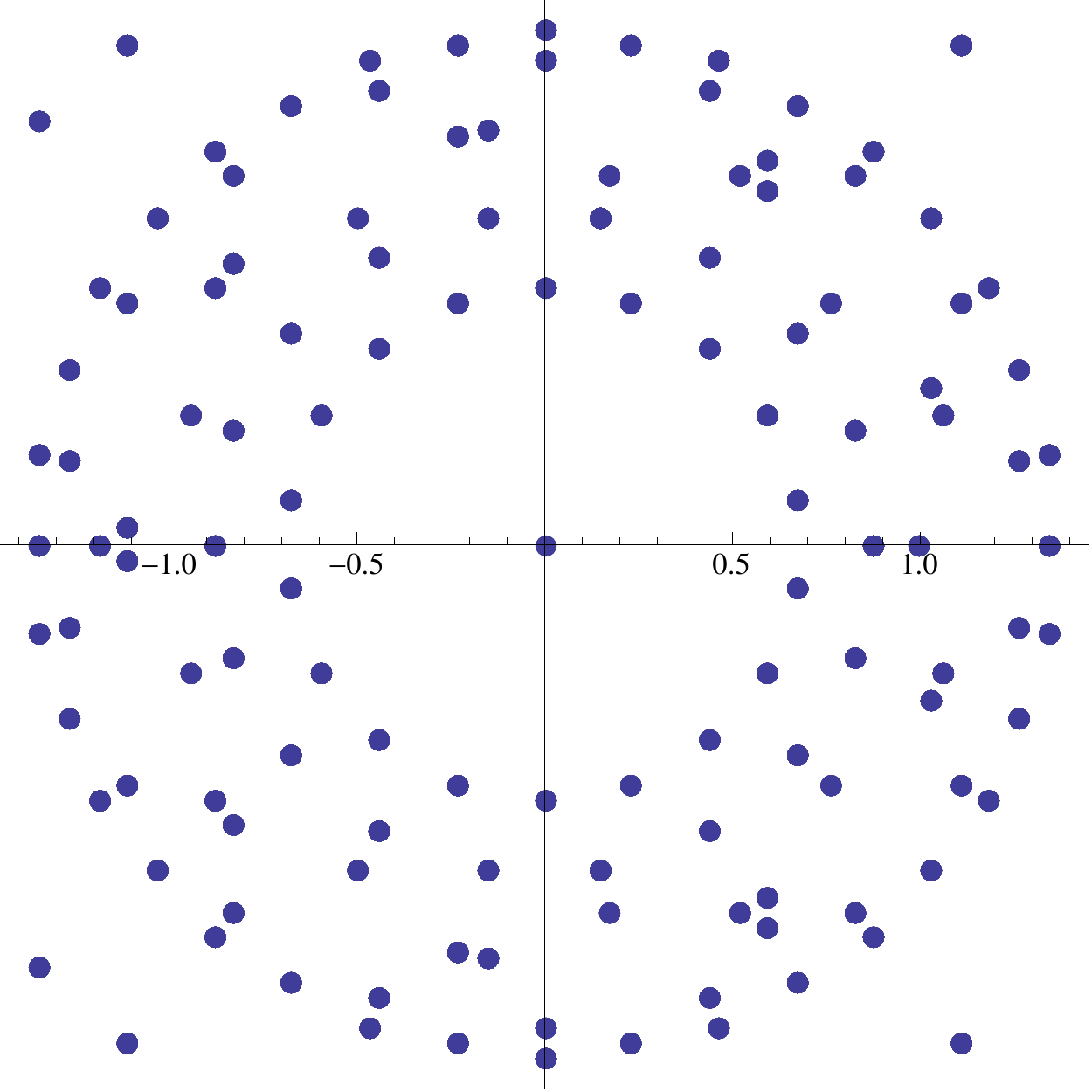}
			\caption{$n=9$, $d=3$, $\frac{9}{(9,3)} = 3$}
		\end{subfigure}
		\quad
		\begin{subfigure}{0.3\textwidth}
			\centering
			\includegraphics[width=\textwidth]{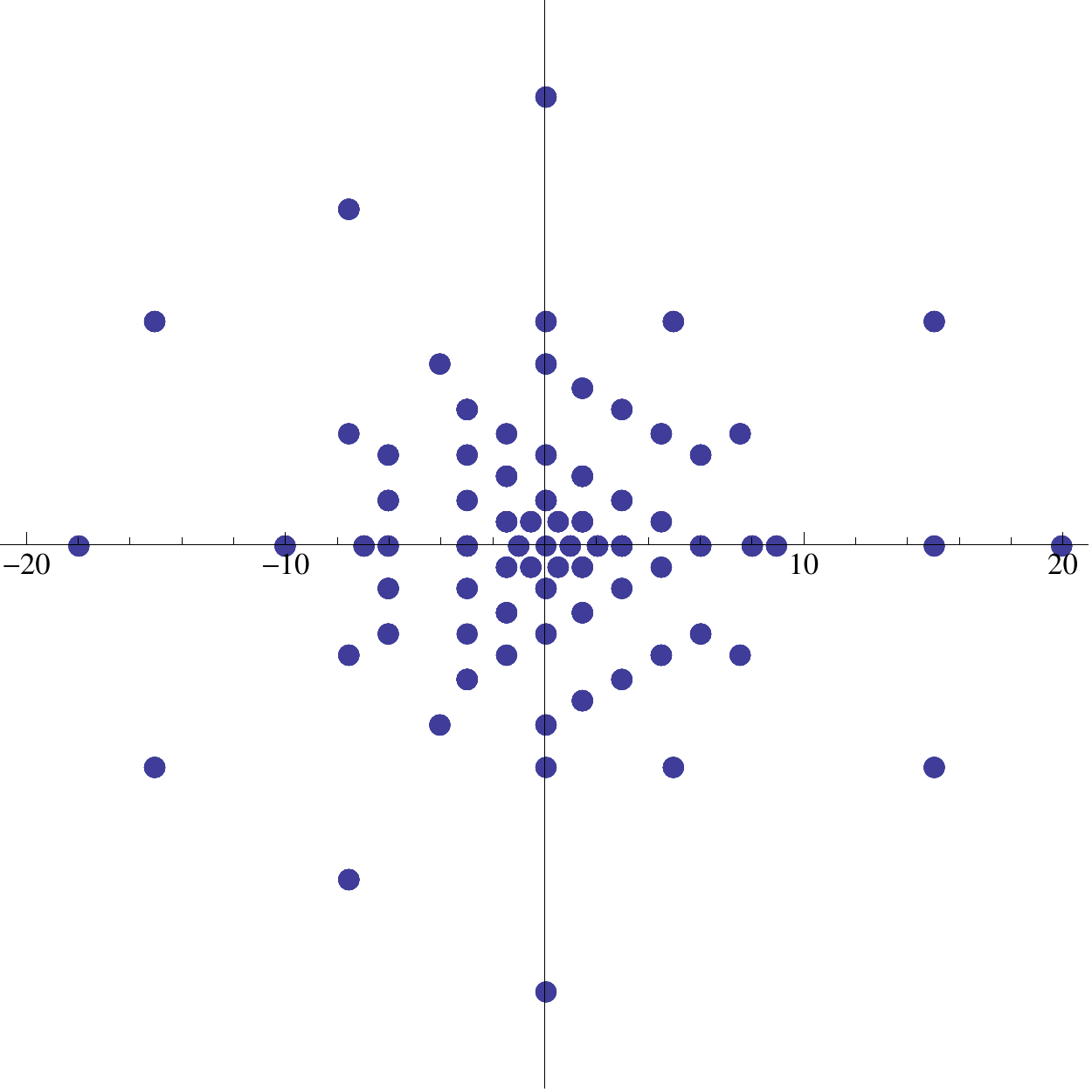}
			\caption{$n=3$, $d=6$, $\frac{3}{(3,6)} = 1$}
		\end{subfigure}
		\quad
		\begin{subfigure}{0.3\textwidth}
			\centering
			\includegraphics[width=\textwidth]{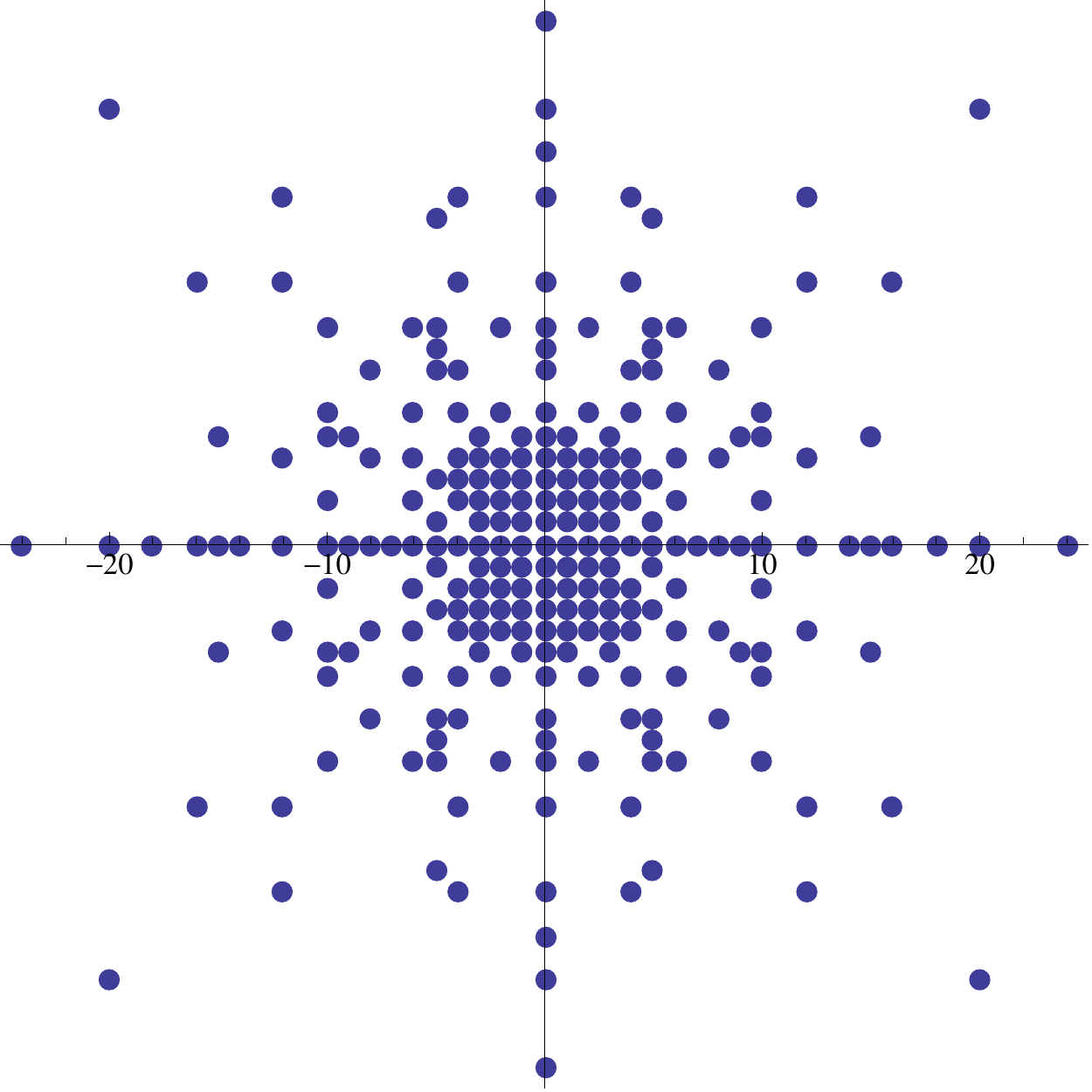}
			\caption{$n=4$, $d=6$, $\frac{4}{(4,6)}=2$}
		\end{subfigure}
		\caption{Examples of the sets $\bigcup_{X \in \X} \sigma_X\big( (\Z/n\Z)^d \big)$ for various $n$ and $d$.
		In accordance with Proposition \ref{PropositionFull}, these plots exhibit $\frac{n}{(n,d)}$-fold dihedral symmetry.}
		\label{FigureFull}
	\end{figure}

\subsection{Restricted walks}
	The simplest nontrivial symmetric supercharacters for which a nice qualitative 
	description exists are those of the form
	$\sigma_X:(\Z/n\Z)^d \to \C$ where $X= S_d(0,0,\ldots,0,1)$.
	Writing $\vec{y} = (y_1,y_2, \ldots, y_d)$, we find that
	\begin{equation}\label{eq-ZS}
		\sigma_X(\vec{y}) = \sum_{i=1}^d e\left(\frac{y_i}{n}\right),
	\end{equation}
	so that the values attained by $\sigma_X$ coincide with the endpoints of all $d$-step
	walks starting at the origin and having steps of unit length which are restricted to the $n$ directions
	$e(\frac{j}{n})$ for $j=1,2,\ldots,n$.
	
	The ``boundary'' of the image $\sigma_X( (\Z/n\Z)^d )$ is easy to describe.
	The furthest from the origin one can get in $d$ steps are the $n$ points $de(\frac{j}{n})$,
	which correspond to taking all $d$ steps in the same direction.  Also accessible to 
	our walker are the points $(d-\ell) e( \frac{j}{n} ) + \ell e ( \frac{j+1}{n})$
	for $0 \leq \ell \leq n$ which lie along the sides of the regular $n$-gon determined by the
	vertices $de(\frac{j}{n})$.  There is also an iterative interpretation of \eqref{eq-ZS}, 
	for we can also think of our $d$-step walk as being the composition of a $m$-step walk followed by a $(d-m)$-step walk
	(see Figure \ref{FigureRW}).
	\begin{figure}[htb!]
		\begin{subfigure}{0.4\textwidth}
			\centering
			\includegraphics[width=\textwidth]{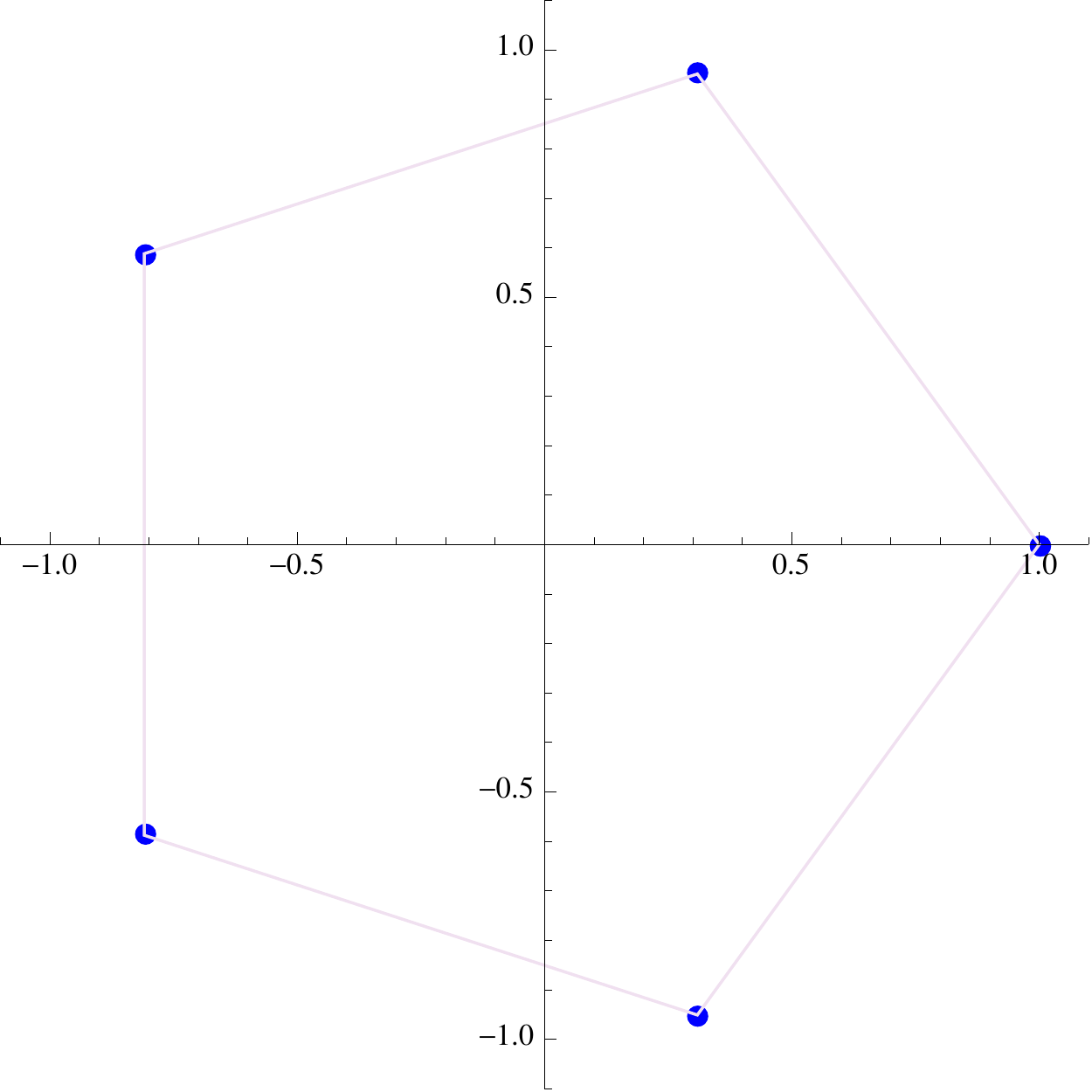}
			\caption{\bf $d=1$, $X = \{1\}$}
		\end{subfigure}
		\qquad
		\begin{subfigure}{0.4\textwidth}
			\centering
			\includegraphics[width=\textwidth]{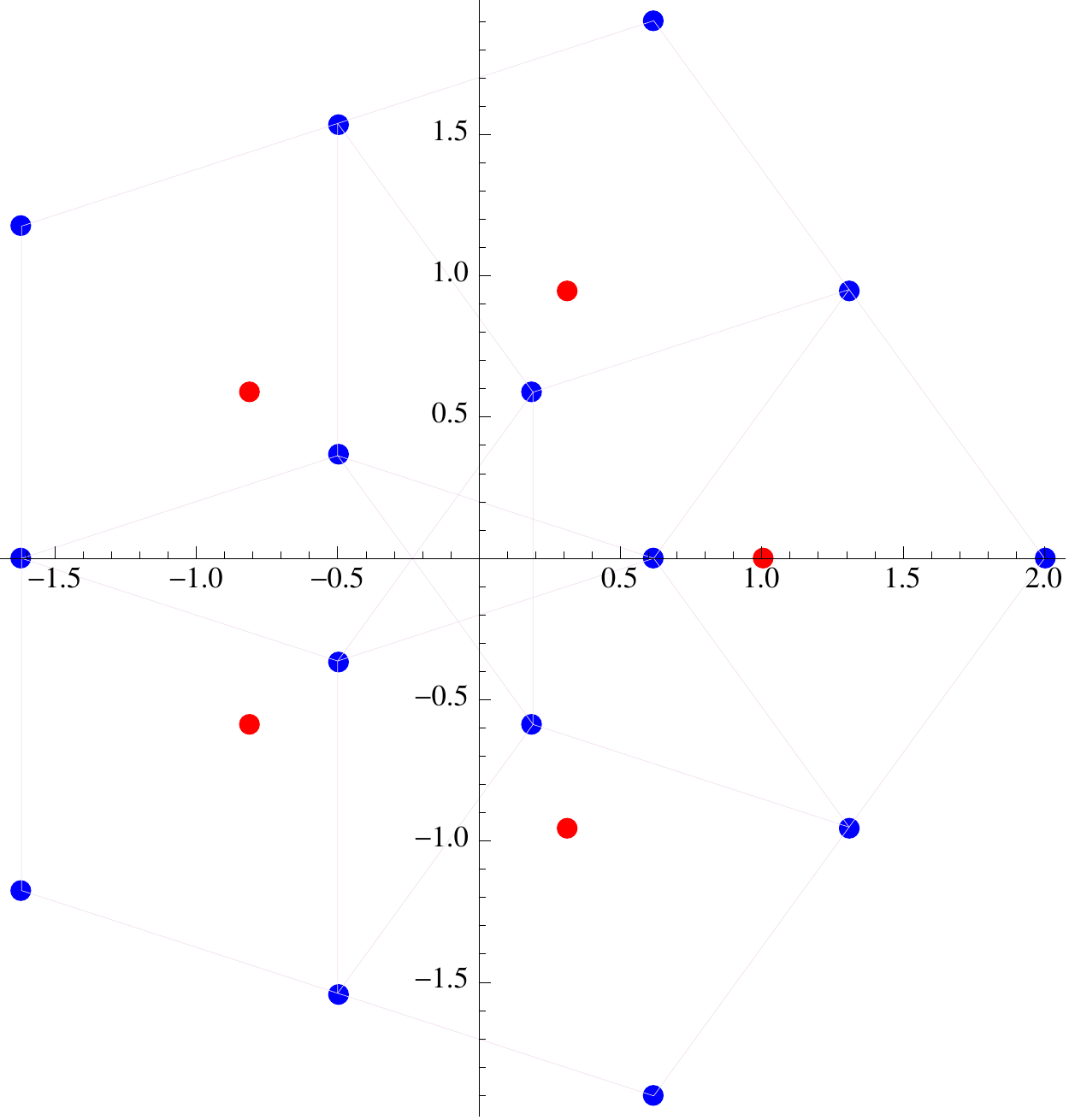}
			\caption{$d=2$, $X= S_2(0,1)$}
		\end{subfigure}
		\bigskip
		
		\begin{subfigure}{0.4\textwidth}
			\centering
			\includegraphics[width=\textwidth]{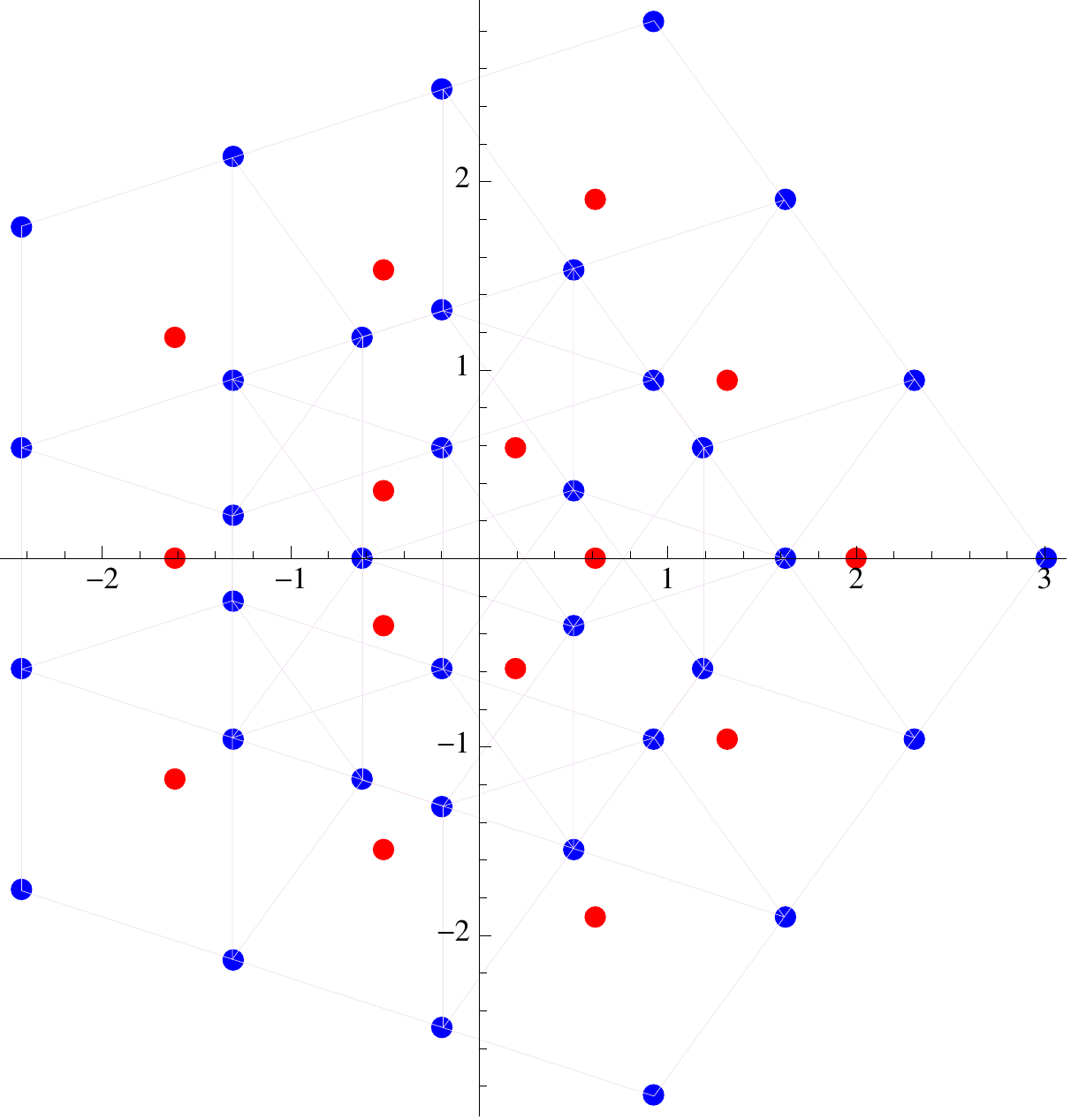}
			\caption{$d=3$, $X = S_3(0,0,1)$}
		\end{subfigure}
		\qquad
		\begin{subfigure}{0.4\textwidth}
			\centering
			\includegraphics[width=\textwidth]{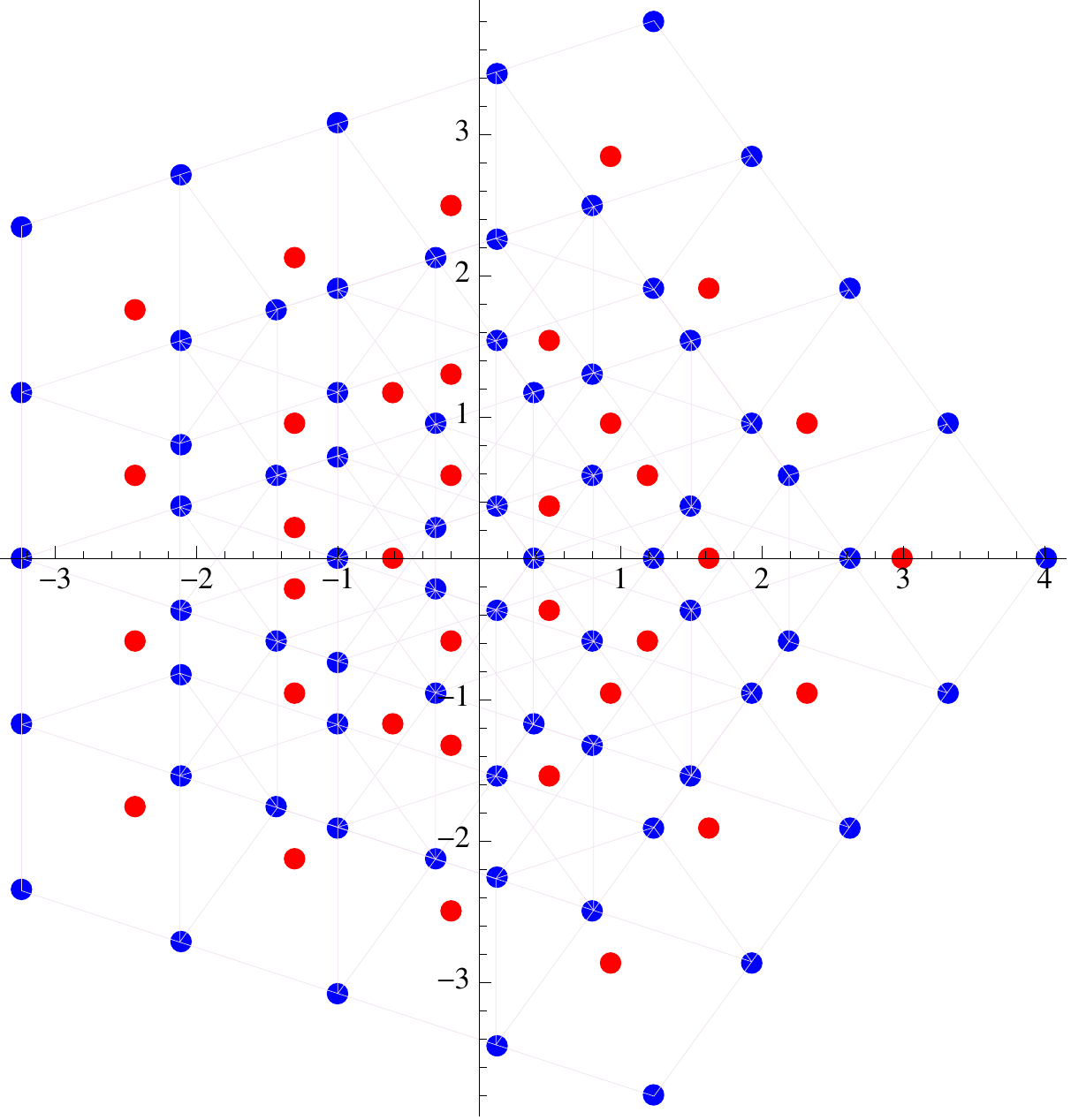}
			\caption{$d=4$, $X = S_4(0,0,0,1)$}
		\end{subfigure}
		\caption{The images $\sigma_X:( \Z/5\Z)^d\to\C$ for $d = 1,2,3,4$ (best viewed in color) illustrate two principles.
		Each successive image is obtained by placing copies of the image corresponding to $d=1$ (i.e., the vertices of the
		regular pentagon) centered at each point of the preceding image.
		Alternatively, the $d$th image can be constructed from $5$ copies of $(d-1)$st image, each centered
		at the $5$th roots of unity. }
		\label{FigureRW}
	\end{figure}

	A similar treatment exists for the supercharacters $\sigma_{S_d(0,0,\ldots,0,a)}:(\Z/n\Z)^d\to\C$.  
	If $(a,n) = 1$, then a glance at \eqref{eq-ZS} indicates that the image is identical to the image
	corresponding to $X = S_d(0,0,\ldots,0,1)$.  On the other hand, if $n = (n,a)r$ where $(n,r)=1$, then the image obtained is identical
	to the image of the supercharacter $\sigma_{S_d(0,0,\ldots,0,1)} : (\Z/r\Z)^d\to\C$.  This corresponds to a $d$-step
	walk with directions restricted to $r = n/(n,a)$ directions (see Figure \ref{FigureRRW}).
	\begin{figure}[htb!]
		\begin{subfigure}{0.3\textwidth}
			\centering
			\includegraphics[width=\textwidth]{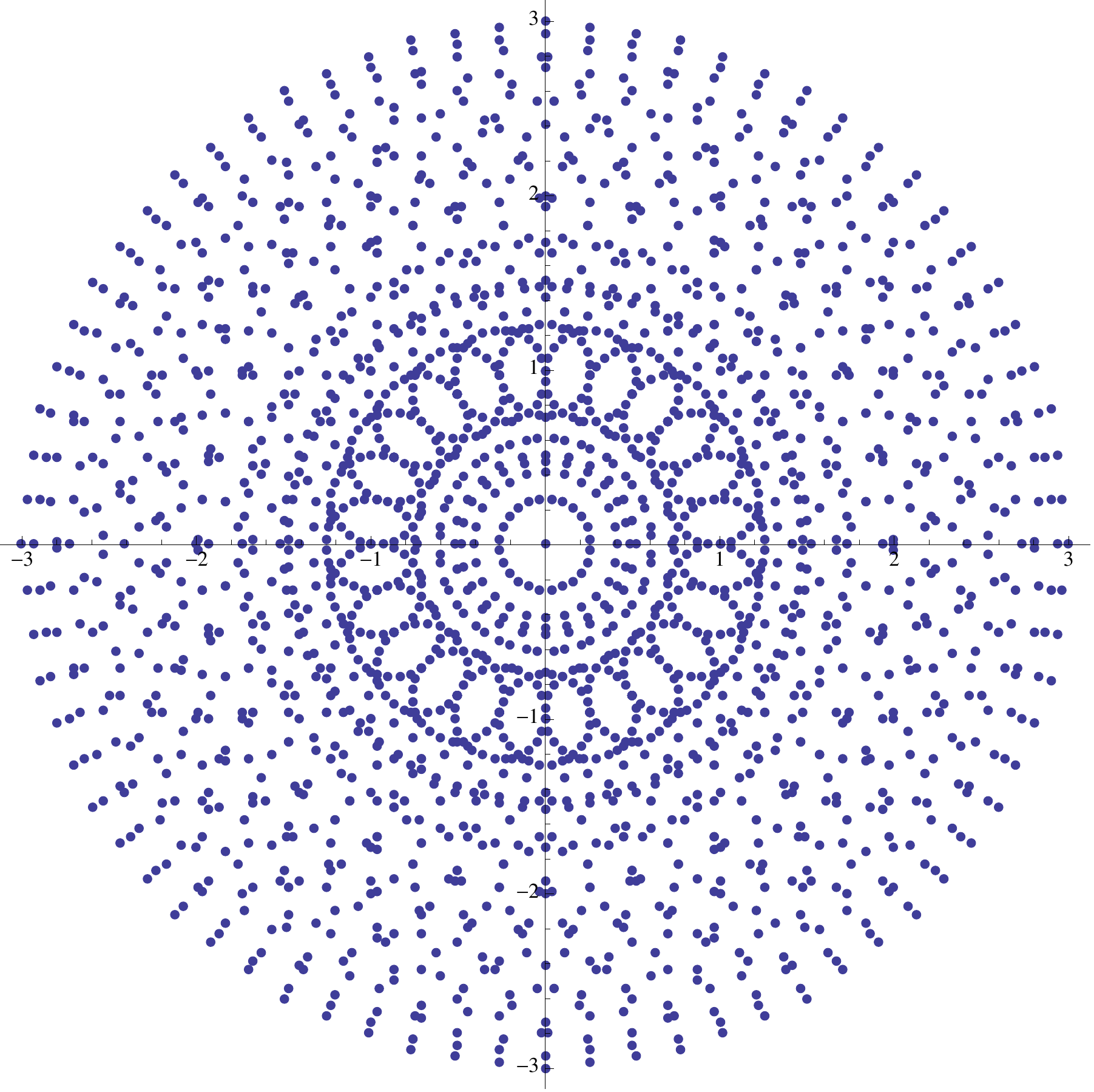}
			\caption{$X=S_3(0,0,1)$}
		\end{subfigure}
		\quad
		\begin{subfigure}{0.3\textwidth}
			\centering
			\includegraphics[width=\textwidth]{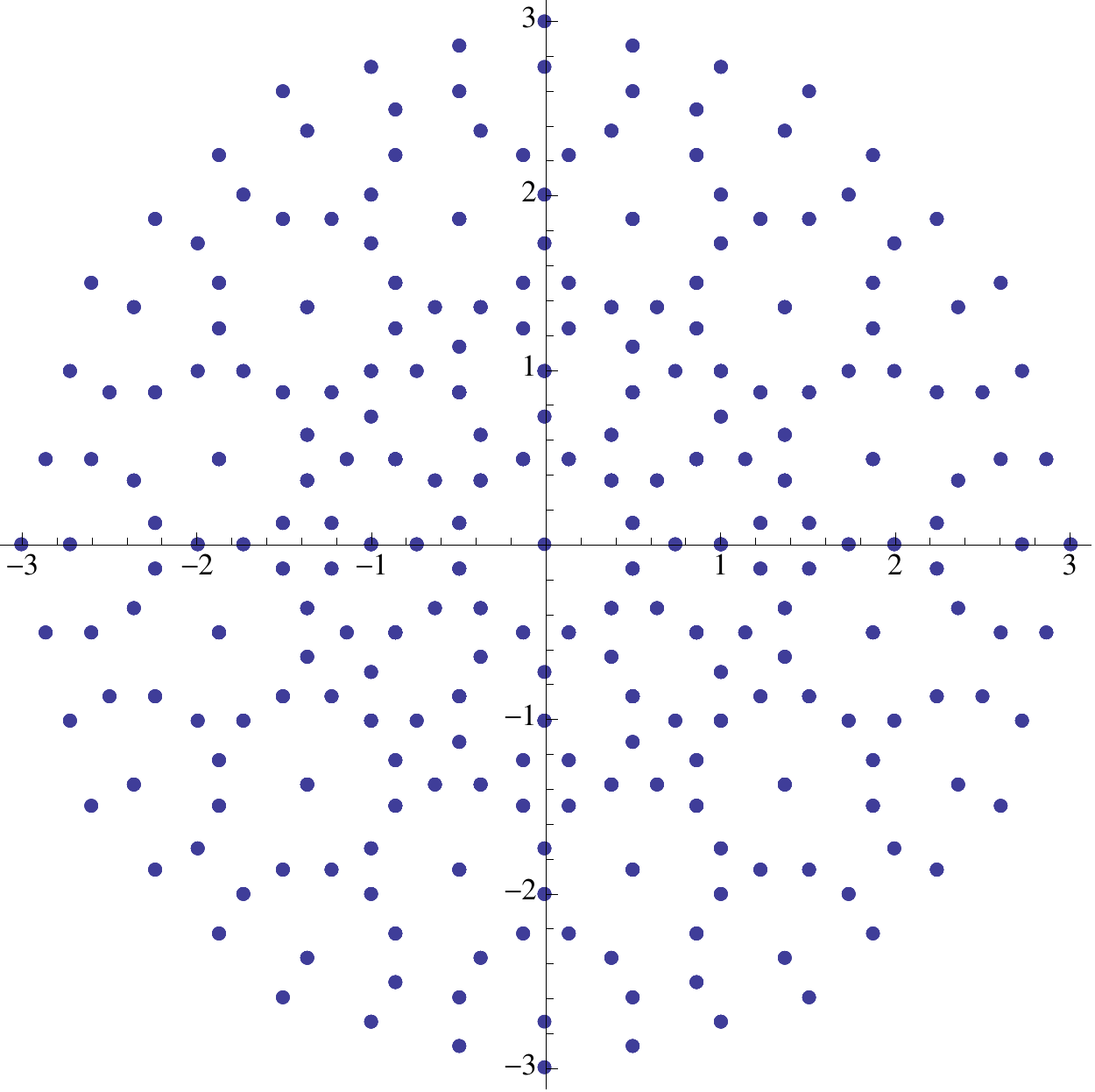}
			\caption{$X=S_3(0,0,2)$}
		\end{subfigure}
		\quad
		\begin{subfigure}{0.3\textwidth}
			\centering
			\includegraphics[width=\textwidth]{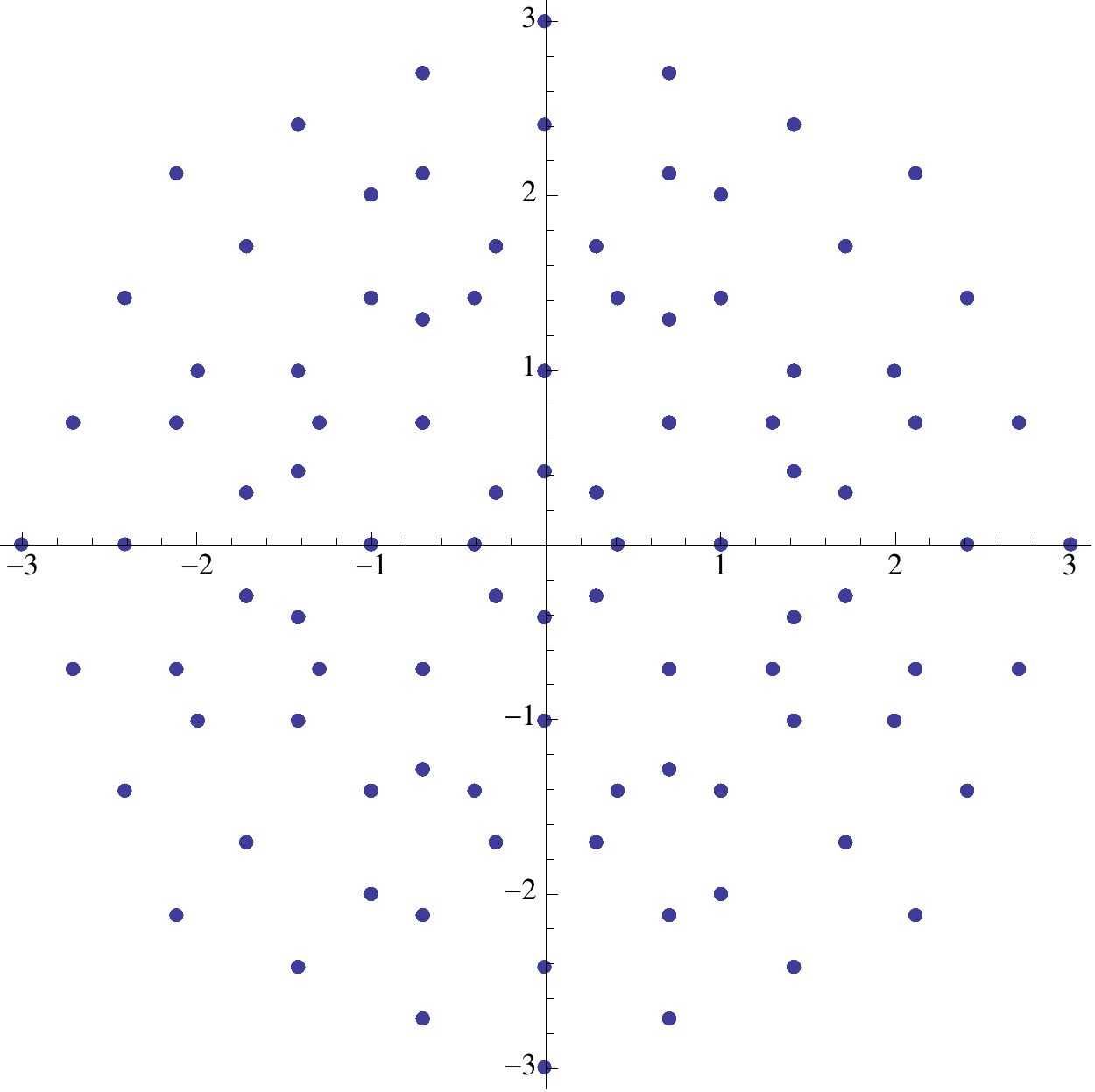}
			\caption{$X=S_3(0,0,3)$}
		\end{subfigure}
		\bigskip

		\begin{subfigure}{0.3\textwidth}
			\centering
			\includegraphics[width=\textwidth]{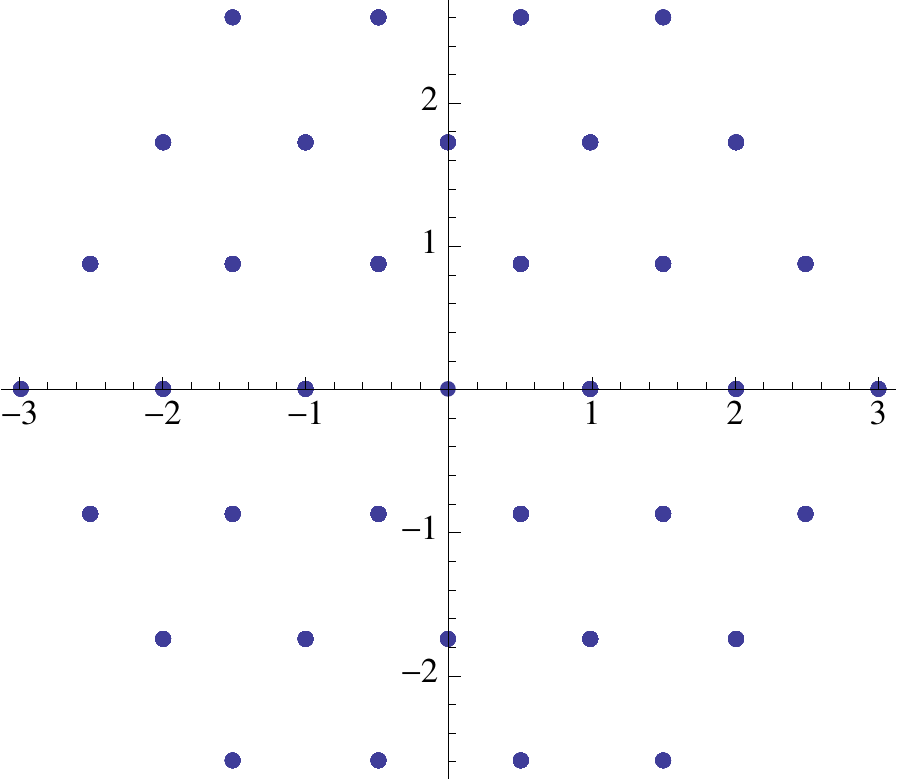}
			\caption{$X=S_3(0,0,4)$}
		\end{subfigure}
		\quad
		\begin{subfigure}{0.3\textwidth}
			\centering
			\includegraphics[width=\textwidth]{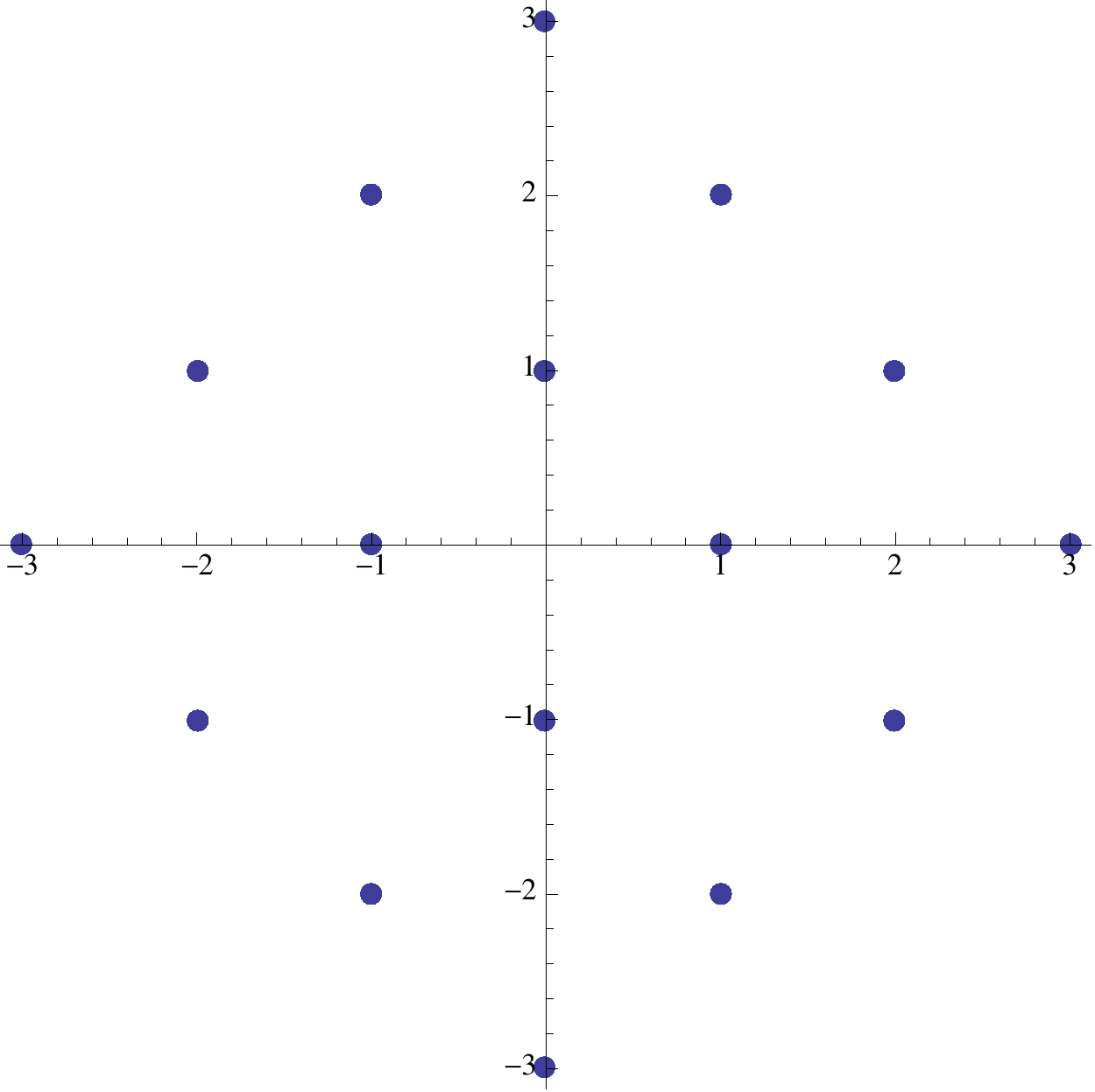}
			\caption{$X = S_3(0,0,6)$}
		\end{subfigure}
		\quad
		\begin{subfigure}{0.3\textwidth}
			\centering
			\includegraphics[width=\textwidth]{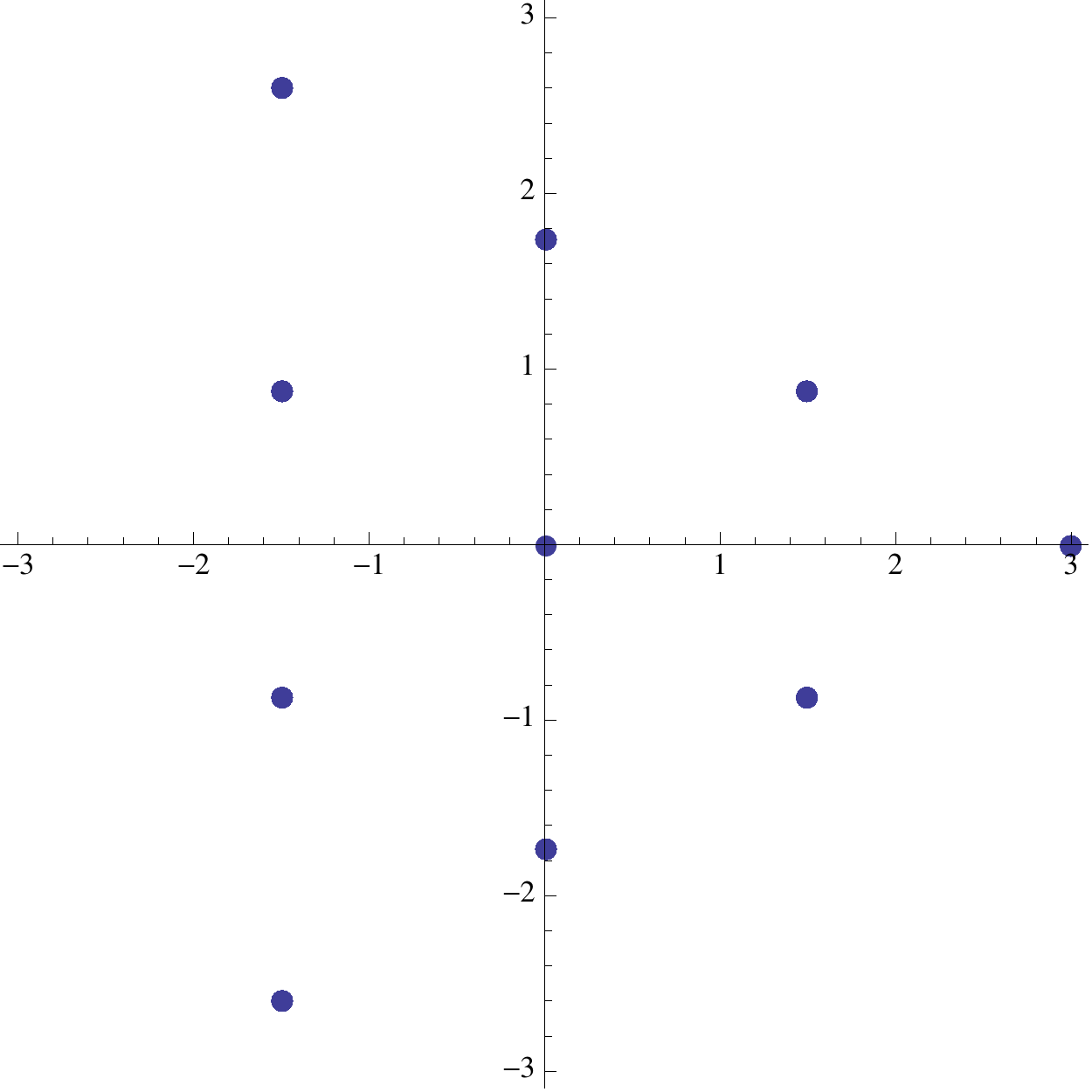}
			\caption{$X=S_3(0,0,8)$}
		\end{subfigure}
		\caption{Images of the supercharacter $\sigma_X:(\Z/24\Z)^3\to\C$ corresponding to 
		various superclasses.  The values attained coincide with the endpoints of all $3$-step
		walks starting at the origin and having steps of unit length which are restricted to the 
		$\frac{24}{(24,a)}$ directions $e(\frac{j(24,a)}{24})$ for $j=1,2,\ldots,\frac{24}{(24,a)}$.}
		\label{FigureRRW}
	\end{figure}

\subsection{Spikes}
	Under certain circumstances, the image of a symmetric supercharacter
	is contained in the union of a relatively small number of evenly spaced rays
	emanating from the origin.  The following proposition provides a simple
	condition which ensures such an outcome (see Figure \ref{FigureSpikes}).
	
	\begin{Proposition}\label{PropositionSpikes}
		If $X = r\vec{1} - X$ for some $r$ in $\Z/n\Z$, then the image of $\sigma_X$ is 
		contained in the union of the $\frac{2n}{(r,n)}$ rays
		given by $\arg z = \frac{\pi j (r,n)}{2n}$ for $j = 1,2,\ldots, \frac{2n}{(r,n)}$.		
	\end{Proposition}

	\begin{proof}
		Since $X = r \vec{1} - X$, for each $\vec{x}$ in $X$ there exists a permutation $\rho$ 
		such that $\vec{x} = r\vec{1} - \rho(\vec{x})$.  Using \eqref{eq-Stab} we find that
		\begin{align*}
			\sigma_X(\vec{y}) 
			&= \frac{1}{|\stab(\vec{x})|}\sum_{\pi \in S_d} e\left( \frac{\pi(\vec{x}) \cdot \vec{y}}{n} \right) \\ 
			&= \frac{1}{|\stab(\vec{x})|} \sum_{\pi \in S_d} e\left( \frac{\pi \big(r\vec{1}-  \rho(\vec{x})\big) \cdot \vec{y}}{n} \right) \\
			&= \frac{1}{|\stab(\vec{x})|} \sum_{\pi' \in S_d} e\left( \frac{r\vec{1} \cdot \vec{y} - \pi'(\vec{x}) \cdot \vec{y}}{n} \right) \\
			&= e\left( \frac{r[\vec{y}]}{n} \right) \frac{1}{|\stab(\vec{x})|} \sum_{\pi' \in S_d}   e\left( \frac{- \pi'(\vec{x}) \cdot \vec{y}}{n} \right) \\
			&= e\left( \frac{r[\vec{y}]}{n} \right) \overline{\sigma_X(\vec{y})} \\
			&= e\left( \frac{j(r,n)}{n} \right) \overline{\sigma_X(\vec{y})} ,
		\end{align*}
		where $j$ is an integer which depends on $\vec{y}$.  It follows from the preceding that
		\begin{equation*}
			\frac{\sigma_X(\vec{y})}{ | \sigma_X(\vec{y}) | }
			= \pm\left( \frac{j(r,n)}{2n} \right),
		\end{equation*}
		from which the desired result follows.
	\end{proof}

	\begin{figure}[htb!]
		\begin{subfigure}{0.4\textwidth}
			\centering
			\includegraphics[width=\textwidth]{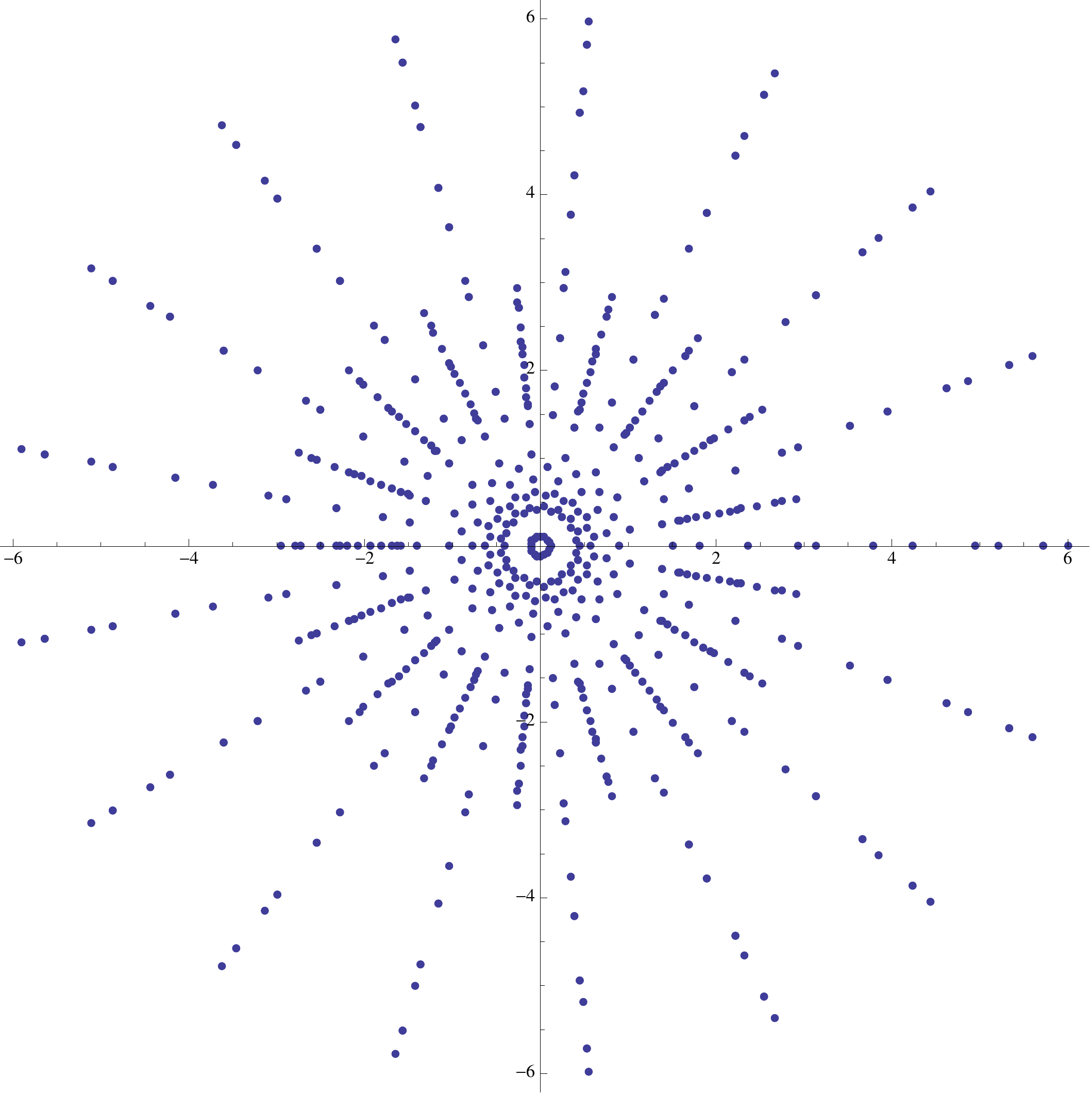}
			\caption{$n=17$, $d=3$, $X=S_3(1,2,3)$, $r=4$, $\frac{2n}{(r,n)}=34$}
		\end{subfigure}
		\qquad		
		\begin{subfigure}{0.4\textwidth}
			\centering
			\includegraphics[width=\textwidth]{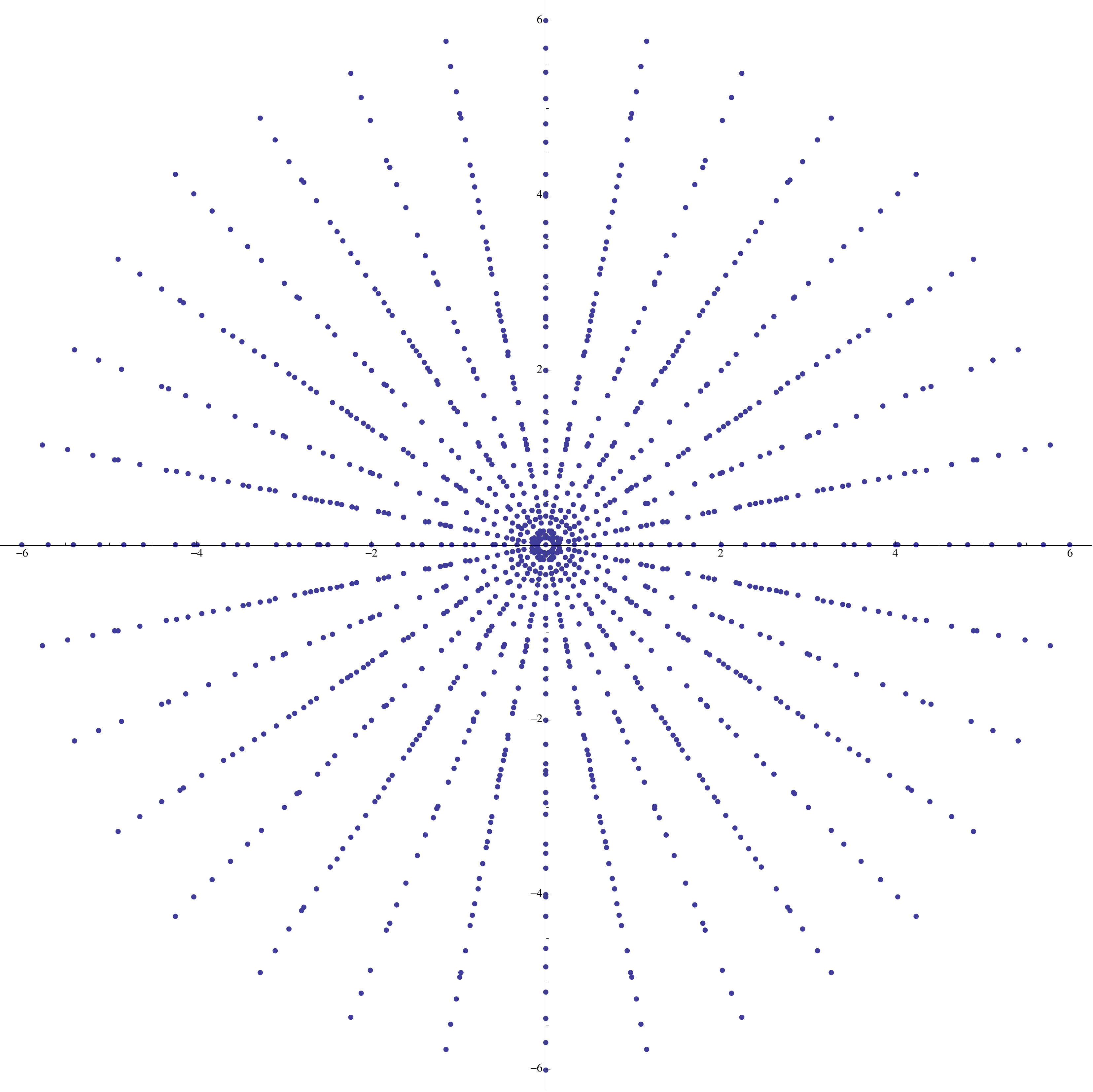}
			\caption{$n=16$, $d=4$, $X=S_4(1,1,10,10)$, $r=11$, $\frac{2n}{(r,n)}=32$}
		\end{subfigure}
		\caption{The image of a supercharacters $\sigma_X:(\Z/n\Z)^d \to \C$ for which $X = r \vec{1}-X$
		is contained in the union of $\frac{2n}{(r,n)}$ rays.}
		\label{FigureSpikes}
	\end{figure}
	
	\begin{Example}
		With more information, we can sometimes improve substantially upon 
		Proposition \ref{PropositionSpikes}.  
		Consider the supercharacter $\sigma_X:(\Z/n\Z)^d \to \C$ corresponding to 
		\begin{equation*}
			X = S_d (0,\underbrace{1,1,\ldots,1}_{d-2\text{ times}},2).
		\end{equation*}
		Since $X = 2\vec{1}-X$, it follows from Proposition \ref{PropositionSpikes} that the image of 
		$\sigma_X$ is contained in the union of $\frac{2n}{(2,n)}$ equally spaced rays emanating from the origin.
		However, we can be much more specific.	
		Letting $\vec{y}=(y_1,y_2,\ldots,y_d)$, we see that
		\begin{align}
			\sigma_X(\vec{y})
			&=\sum_{\vec{x} \in X}e\left(\frac{\vec{x}\cdot\vec{y}}{n}\right) \nonumber\\
			&=\sum_{1\leq j \neq k\leq d}e\left(\frac{[\vec{y}]+y_j-y_k}{n}\right) \nonumber\\
			&=e\left(\frac{[\vec{y}]}{n}\right)\sum_{1\leq j \neq k\leq d}e\left(\frac{y_j-y_k}{n}\right) \nonumber\\
			&=e\left(\frac{[\vec{y}]}{n}\right)\left[\left(\sum_{j=1}^d e\left(\frac{y_j}{n}\right)\right)
				\overline{\left(\sum_{k=1}^d e\left(\frac{y_k}{n}\right)\right)}-d\right] \nonumber\\
			&=e\left(\frac{[\vec{y}]}{n}\right)\left( \bigg| \sum_{j=1}^d e\left(\frac{y_j}{n}\right)\bigg|^2-d\right).
			\label{eq-DDD}
		\end{align}
		Noting that the second term in \eqref{eq-DDD} lies in the interval $[-d,d^2-d]$,
		we can see that the qualitative behavior of $\sigma_X$ depends heavily upon the parity of $n$
		(see Figure \ref{FigureHSB}).
		\begin{itemize}\addtolength{\itemsep}{0.5\baselineskip}
			\item If $n$ is odd, then the image of $\sigma_X$ is contained in the union of $2n$ rays,
				$n$ of which have length $d^2-d$ while the remaining $n$ rays have length approximately $d$
				(see Figure \ref{FigureHSB1}).
	
			\item If $n$ is even, then the image of $\sigma_X$ is contained in the union of $n$ rays, each of length $d^2 - d$
				(see Figure \ref{FigureHSB2}).
				Due to ``double coverage,'' the concentration of points is higher in the region $|z| \leq d$.	
		\end{itemize}
	\end{Example}

	\begin{figure}[h]
		\begin{subfigure}{0.4\textwidth}
			\centering
			\includegraphics[width=\textwidth]{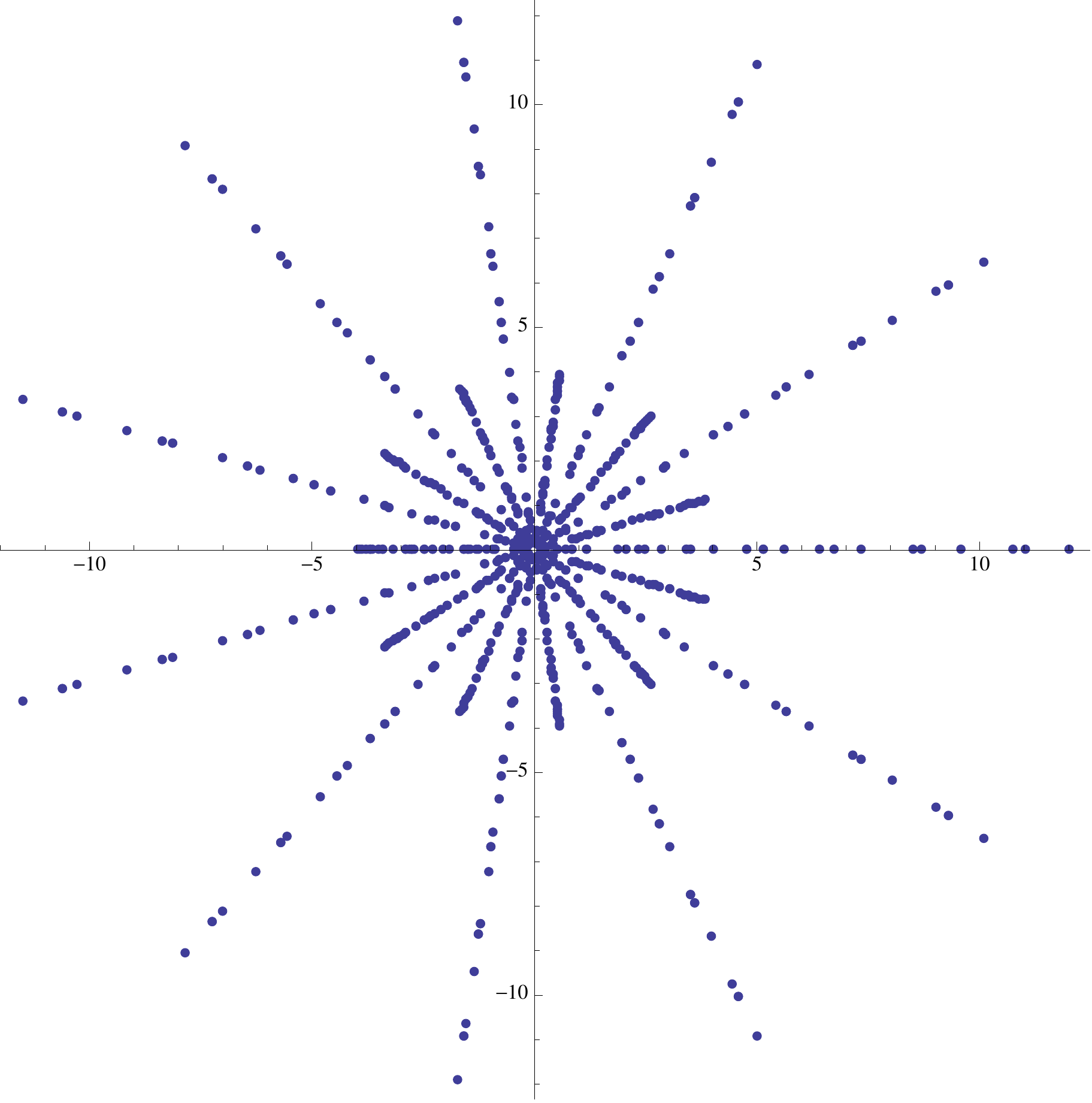}
			\caption{$n=11$, $d=4$, $\frac{2n}{(2,n)}=22$}
			\label{FigureHSB1}
		\end{subfigure}
		\qquad
		\begin{subfigure}{0.4\textwidth}
			\centering
			\includegraphics[width=\textwidth]{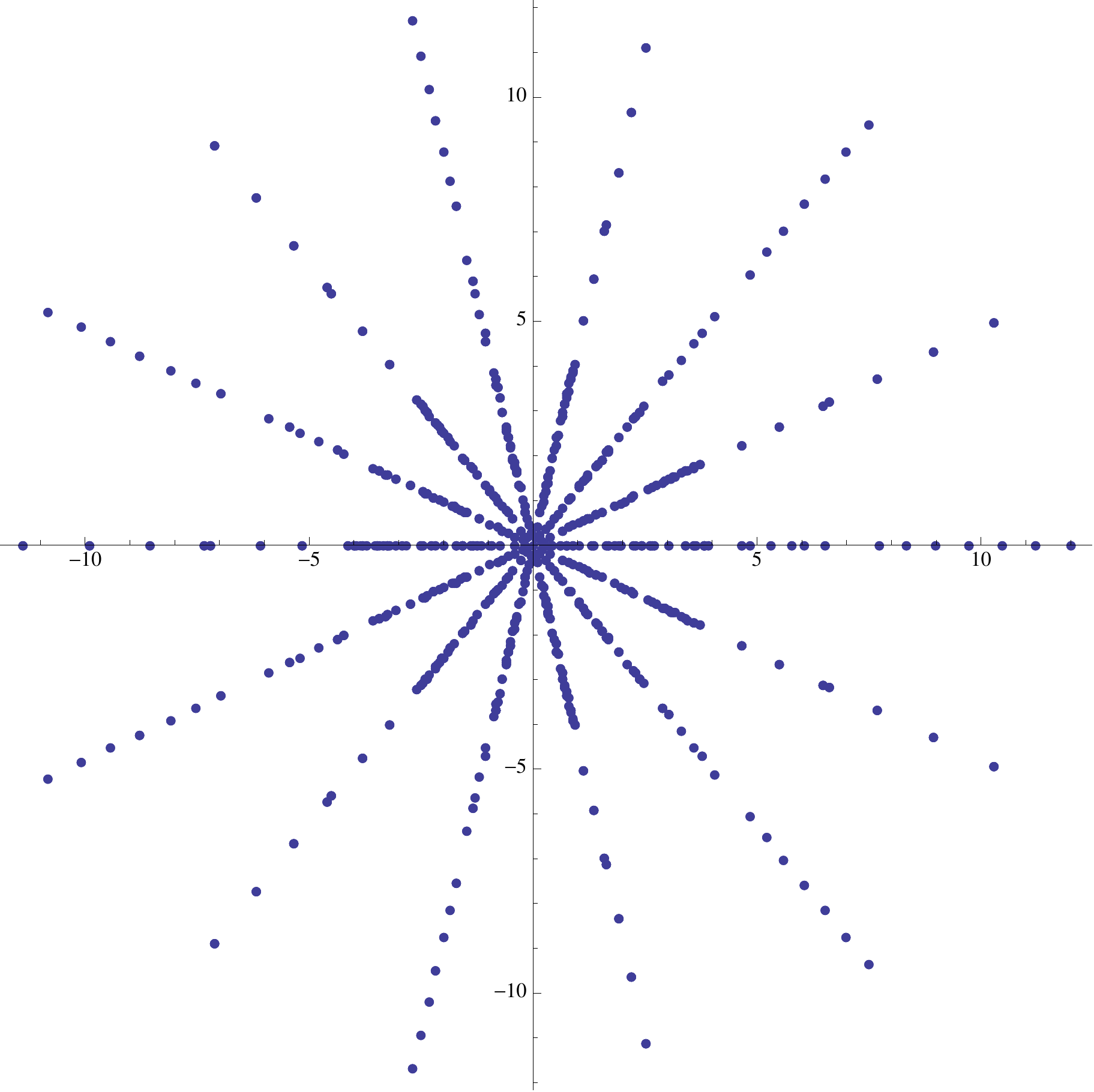}
			\caption{$n=14$, $d=4$, $\frac{2n}{(2,n)}=14$}
			\label{FigureHSB2}
		\end{subfigure}
		\caption{Images of the supercharacter $\sigma_X:(\Z/n\Z)^4\to\C$ corresponding to the superclass
		$X = S_4 (0,1,1,2)$ for several values of $n$.  In accordance with Proposition \ref{PropositionSpikes},
		the number of rays is given by $\frac{2n}{(2,n)}$ since $X = 2\vec{1} - X$.}
		\label{FigureHSB}
	\end{figure}
	
\subsection{Hypocycloids}\label{SubsectionHypocycloids}

	Hypocycloids are suggested in many symmetric supercharacter plots and 
	they form the basis of more complicated constructions.
	Recall that a \emph{hypocycloid} is a planar curve obtained by tracing the path of a fixed point on a small circle
	which rolls inside a larger circle.  The parametric equations
	\begin{equation}\label{eq-HypoPar}
		\begin{split}
			x(\theta) &= (d-1) \cos \theta + \cos[(d-1)\theta], \\
			y(\theta) &= (d-1)\sin \theta - \sin[(d-1)\theta],
		\end{split}
	\end{equation}
	describe the $d$-cusped hypocycloid obtained by rolling a circle of radius $1$ inside of a circle of integral
	radius $d$ centered at the origin.	

	\begin{Proposition}\label{PropositionHypo}
		If $X = S_d(1,1,\ldots,1,1-d)$, then the image of $\sigma_X:(\Z/n\Z)^d\to \C$ is bounded by the
		hypocycloid determined by \eqref{eq-HypoPar}.	
	\end{Proposition}
		
	\begin{proof}
		Given $\vec{y} = (y_1,y_2,\ldots,y_d)$, let $\xi_{\ell} = [\vec{y}] - dy_{\ell}$ for $1\leq \ell \leq d$, noting that
		\begin{equation*}
			\xi_1+\xi_2+\cdots + \xi_d = d[\vec{y}] - d(y_1+y_2+\cdots+y_d) = 0,
		\end{equation*}
		so that $\xi_d = -(\xi_1+\xi_2+\cdots+\xi_{d-1})$.
		Next, we observe that
		\begin{align*}
			\sigma_X(\vec{y})
			&= \sum_{\ell=1}^d e\left( \frac{(1-d)y_{\ell} + \sum_{\substack{i=1\\ i \neq \ell}}^d y_i  }{n}\right)  \\
			&= \sum_{\ell=1}^d e\left( \frac{ [\vec{y}] - d y_{\ell}}{n}\right) 
			= \sum_{\ell=1}^d e\left( \frac{\xi_{\ell}}{n}\right) \\
			&= \sum_{\ell=1}^{d-1} e\left( \frac{\xi_{\ell}}{n}\right) 
			+ e\left( \frac{ -( \xi_1+\xi_2+\cdots+\xi_{d-1})}{n}\right).
		\end{align*}
		Letting $\T$ denote the unit circle $|z|=1$ in $\C$, we see that
		the range of $\sigma_X$ is contained in the range of the function
		$g:\T^{d-1}\to\C$ defined by 
		\begin{equation}\label{eq-HypocycloidMap}
			g(z_1,z_2,\ldots,z_{d-1}) = z_1 + z_2 + \cdots + z_{d-1} + \frac{1}{z_1 z_2 \cdots z_{d-1}},
		\end{equation}
		which is known to be the filled hypocycloid determined by 
		\eqref{eq-HypoPar} \cite[Sect.~3]{Kaiser}.
	\end{proof}

	Noting that $\xi_1,\xi_2,\ldots,\xi_d$ are completely determined by the values of the arbitrary parameters
	$y_1,y_2,\ldots,y_{d-1}$ and $[\vec{y}]$ in $\Z/n\Z$, it follows that $(\xi_1,\xi_2,\ldots,\xi_{d-1})$
	runs over all values in $(\Z/n\Z)^{d-1}$ when $(n,d)=1$ and over the set of all $(d-1)$-tuples of the form
	$j\vec{1} + (n,d) ( v_1,v_2,\ldots,v_{d-1})$ otherwise.
	Thus the image of $\sigma_X$ closely approximates the filled hypocycloid when $n/(n,d)$ is large
	(see Figure \ref{FigureHypocycloid}).  

	\begin{figure}[h]
		\centering
		\begin{subfigure}[b]{0.3\textwidth}
	                \centering
	                \includegraphics[width=\textwidth]{HYPO-n19d6}
	                \caption{$n=19$}
	        \end{subfigure}
			\quad
	        \begin{subfigure}[b]{0.3\textwidth}
	                \centering
	                \includegraphics[width=\textwidth]{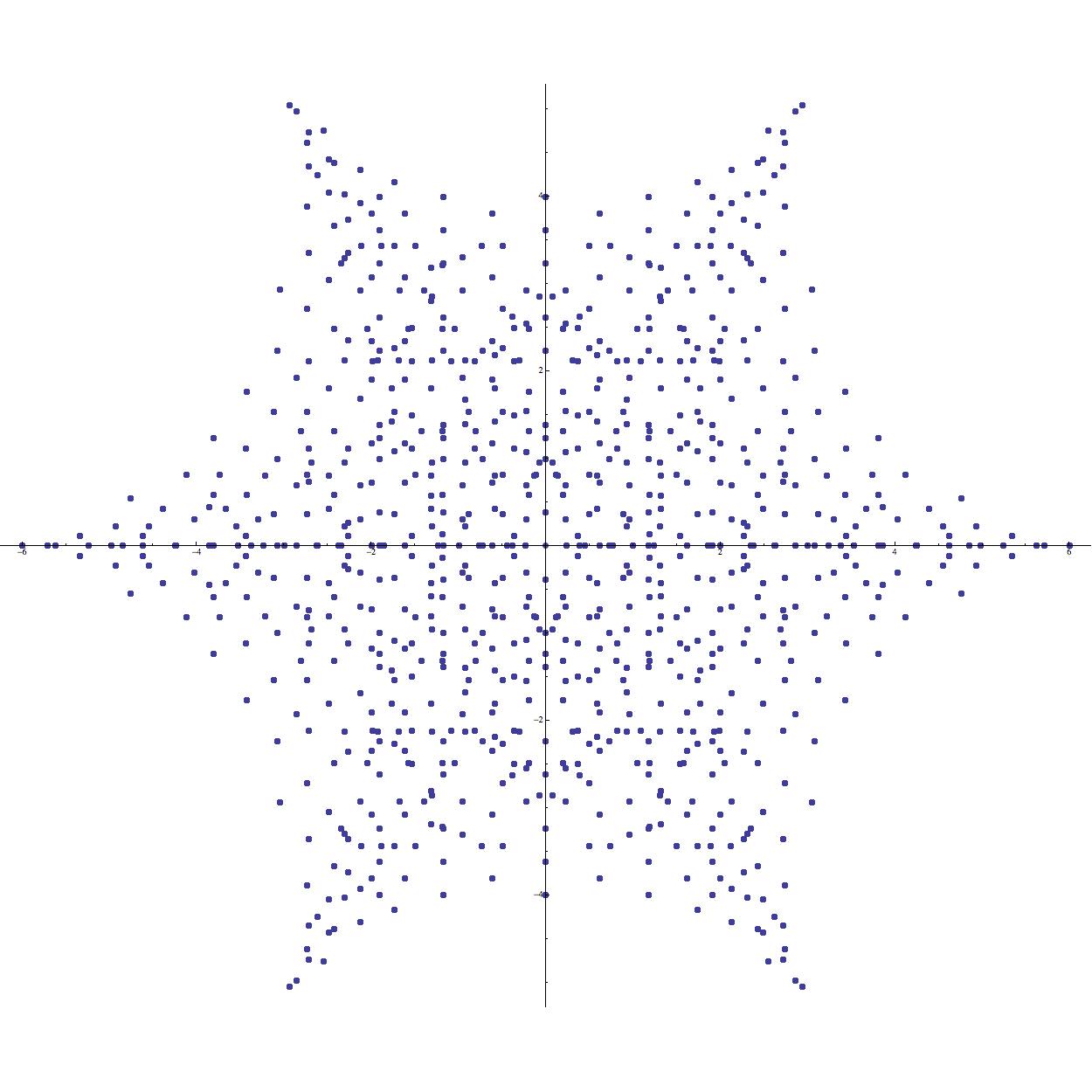}
	                \caption{$n=20$}
	        \end{subfigure}
	        \quad
	        \begin{subfigure}[b]{0.3\textwidth}
	                \centering
	                \includegraphics[width=\textwidth]{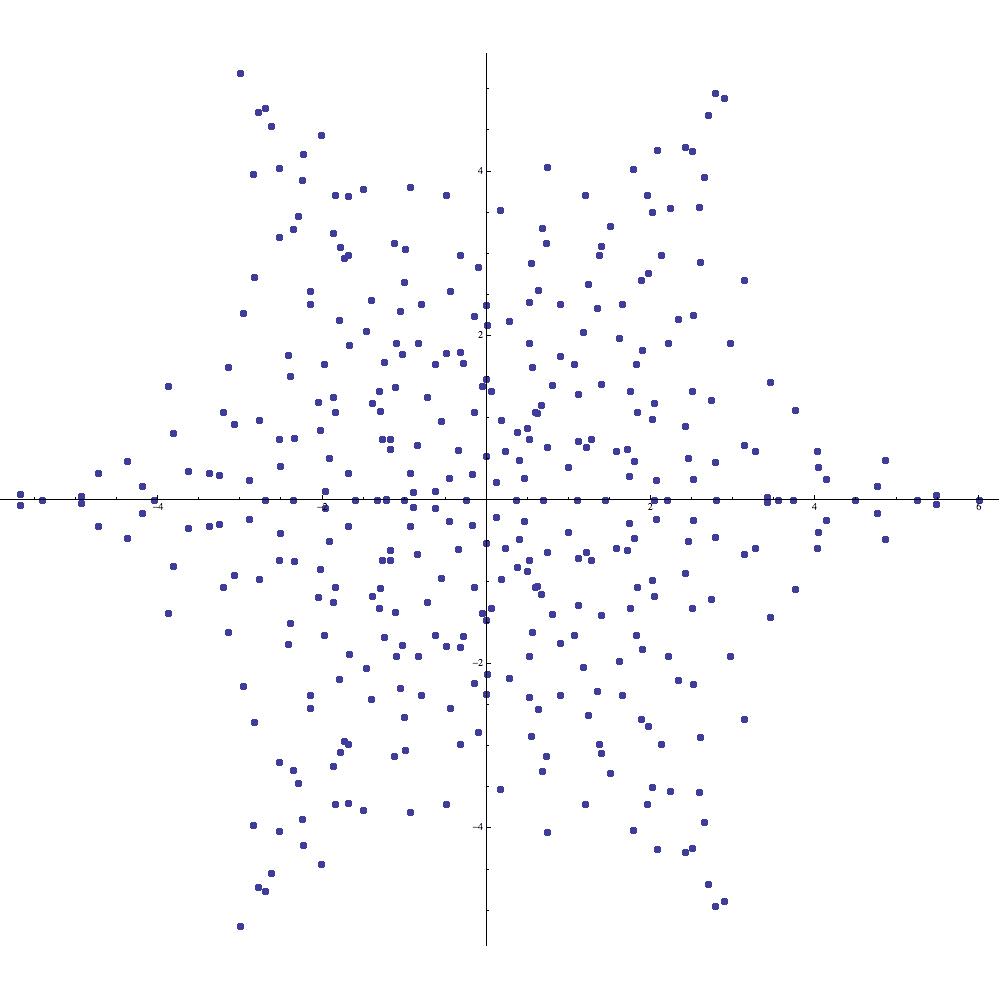}
	                \caption{$n=21$}
	        \end{subfigure}

		\begin{subfigure}[b]{0.3\textwidth}
	                \centering
	                \includegraphics[width=\textwidth]{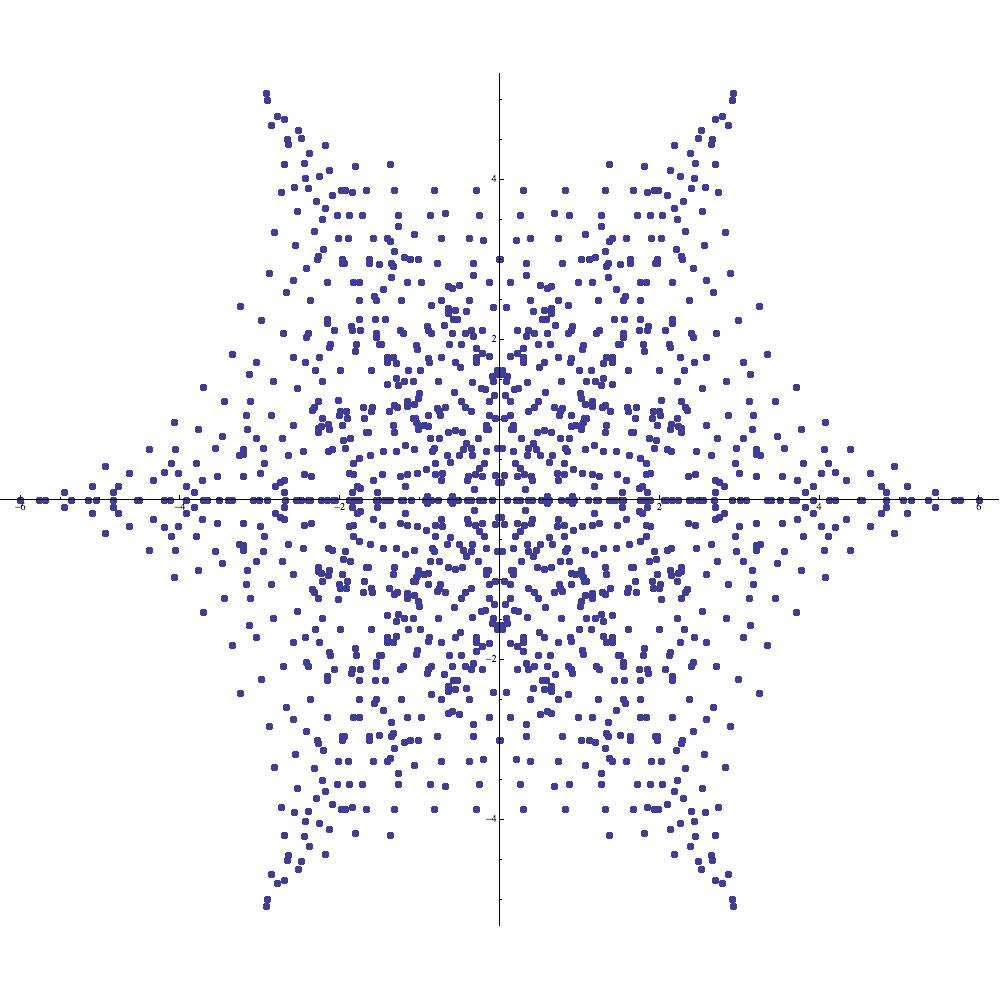}
	                \caption{$n=22$}
	        \end{subfigure}
			\quad
	        \begin{subfigure}[b]{0.3\textwidth}
	                \centering
	                \includegraphics[width=\textwidth]{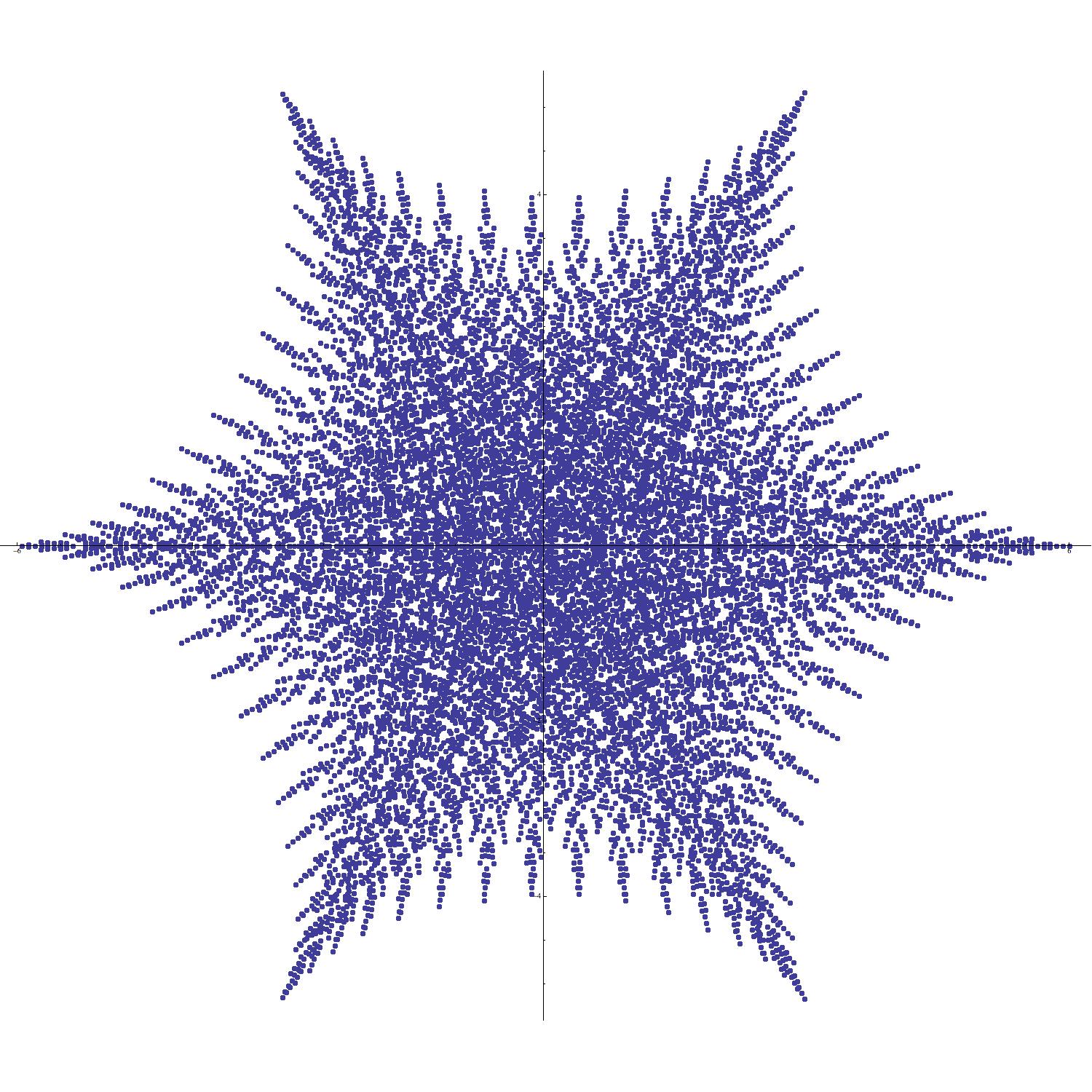}
	                \caption{$n=23$}
	        \end{subfigure}
	        \quad
	        \begin{subfigure}[b]{0.3\textwidth}
	                \centering
	                \includegraphics[width=\textwidth]{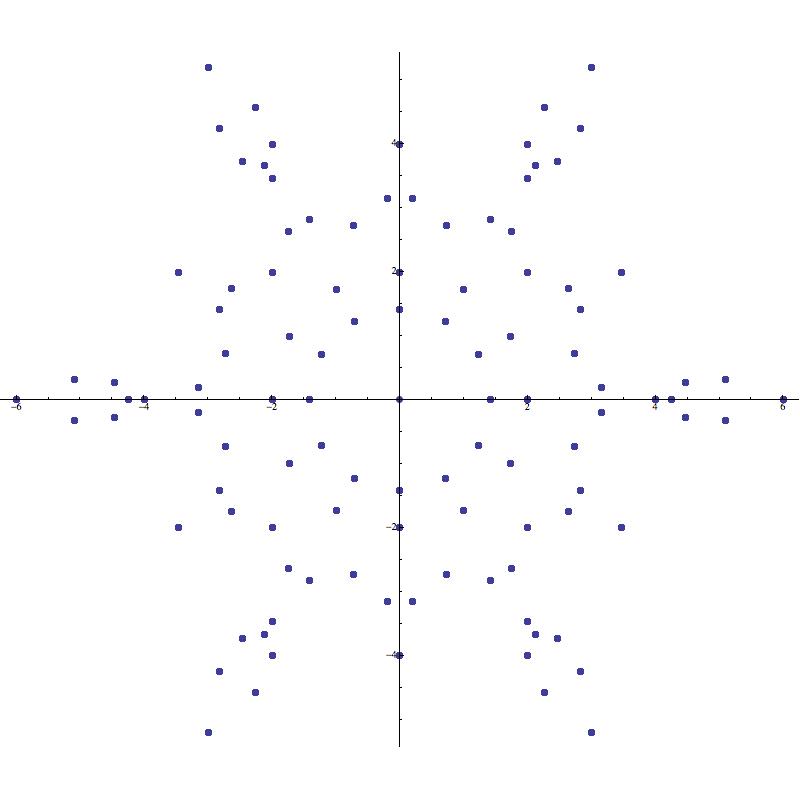}
	                \caption{$n=24$}
	        \end{subfigure}
	        
	        \caption{Plots of the supercharacter $\sigma_X:(\Z/n\Z)^6\to\C$ corresponding
	        to the orbit $X = S_6(1,1,1,1,1,n-5)$ for several values of $n$.  The density of the image
	        depends upon $n/(n,6)$, being high for the primes $n = 19$ and $n =23$ and very
	        low for $n = 24$.}
	        \label{FigureHypocycloid}
	\end{figure}		

	Before proceeding, we should remark that the fact that the range of the multivariate complex mapping \eqref{eq-HypocycloidMap}
	is indeed a filled hypocycloid is familiar to specialists in Lie theory and mathematical physics.  
	Indeed, if $U$ belongs to $SU(d)$, the Lie group consisting of all $d \times d$ unitary matrices having
	determinant $1$, then the eigenvalues $\lambda_1,\lambda_2,\ldots,\lambda_d$ of $U$ are of unit modulus
	and satisfy $\lambda_1 \lambda_2 \cdots \lambda_d = 1$.  In particular, this implies that
	\begin{equation*}
		\tr U = \lambda_1 + \lambda_2 + \cdots + \lambda_{d-1} + \frac{1}{\lambda_1\lambda_2 \cdots \lambda_{d-1}}.
	\end{equation*}
	Therefore the range of the mapping \eqref{eq-HypocycloidMap} is precisely the set of all possible traces
	of matrices in $SU(d)$ (we thank G.~Sarkis for calling this result to our attention).  
	A similar treatment of several other classical Lie groups can be found in \cite{Kaiser}.

\subsection{Fireballs}

	For the sake of reference, we recall the key identity \eqref{eq-xyjk} and write it here for convenient reference:
	\begin{equation}\label{eq-MainIdentity}
		\sigma_{X + j \vec{1}}(Y + k \vec{1})
		= e\left( \frac{ [Y]j + [X]k + djk}{n} \right) \sigma_X(Y).
	\end{equation}
	By starting with a known supercharacter $\sigma_X$ whose image is of interest to us and then
	selecting the parameters $j$ and $k$ appropriately, we can often generate spectacular
	new supercharacter plots.
	This turns out to be a matter of some finesse and experimentation.  Rather than attempt to formulate
	a general theorem, we prefer to present the main ideas 
	through a series of instructive examples.
	
	\begin{Example}\label{ExampleFireball}
		By Proposition \ref{PropositionHypo}, the image of the
		supercharacter $\sigma_{ S_4(1,1,1,12)}:(\Z/15\Z)^4\to\C$ is contained in
		the four-sided hypocycloid (sometimes referred to as an \emph{astroid}) centered at the
		origin and having one of its cusps located at the point $z=4$ (see Figure \ref{SubfigureFB4}).  
		\begin{figure}[h]
			\begin{subfigure}{0.3\textwidth}
				\centering
				\includegraphics[width=\textwidth]{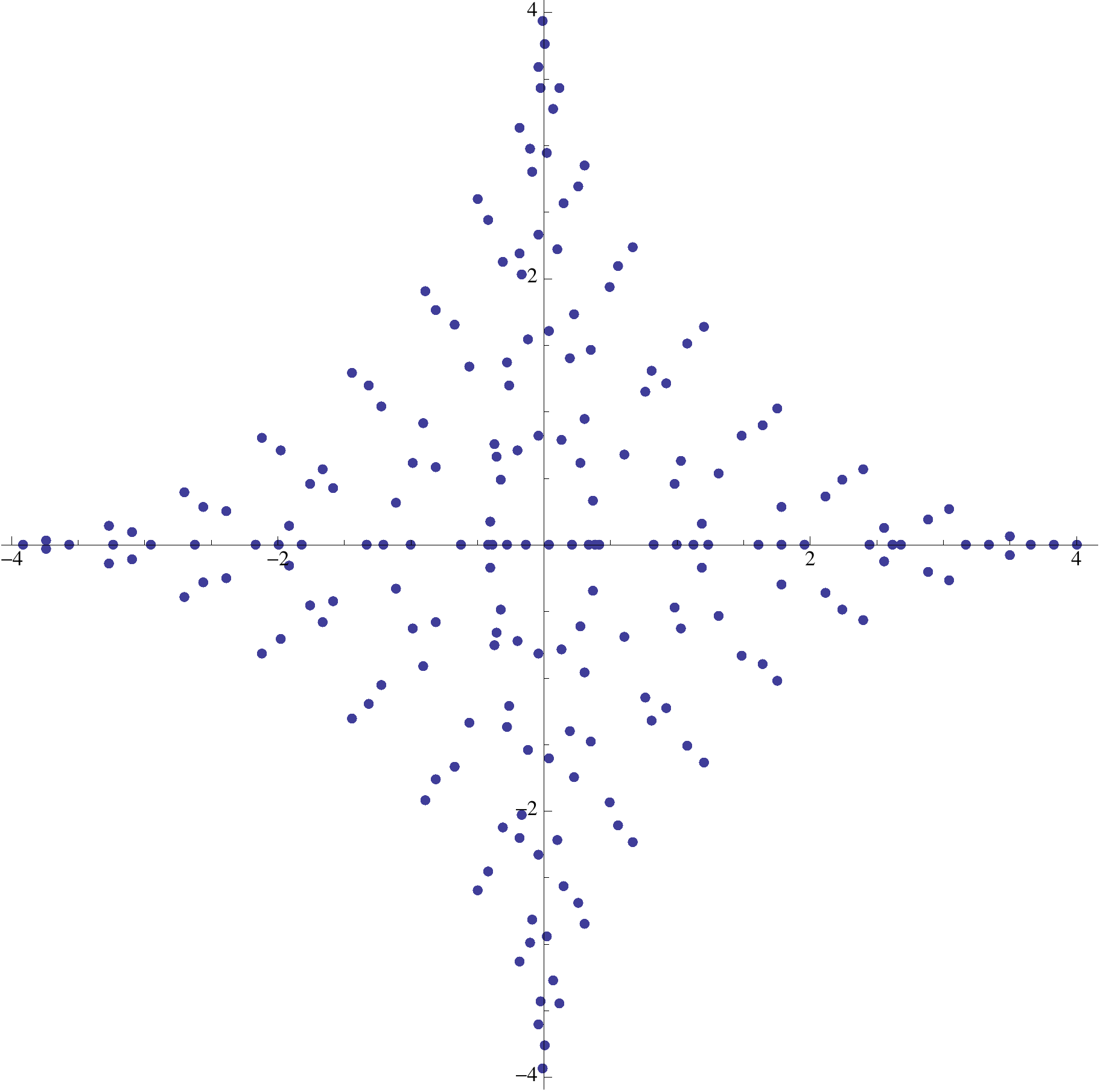}
				\caption{$X = S_4(1,1,1,12)$, $4$ cusps, $D_1$ symmetry}
				\label{SubfigureFB4}
			\end{subfigure}
			\quad
			\begin{subfigure}{0.3\textwidth}
				\centering
				\includegraphics[width=\textwidth]{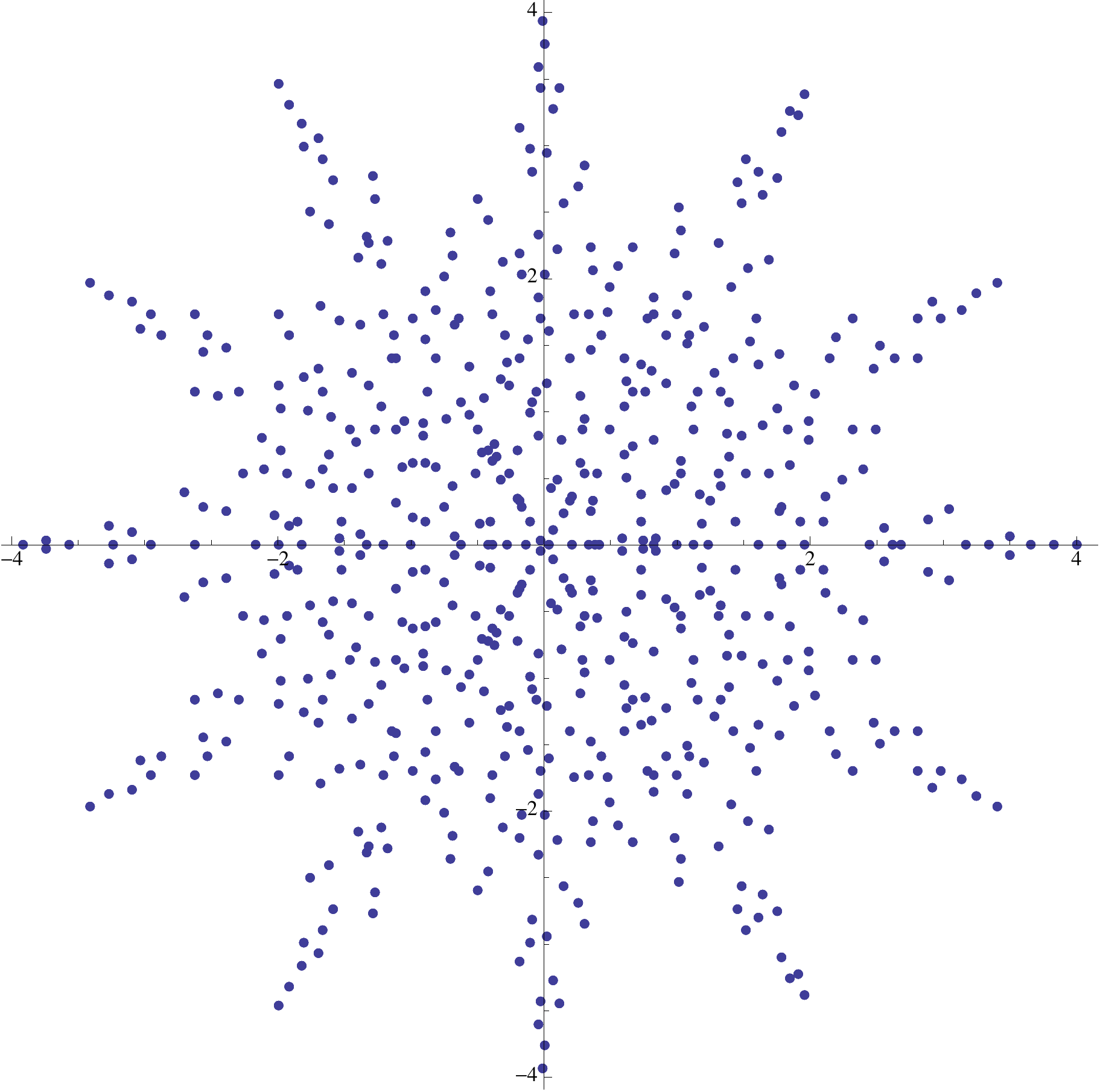}
				\caption{$X=S_4(2,6,6,6)$, $12$ cusps, $D_3$ symmetry}
				\label{SubfigureFB12}
			\end{subfigure}
			\quad
			\begin{subfigure}{0.3\textwidth}
				\centering
				\includegraphics[width=\textwidth]{FB-n15d4-1,1,1,3.pdf}
				\caption{$X = S_4(0,4,4,4)$, $20$ cusps, $D_5$ symmetry}
				\label{SubfigureFB20}
			\end{subfigure}
			\caption{Images of supercharacters $\sigma_X:(\Z/15\Z)^4\to\C$
			which illustrate the method of Example \ref{ExampleFireball}.}
			\label{FigureFB}
		\end{figure}
		Using the identity
		\eqref{eq-MainIdentity} with the parameters
		$n=15$, $d=4$, $[X] = 15$, $j = 5$, and $k = 4(\ell - [Y])$,
		we find that 
		\begin{equation*}
			\sigma_{S_4(2,6,6,6)}(Y+k\vec{1}) = e\left(\frac{\ell}{3}\right) \sigma_X(Y).
		\end{equation*}
		Since the integer $\ell$ is arbitrary, 
		it follows from the preceding computation that the image of $\sigma_{S_4(2,6,6,6)}:(\Z/15\Z)^4\to\C$
		is precisely the union of three rotated copies of our original image, yielding a ``fireball'' with 
		$4 \times 3 = 12$ cusps (see Figure \ref{SubfigureFB12}).  Because the original
		image (Figure \ref{SubfigureFB4}) has no rotational symmetry, the resulting
		fireball image possesses only $D_3$ symmetry, despite the fact that a careless glance at the figure
		suggests the existence of $D_{12}$ symmetry.
		Along similar lines, setting $j = 3$ and $k = 4(\ell - [Y])$ yields
		\begin{equation*}
			\sigma_{S_4(0,4,4,4)}(Y + k \vec{1}) = e\left( \frac{\ell}{5} \right) \sigma_X(Y),
		\end{equation*}
		so that the resulting ``fireball'' has $4 \cdot 5 = 20$ cusps (see Figure \ref{SubfigureFB20}).  
	\end{Example}

	\begin{Example}
		In light of Proposition \ref{PropositionHypo},
		the image of the supercharacter $\sigma_{S_5(1,1,1,1,8)}:(\Z/12\Z)^5\to\C$
		resembles a filled hypocycloid with five cusps (see Figure \ref{SubfigureFB2.1}).  
		It follows from the method of the preceding example that replacing
		$(1,1,1,1,8)$ with $(1,1,1,1,8) + j \vec{1}$ yields a ``fireball'' having
		$\frac{5\cdot12}{(j,12)}$ cusps (see Figure \ref{FigureFB2}).
	\end{Example}

	\begin{figure}[h]
		\begin{subfigure}{0.3\textwidth}
			\centering
			\includegraphics[width=\textwidth]{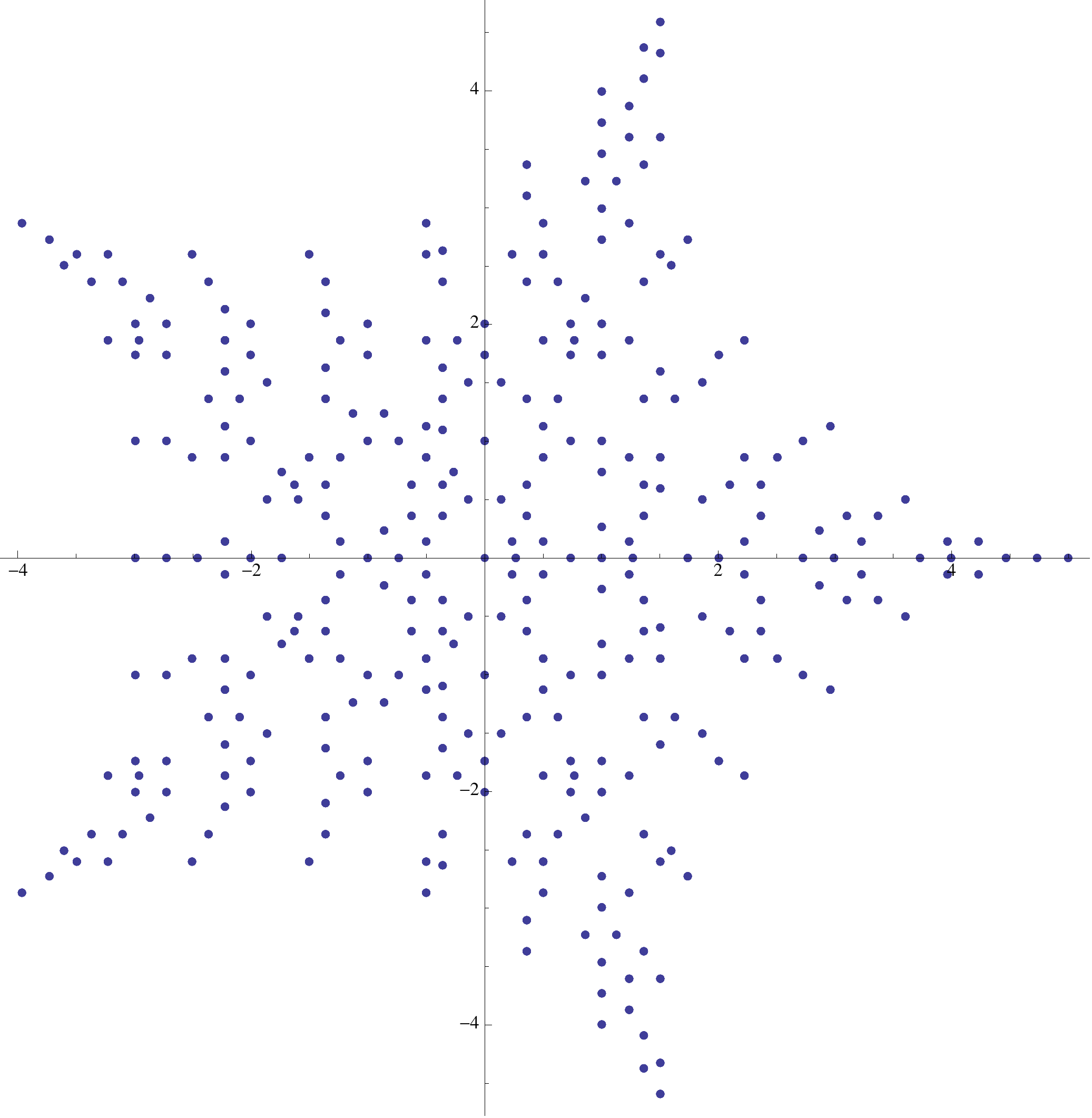}
			\caption{$j=0$, $D_1$-symmetry, $5$ cusps}
			\label{SubfigureFB2.1}
		\end{subfigure}
		\quad
		\begin{subfigure}{0.3\textwidth}
			\centering
			\includegraphics[width=\textwidth]{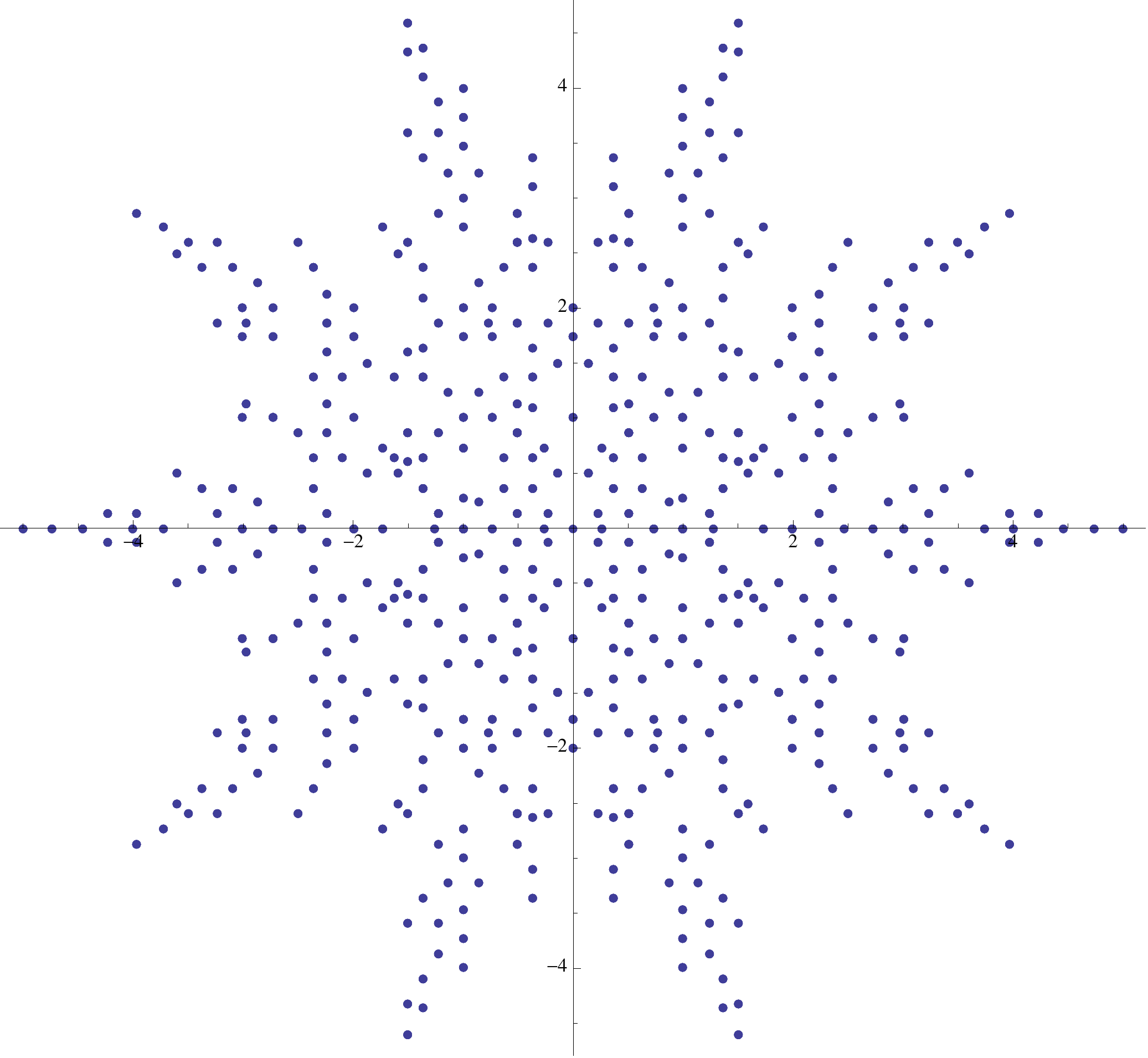}
			\caption{$j=6$, $D_2$-symmetry, $10$ cusps}
			\label{SubfigureFB2.2}
		\end{subfigure}
		\quad
		\begin{subfigure}{0.3\textwidth}
			\centering
			\includegraphics[width=\textwidth]{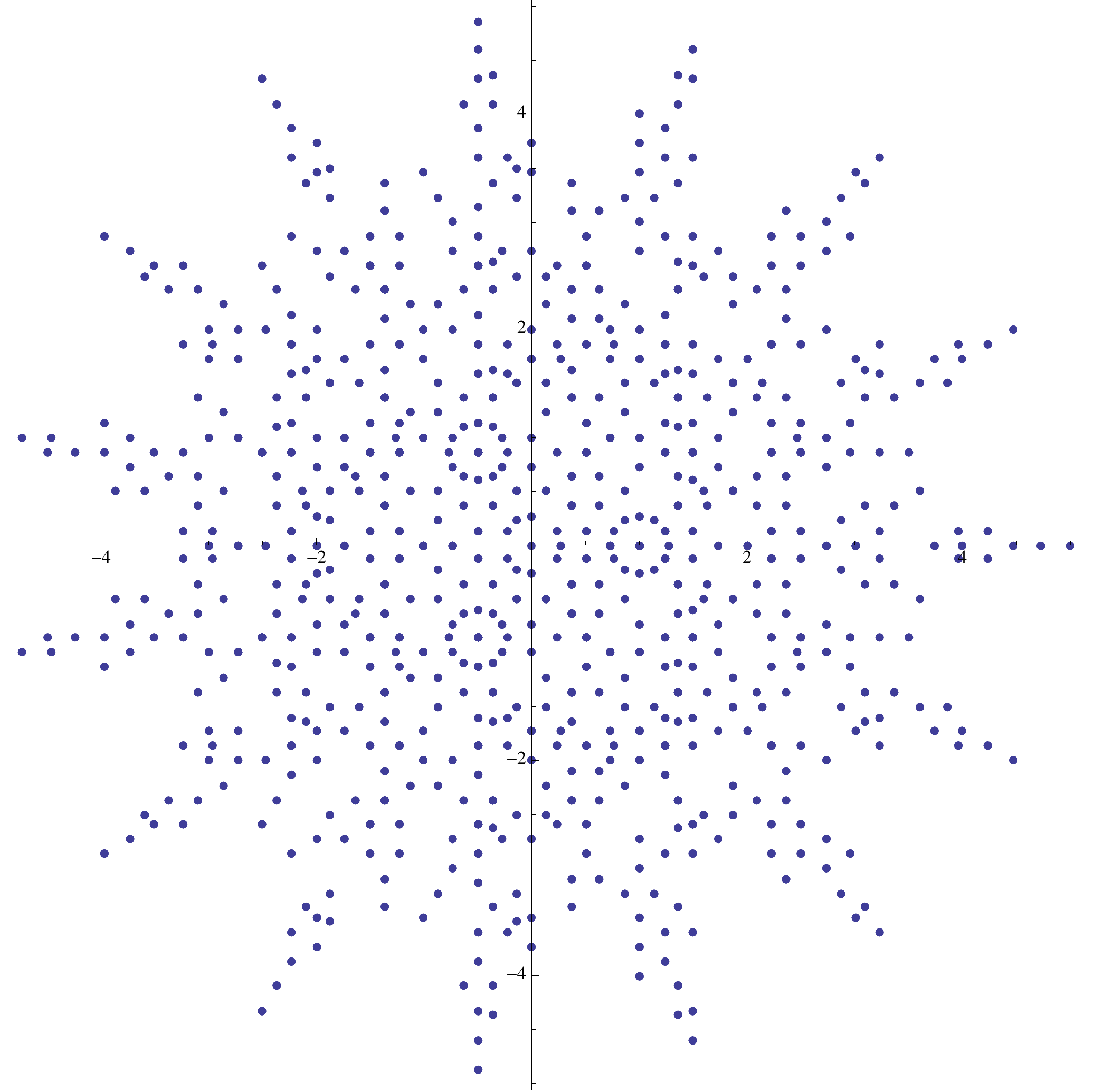}
			\caption{$j=4$, $D_3$-symmetry, $15$ cusps}
			\label{SubfigureFB2.3}
		\end{subfigure}
		\bigskip
		
		\begin{subfigure}{0.3\textwidth}
			\centering
			\includegraphics[width=\textwidth]{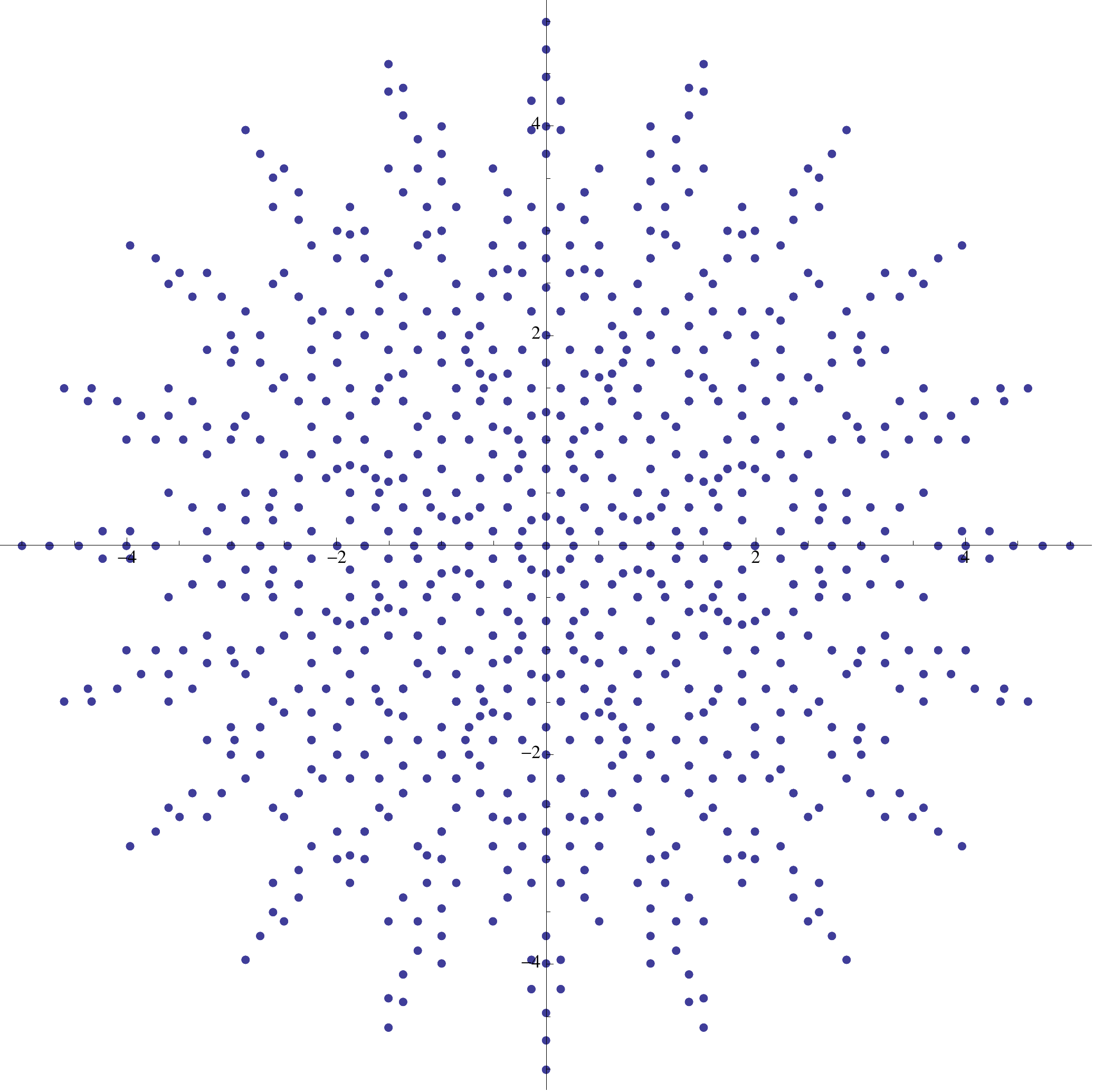}
			\caption{$j=3$, $D_4$-symmetry, $20$ cusps}
			\label{SubfigureFB2.4}
		\end{subfigure}
		\quad
		\begin{subfigure}{0.3\textwidth}
			\centering
			\includegraphics[width=\textwidth]{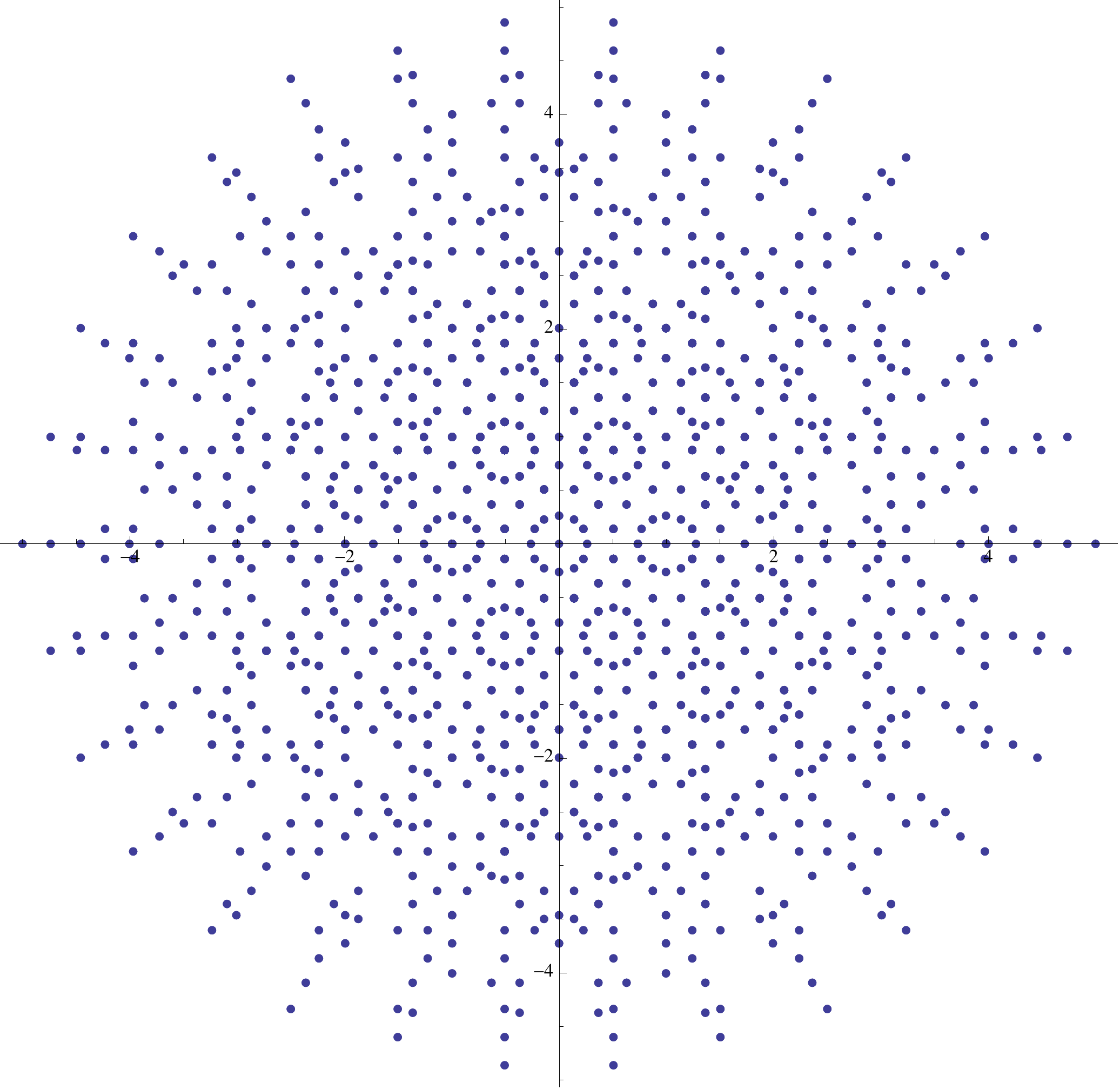}
			\caption{$j=2$, $D_6$-symmetry, $30$ cusps}
			\label{SubfigureFB2.5}
		\end{subfigure}
		\quad
		\begin{subfigure}{0.3\textwidth}
			\centering
			\includegraphics[width=\textwidth]{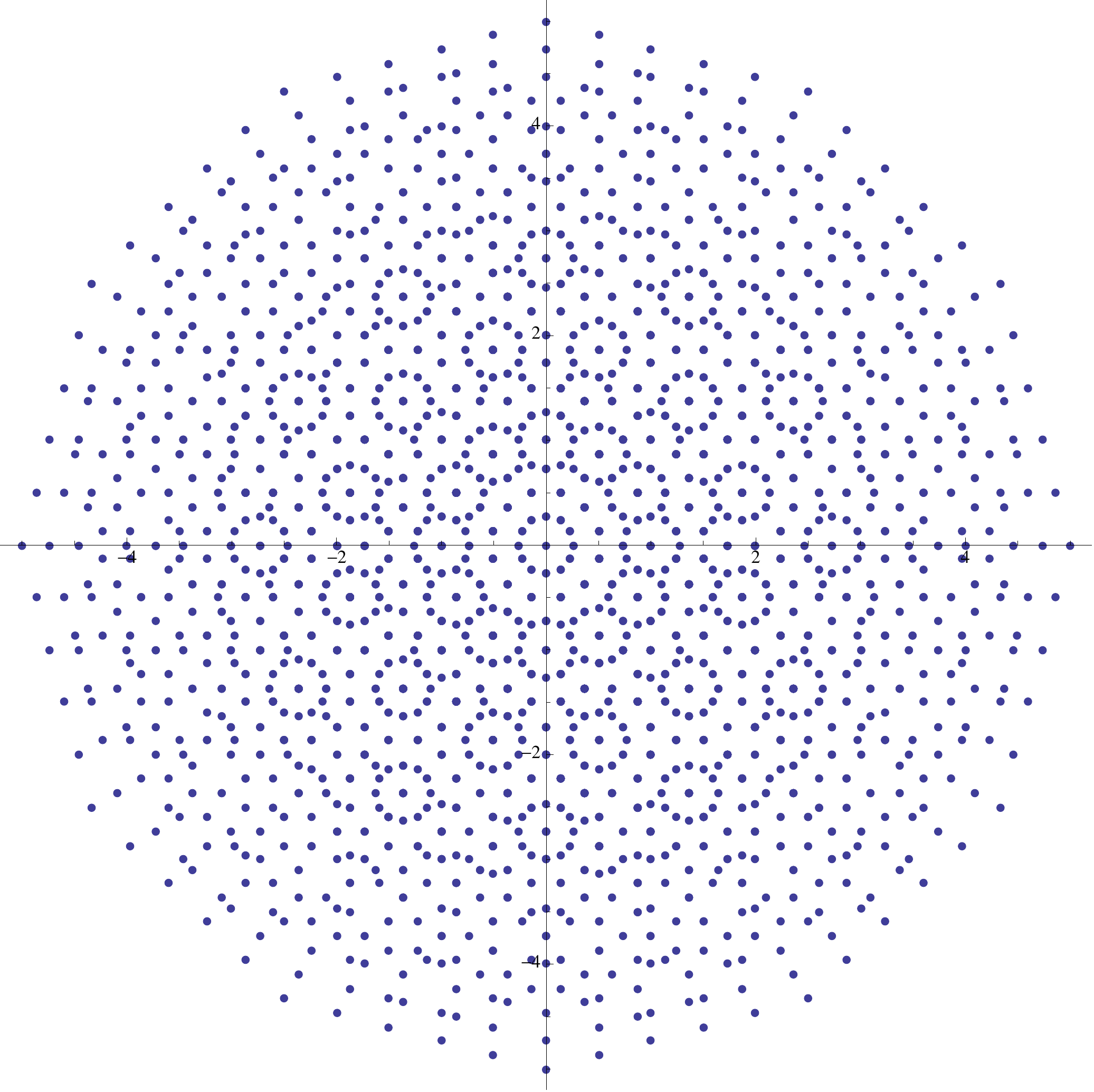}
			\caption{$j=1$, $D_{12}$-symmetry, $60$ cusps}
			\label{SubfigureFB2.6}
		\end{subfigure}
		\bigskip
		\caption{Images of the supercharacter $\sigma_X:(\Z/12\Z)^5\to\C$ where
		$X = S_5(1,1,1,1,8)+j\vec{1}$ for several values of $j$.}
		\label{FigureFB2}
	\end{figure}

\subsection{Asymptotic results}\label{SubsectionMatrix}
	As suggested in the comments following Proposition \ref{PropositionHypo}
	and illustrated in Figure \ref{FigureHypocycloid}, it is occasionally possible to
	predict the asymptotic behavior of specific families of supercharacter plots
	as the modulus $n$ tends to infinity (perhaps with some congruence restrictions imposed upon $n$).
	We develop here a general approach to determining the asymptotic appearance of certain families of symmetric supercharacter plots.

	To begin with, fix $n$ and $d$ and let $X = \{ \vec{x}_1, \vec{x}_2, \ldots, \vec{x}_r\}$ be a $S_d$-orbit in $(\Z/n\Z)^d$.
	We next form the $d \times r$ matrix $A = [ \vec{x}_1 | \vec{x}_2 | \cdots | \vec{x}_r]$ whose
	columns are the vectors $\vec{x}_i$ in $(\Z/n\Z)^d$.  
	Suppose that $A$ can be row reduced over $\Z/n\Z$ to obtain
	a simpler matrix $B = [ \vec{b}_1 | \vec{b}_2 | \cdots | \vec{b}_r]$.  
	This is equivalent to asserting that $B = RA$ where $R$ is a $d\times d$ matrix over 
	$\Z/n\Z$ such that $\det R$ belongs to $(\Z/n\Z)^{\times}$.
	Moreover, we also see that $R\vec{x}_i = \vec{b}_i$ for $i=1,2,\ldots,d$.
	Since $\det R$ is invertible modulo $n$, it follows from Cramer's Rule
	that for any $\vec{\xi} = (\xi_1,\xi_2,\ldots,\xi_d)$ in $(\Z/n\Z)^d$ the system
	$R^{-T}\vec{y} = \vec{\xi}$ has a solution $\vec{y}$.  Putting this all together we find that
	\begin{align*}
		\sigma_X(\vec{y})
		&= \sum_{\ell =1}^r e\left( \frac{\vec{x}_{\ell} \cdot \vec{y}}{n} \right)  
		= \sum_{\ell =1}^r e\left( \frac{\vec{x}_{\ell} \cdot R^T\vec{\xi}}{n} \right) \\ 
		&= \sum_{\ell =1}^r e\left( \frac{R\vec{x}_{\ell} \cdot \vec{\xi}}{n} \right) 
		= \sum_{\ell =1}^r e\left( \frac{\vec{b}_{\ell} \cdot \vec{\xi}}{n} \right)\\
		&= \sum_{\ell=1}^r e\left(\frac{b_{1,\ell}\xi_1 + b_{2,\ell}\xi_2 + \cdots + b_{d,\ell} \xi_d }{n} \right) \\
		&= \sum_{\ell=1}^r e\left( \frac{\xi_1}{n} \right)^{b_{1,\ell}}\!\!\!\! e\left( \frac{\xi_2}{n} \right)^{b_{2,\ell}}
			\!\!\!\!\!\!\cdots e\left( \frac{\xi_d}{n} \right)^{b_{d,\ell}}.
	\end{align*}
	If the final $k$ rows of $B$ are zero (which happens, for instance, if $[X]\equiv 0 \pmod{n}$), 
	then the preceding computations imply that
	the image of $\sigma_X:(\Z/n\Z)^d\to\C$ roughly approximates the image of the function $g:\T^{d-k}\to\C$ defined by
	\begin{equation}\label{eq-g}
		g(z_1,z_2,\ldots,z_{d-k}) = \sum_{\ell=1}^r \prod_{j=1}^{d-k} z_j^{b_{j \ell}}.
	\end{equation}
	Although there may be complications which arise depending upon the arithmetic relationships between the
	entries of $B$ and the modulus $n$, in many cases the preceding recipe is sufficient to determine the
	asymptotic appearance of a family of supercharacter plots.

	\begin{Example}[Hypocycloids revisited]\label{ExampleHypo}
		Suppose that $(n,d) = 1$ and let $X = S_d(1,1,\ldots,1,1-d)$ denote the $S_d$-orbit of the vector
		$(1,1,\ldots,1,1-d)$ in $(\Z/n\Z)^d$.  Following the preceding recipe note that 
		\begin{equation*}\small
			\underbrace{
			\begin{bmatrix}
				0 &1 & 1 & \cdots & 2\\
				1 & 0 & 1 & \cdots & 2 \\
				1 & 1 & 0 & \cdots & 2 \\
				\vdots & \vdots & \vdots & \ddots & \vdots \\
				1 & 1 & 1 & \cdots & 1
			\end{bmatrix}
			}_R
			\underbrace{
			\begin{bmatrix}
				1-d &1 & 1 & \cdots & 1 \\
				1 & 1-d & 1 & \cdots & 1 \\
				1 & 1 & 1-d & \cdots & 1 \\
				\vdots & \vdots & \vdots & \ddots & \vdots \\
				1 & 1 & 1 & \cdots & 1-d
			\end{bmatrix}
			}_A
			=
			\underbrace{
			\begin{bmatrix}
				1 & 0  & \cdots & 0 & -1 \\
				0 & 1 & \cdots & 0 & -1\\
				 \vdots & \vdots & \ddots & \vdots & \vdots \\
				0 & 0 &  \cdots & 1 & -1 \\
				0 & 0 &  \cdots & 0 & 0
			\end{bmatrix}
			}_B,
		\end{equation*}
		where $\det R = (-1)^{d+1} d$ is invertible over $\Z/n\Z$ (here $R$ is obtained from the all ones
		matrix by first subtracting the identity matrix and then adding $1$s to every entry of the $d$th column).
		It follows that the image of $\sigma_X:(\Z/n\Z)^d\to\C$ is contained in the range of the
		function \eqref{eq-HypocycloidMap} encountered in our previous treatment of hypocycloids.
	\end{Example}

	\begin{Example}[Hummingbirds]\label{ExampleHumming}
		Suppose that $(n,5)=1$ and let $X = S_3(1,2,n-3)$.  In this case,
		\begin{equation*}
			\underbrace{
			\megamatrix{3}{1}{0}{2}{-1}{0}{1}{1}{1}
			}_R
			\underbrace{
			\begin{bmatrix}
				 1 & 1 & 2 & 2 & -3 & -3 \\
				 2 & -3 & 1 & -3 & 1 & 2 \\
				 -3 & 2 & -3 & 1 & 2 & 1 \\
			\end{bmatrix}
			}_A
			=
			\underbrace{
			\begin{bmatrix}
				 5 & 0 & 7 & 3 & -8 & -7 \\
				 0 & 5 & 3 & 7 & -7 & -8 \\
				 0 & 0 & 0 & 0 & 0 & 0 \\
			\end{bmatrix}
			}_B,
		\end{equation*}
		where $\det R = -5$ is a unit in $\Z/n\Z$.  It follows that as $n$ increases, the image
		of $\sigma_X:(\Z/n\Z)^3 \to \C$ resembles the image of the function
		$g:\T^2 \to \C$ defined by
		\begin{equation}\label{eq-HummingMap}
			g(z_1,z_2) = z_1^5 + z_2^5 + z_1^7 z_2^3 + z_1^3 z_2^7 + \frac{1}{z_1^8 z_2^7} + \frac{1}{z_1^7 z_2^8}.
		\end{equation}
		The result is a plot which, for sufficiently large $n$, one might say, resembles a hummingbird in flight
		(see Figure \ref{FigureHummingbird}).
	\end{Example}

	\begin{figure}[htb!]
		\centering
		\begin{subfigure}[b]{0.3\textwidth}
			\centering
			\includegraphics[width=\textwidth]{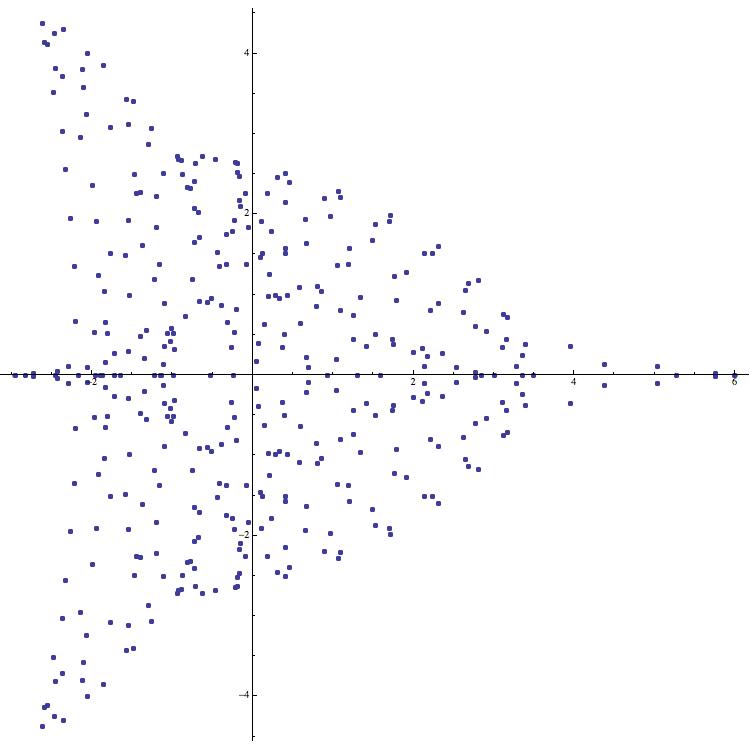}
			\caption{$n=47$}
	        \end{subfigure}
		\quad
		\begin{subfigure}[b]{0.3\textwidth}
			\centering
			\includegraphics[width=\textwidth]{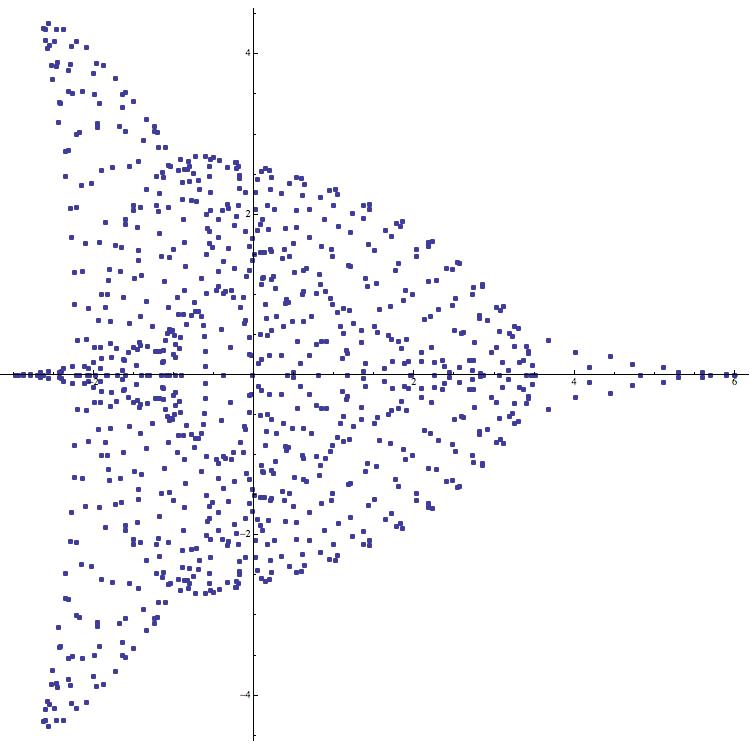}
			\caption{$n=73$}
	        \end{subfigure}
		\quad
		\begin{subfigure}[b]{0.3\textwidth}
			\centering
			\includegraphics[width=\textwidth]{HM-n173d3x1,2,170.jpg}
			\caption{$n=173$}
	        \end{subfigure}

	        \caption{Plots of $\sigma_X:(\Z/n\Z)^3\to\C$ corresponding to $X = S_3(1,2,n-3)$ for several values of $n$.}
	        \label{FigureHummingbird}
	\end{figure}	
	
	\begin{Example}[Manta rays]\label{ExampleManta}
		Consider the symmetric supercharacter $\sigma_X:(\Z/n\Z)^4\to\C$ corresponding to 
		$X = S_4(0,1,1,n-2)$.	  Working over $\Z/n\Z$ we find that
		\begin{align*}
			&\footnotesize
			\underbrace{
			\begin{bmatrix}
				 0 & 1 & 0 & 0 \\
				 0 & 0 & 1 & 0 \\
				 1 & 2 & -1 & 0 \\
				 1 & 1 & 1 & 1 \\
			\end{bmatrix}
			}_R
			\underbrace{
			\left[
			\begin{array}{cccccccccccc}
				 -2 & 1 & -2 & 0 & 0 & 1 & 1 & 1 & 1 & -2 & 0 & 1 \\
				 1 & 0 & 1 & 1 & 1 & -2 & -2 & 0 & 1 & 0 & -2 & 1 \\
				 0 & 1 & 1 & 1 & -2 & 1 & 0 & -2 & 0 & 1 & 1 & -2 \\
				 1 & -2 & 0 & -2 & 1 & 0 & 1 & 1 & -2 & 1 & 1 & 0 \\
			\end{array}
			\right]
			}_A\\
			&\qquad \footnotesize =
			\underbrace{
			\left[
			\begin{array}{cccccccccccc}
				 1 & 0 & 1 & 1 & 1 & -2 & -2 & 0 & 1 & 0 & -2 & 1 \\
				 0 & 1 & 1 & 1 & -2 & 1 & 0 & -2 & 0 & 1 & 1 & -2 \\
				 0 & 0 & -1 & 1 & 4 & -4 & -3 & 3 & 3 & -3 & -5 & 5 \\
				 0 & 0 & 0 & 0 & 0 & 0 & 0 & 0 & 0 & 0 & 0 & 0 \\
			\end{array}
			\right]
			}_B,
		\end{align*}
		where $\det R = 1$.  Thus the image of $\sigma_X$ is contained in the image of the function
		$g:\T^3\to\C$ defined by
		\begin{align*}
			g(z_1,z_2,z_3)
			&= z_1+z_2+\left(\frac{z_1 z_2}{z_3}+z_1 z_2 z_3\right)
			+\left(\frac{z_1 z_3^4}{z_2^2}+\frac{z_2}{z_1^2 z_3^4}\right)
			+\left(\frac{1}{z_1^2 z_3^3}+\frac{z_3^3}{z_2^2}\right) \\
			&\qquad +\left(z_1 z_3^3+\frac{z_2}{z_3^3}\right) 
			+\left(\frac{z_2}{z_1^2 z_3^5}+\frac{z_1 z_3^5}{z_2^2} \right).
		\end{align*}
		For large $n$, the image of $\sigma_X:(\Z/n\Z)^4\to\C$ bears an uncanny resemblance
		to a manta ray followed by a trail of bubbles (see Figure \ref{FigureManta}).
	\end{Example}
	
	\begin{figure}[htb!]
		\centering
		\begin{subfigure}[b]{0.3\textwidth}
	                \centering
	                \includegraphics[width=\textwidth]{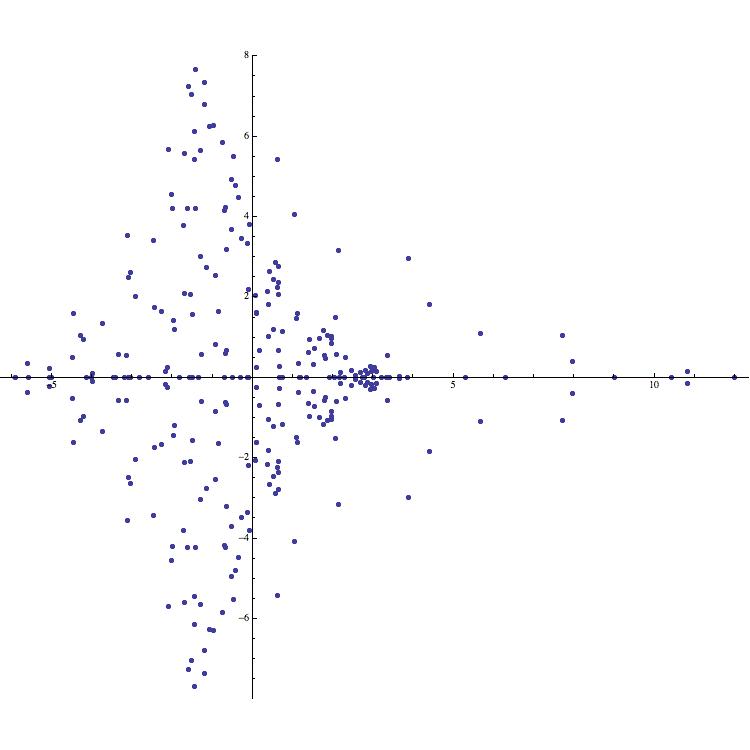}
	                \caption{$n=17$}
	        \end{subfigure}
	        \quad
		\begin{subfigure}[b]{0.3\textwidth}
	                \centering
	                \includegraphics[width=\textwidth]{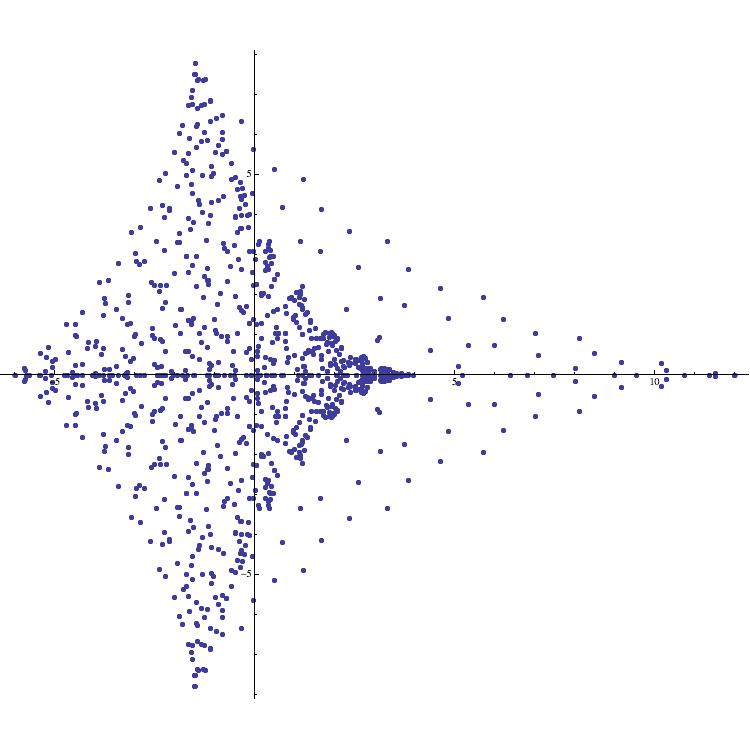}
	                \caption{$n=27$}
	        \end{subfigure}
	        \quad
		\begin{subfigure}[b]{0.3\textwidth}
	                \centering
	                \includegraphics[width=\textwidth]{MANTA-n47d4x0,1,1,45.jpg}
	                \caption{$n=47$}
	        \end{subfigure}

	        \caption{Plots of $\sigma_X:(\Z/n\Z)^4\to\C$ corresponding
	        to $X = S_4(0,1,1,n-2)$ for several values of $n$.} 
	        \label{FigureManta}
	\end{figure}		
	
	We conclude this note with a final example which combines a number of the ideas
	and techniques developed in the preceding pages.	
	
	\begin{Example}[Manta rays as graphical elements]
		In light of Example \ref{ExampleManta}, plotting the symmetric supercharacter 
		$\sigma_X:(\Z/30\Z)^4\to\C$ corresponding to $X = S_4(0,1,1,28)$ yields the familiar
		manta ray (see Figure \ref{SubfigureMR0}).  Using 
		\eqref{eq-MainIdentity} we are able to construct a number of intriguing new images
		based upon the initial manta ray image.  Applying \eqref{eq-MainIdentity} with the parameters
		$n=30$, $d=4$, $j=15$, and $k=0$ we find that
		\begin{equation*}
			\sigma_{S_4(15,16,16,13)}(Y) = e\big( \tfrac{[Y]}{2}\big) \sigma_X(Y).
		\end{equation*}
		Roughly speaking, the image of $\sigma_{S_4(15,16,16,13)}$ is obtained by leaving
		half of the points in $\sigma_X$ in place while negating the other half.  This is precisely the
		behavior that is illustrated in Figure \ref{SubfigureMR15}.  In particular, the image
		in Figure \ref{SubfigureMR15} is \emph{not} symmetric with respect to the imaginary axis.
		On the other hand, applying \eqref{eq-MainIdentity} with $n=30$, $d=4$, $j = 10$, and $k = \ell - [Y]$ yields
		\begin{equation*}
			\sigma_{S_4(10,11,11,8)}(Y + k \vec{1}) = e\big(\tfrac{\ell}{3} \big) \sigma_X(Y)
		\end{equation*}
		whence the image of $\sigma_{S_4(10,11,11,8)}$ is precisely the union of three rotated copies of Figure 
		\ref{SubfigureMR0} (see Figure \ref{SubfigureMR15}).  Plots corresponding to various values of the parameter $j$
		are provided in Figure \ref{FigureRotate}.  The reader is invited to experiment with \eqref{eq-MainIdentity} in order
		to explain the symmetries of these figures	(e.g., Figure \ref{SubfigureMR5} gives 
		the initial impression of having $D_6$ symmetry, whereas a closer
		inspection reveals that it only enjoys $D_3$ symmetry).
	\end{Example}

	\begin{figure}[h]
		\begin{subfigure}{0.4\textwidth}
			\centering
			\includegraphics[width=\textwidth]{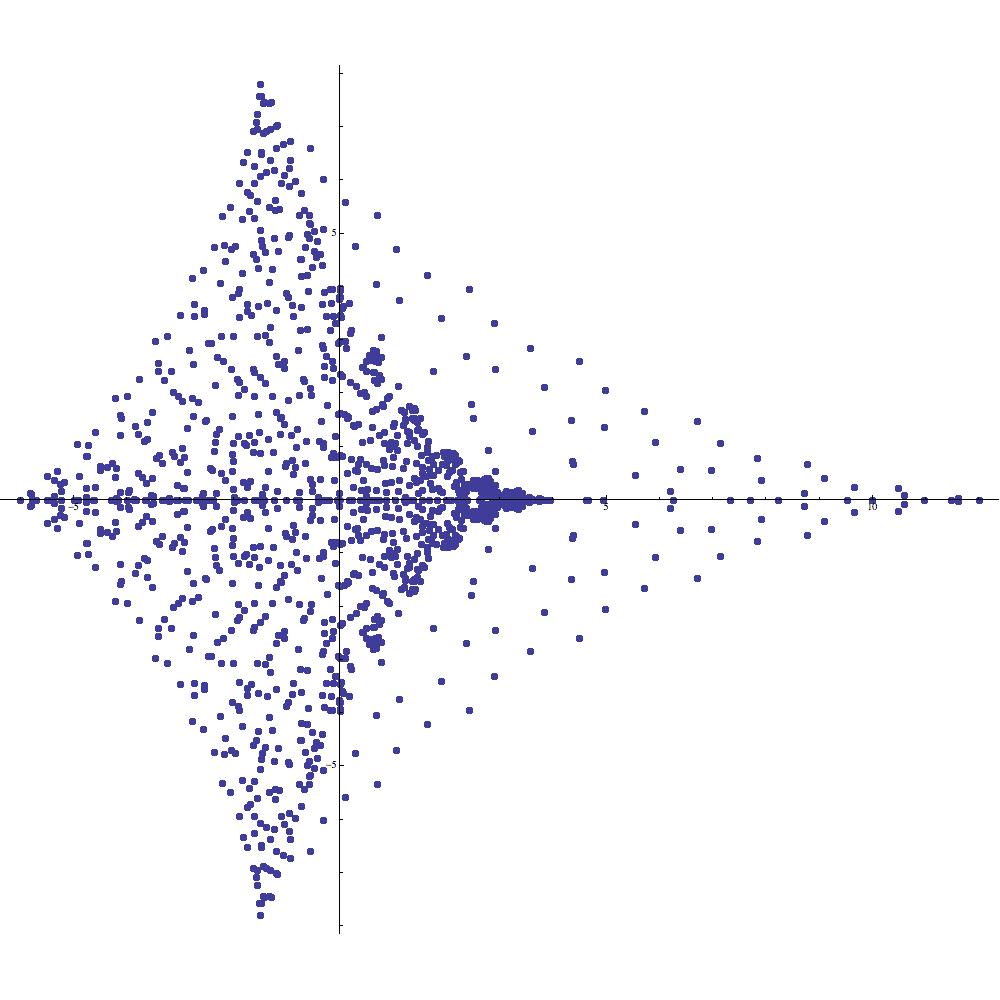}
			\caption{$j=0$}
			\label{SubfigureMR0}
		\end{subfigure}
		\quad
		\begin{subfigure}{0.5\textwidth}
			\centering
			\includegraphics[width=\textwidth]{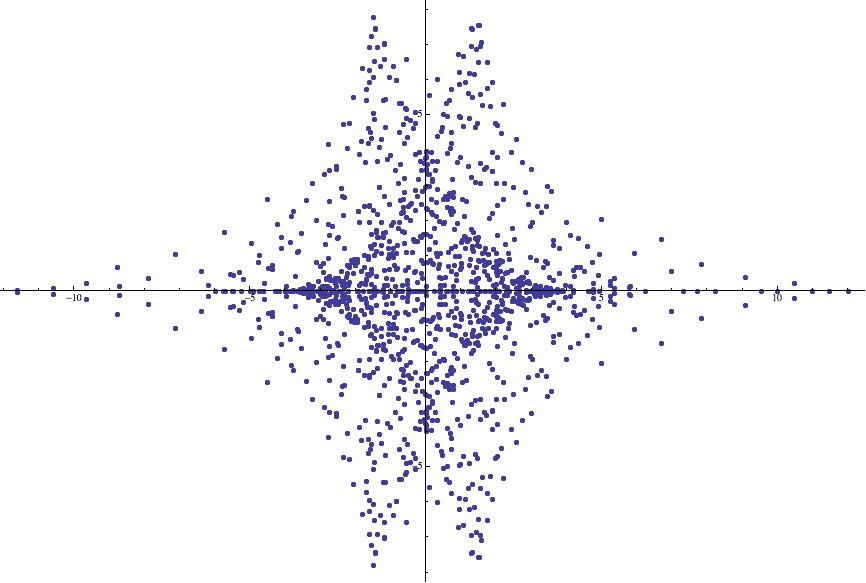}
			\caption{$j=15$}
			\label{SubfigureMR15}
		\end{subfigure}
		\bigskip

		\begin{subfigure}{0.3\textwidth}
			\centering
			\includegraphics[width=\textwidth]{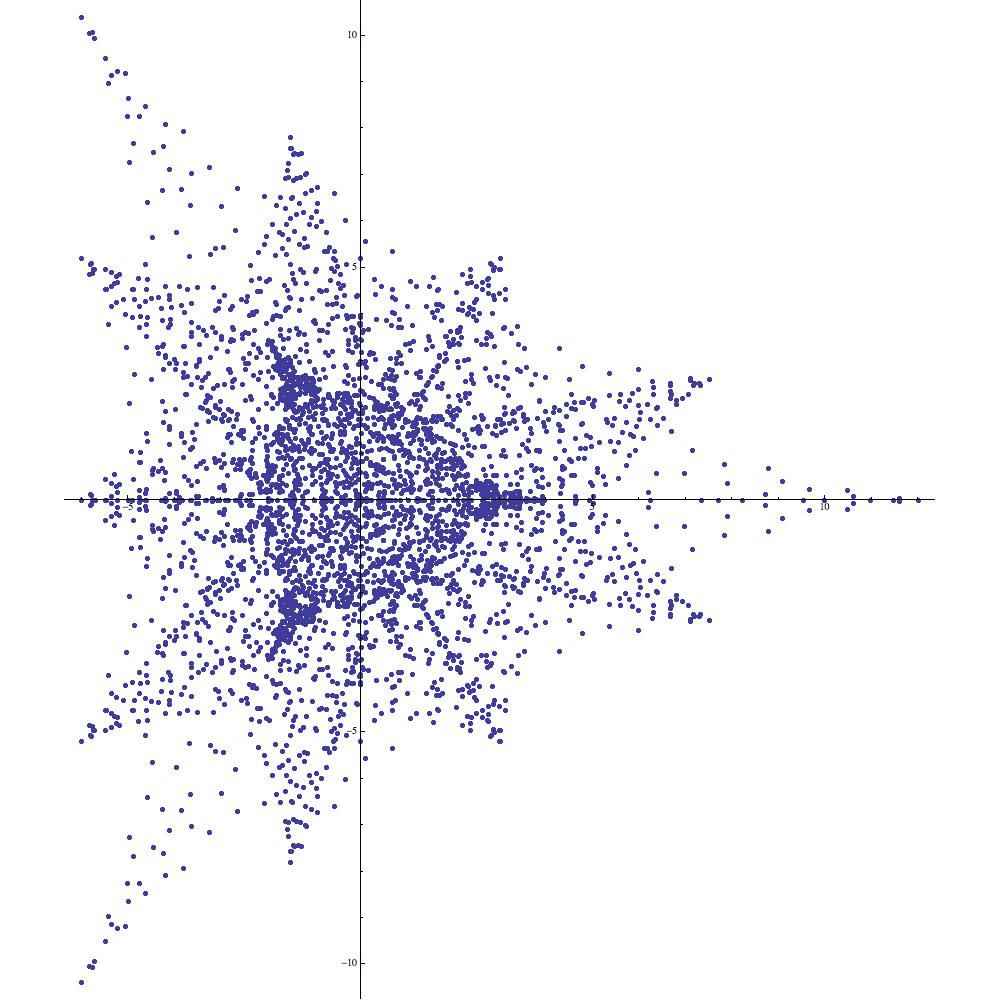}
			\caption{$j=10$}
			\label{SubfigureMR10}
		\end{subfigure}
		\quad
		\begin{subfigure}{0.3\textwidth}
			\centering
			\includegraphics[width=\textwidth]{MANTA-n30d4j6.jpg}
			\caption{$j=6$}
			\label{SubfigureMR6}
		\end{subfigure}
		\quad
		\begin{subfigure}{0.3\textwidth}
			\centering
			\includegraphics[width=\textwidth]{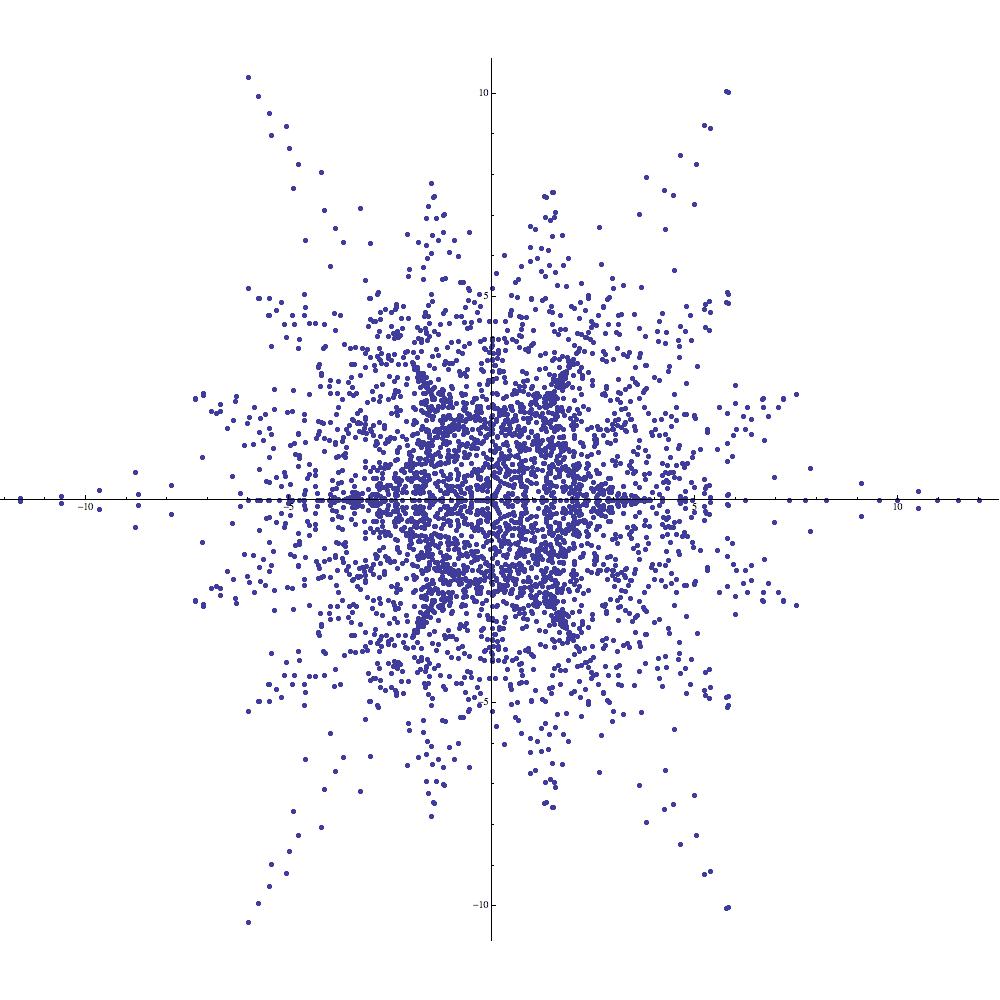}
			\caption{$j=5$}
			\label{SubfigureMR5}
		\end{subfigure}
		\bigskip

		\begin{subfigure}{0.3\textwidth}
			\centering
			\includegraphics[width=\textwidth]{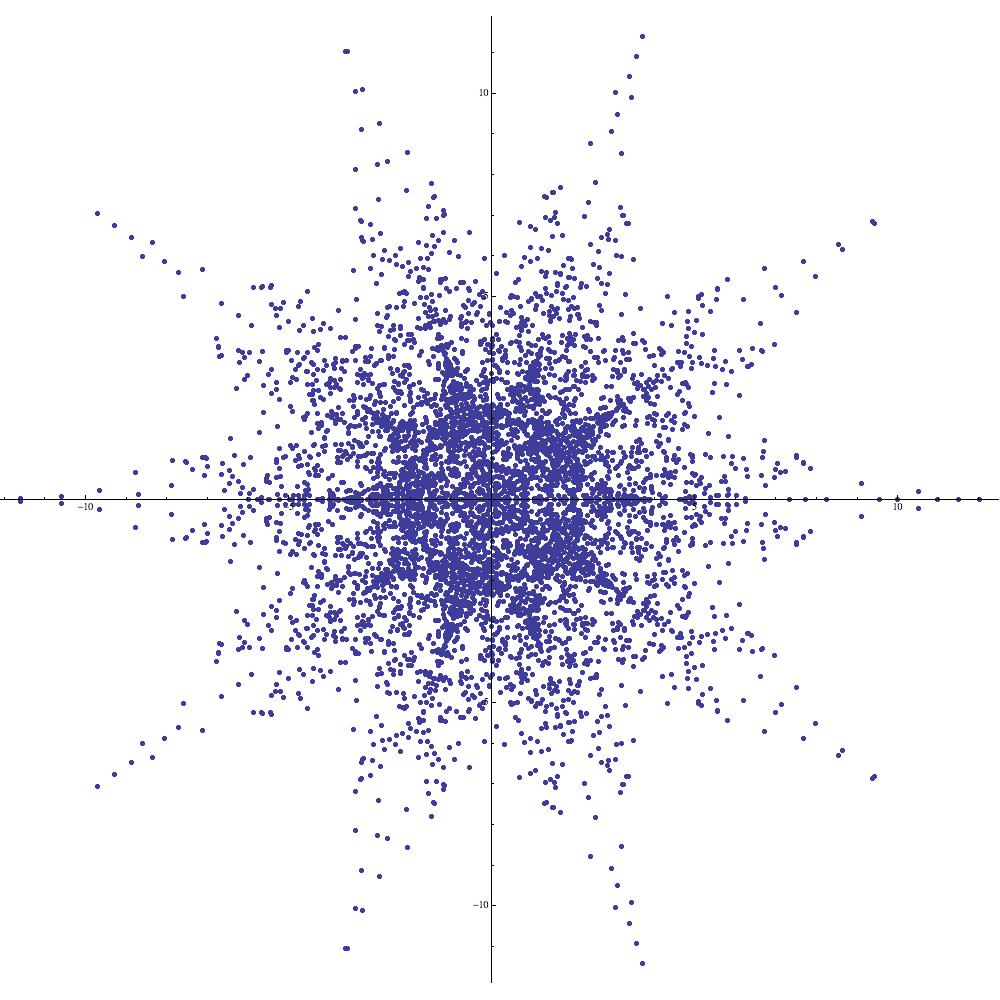}
			\caption{$j=3$}
			\label{SubfigureMR3}
		\end{subfigure}
		\quad
		\begin{subfigure}{0.3\textwidth}
			\centering
			\includegraphics[width=\textwidth]{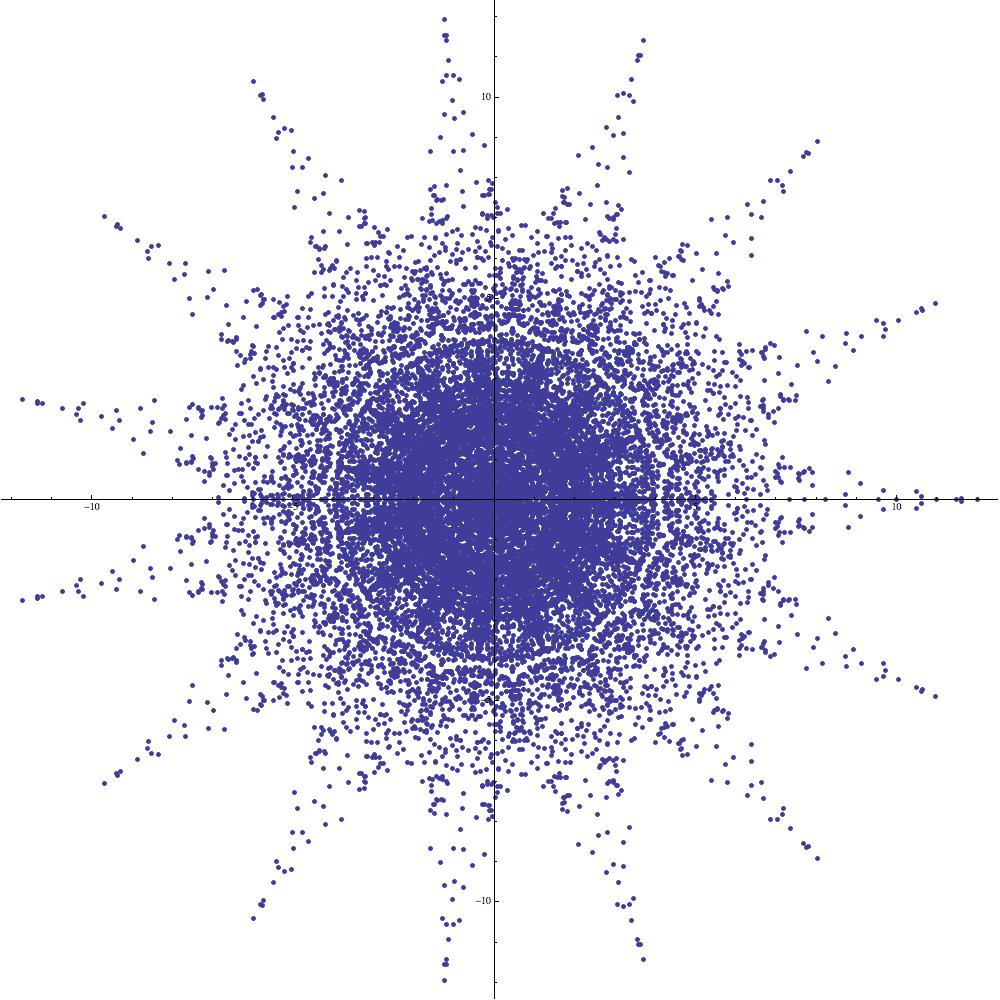}
			\caption{$j=2$}
			\label{SubfigureMR2}
		\end{subfigure}
		\quad
		\begin{subfigure}{0.3\textwidth}
			\centering
			\includegraphics[width=\textwidth]{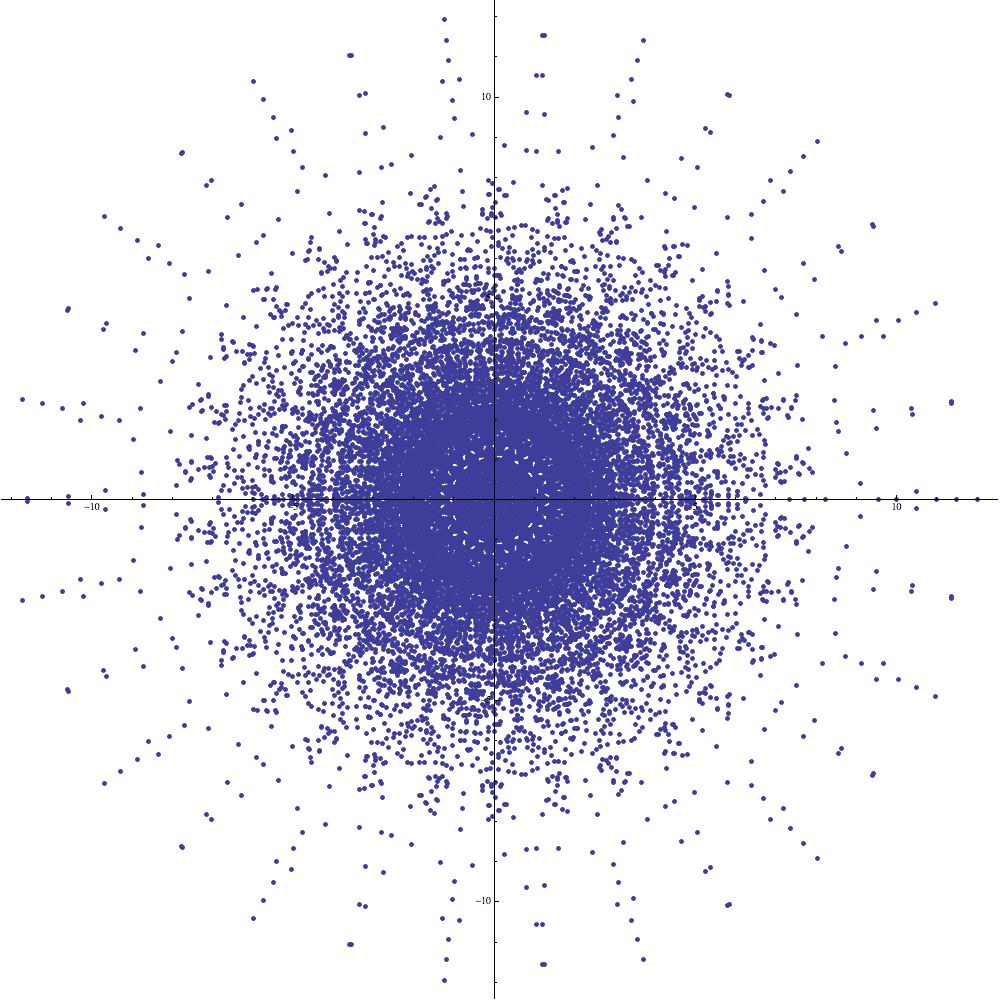}
			\caption{$j=1$}
			\label{SubfigureMR1}
		\end{subfigure}

		\caption{Images of the supercharacter $\sigma_X:(\Z/30\Z)^4\to\C$ where
		$X = S_4(0,1,1,28)+j\vec{1}$ for several values of $j$.}
		\label{FigureRotate}
	\end{figure}

\section{Open questions}\label{SectionOpen}
	As we have seen, symmetric supercharacters enjoy a vast and diverse array
	of striking visual features.  We have surveyed many of these phenomena and provided
	explanations for quite a few of them.  However, the discovery of these enigmatic plots and our
	subsequent investigations seem to generate more questions than answers.
	It is clear that our understanding of symmetric supercharacters is far from complete.
	Indeed, our work suggests that the deepest properties of these sums are yet to be discovered.
	In this spirit, we conclude this article with a number of open questions inspired by the
	preceding investigations.\footnote{In order to facilitate the work of other researchers, we have included in the following appendix
	the \texttt{Mathematica} code for generating supercharacter plots.}

	\begin{Question}
		What happens if one considers the action of the alternating group $A_d$ on $(\Z/n\Z)^d$ in place
		of $S_d$?  How do the resulting images relate to the images obtained from $S_d$?
	\end{Question}
	
	\begin{Question}
		Recall from the introduction that the values of a symmetric supercharacter can be interpreted in terms of matrix permanents.
		Does this interpretation shed any light on the properties of symmetric supercharacters?
	\end{Question}

	\begin{Question}
		Some families of supercharacters (e.g., Gaussian periods \cite{GNGP}, Kloosterman sums \cite{SESUP}, Ramanujan sums \cite{RSS}; 
		see Table \ref{FigureNT}) enjoy certain multiplicative properties.  Do symmetric supercharacters possess similar properties?
	\end{Question}

	\begin{Question}
		What is the mechanism which produces the nested pentagons in Figure \ref{FigureHook}(B)?  
	\end{Question}

	\begin{Question}
		What is the mechanism which produces the jagged linear features appearing in Figure \ref{FigureHook2}(A)?
	\end{Question}

	\begin{Question}
		What is the mechanism which produces the parabolic features appearing in Figure \ref{FigureHook2}(C)?
	\end{Question}

	\begin{Question}
		Examples \ref{ExampleHypo}, \ref{ExampleHumming}, and \ref{ExampleManta} all concern certain families of symmetric supercharacters
		which exhibit clearly discernible asymptotic behavior as the modulus $n$ tends to infinity (perhaps with some congruence restrictions). 
		Can one develop a general theory describing the asymptotic behavior of symmetric supercharacters?
	\end{Question}
	
	\begin{Question}
		If $X = -X$, then the corresponding supercharacter is real-valued (see Subsection \ref{SubsectionConjugate}
		and Figure \ref{FigureReal}).  Is it possible to describe the statistical distribution of the values of certain families of real-valued 
		symmetric supercharacters?
	\end{Question}

	\begin{Question}
		As mentioned in the introduction, the matrix \eqref{eq-U} encodes an analogue of the discrete Fourier transform
		on the space of all superclass functions \cite{SESUP}.  Are there compelling applications of such a transform?
	\end{Question}

	\begin{Question}
		Given the central role that the symmetric groups have played in this work, it is perhaps surprising that our approach has
		involved only a minimal level of combinatorial reasoning.  What properties of symmetric supercharacters can be discovered using 
		more sophisticated combinatorial tools?
	\end{Question}

	\begin{Question}
		Is there a relatively simple algebraic procedure which can be used to obtain the boundary 
		of the image of a function of the form \eqref{eq-g}?  For instance, is there a simple description of the
		boundary of the ``hummingbird'' from Figure \ref{FigureHummingbird}?
		Is there a simple description of the two regions suggested by the ``manta ray'' from Figure \ref{FigureManta}?
	\end{Question}
	
	\begin{Question}
		We have seen in Subsection \ref{SubsectionHypocycloids} that the orbit $X = S_d(1,1,\ldots,1,1-d)$ yields an image which, 
		for sufficiently large $n/(n,d)$, closely approximates the set of all traces of matrices in $SU(d)$.  Can one identify other families 
		of supercharacters whose asymptotic behavior is related to the traces of various subgroups of the unitary group $U(d)$?
	\end{Question}

\appendix

\section{\texttt{Mathematica} Code}\label{SectionCode}
We provide below the basic \texttt{Mathematica} code for producing plots of symmetric supercharacters.
\bigskip

\noindent
\texttt{superChar[n, char, class]} numerically evaluates the supercharacter corresponding to the $S_d$ orbit of \texttt{char} (where $d$ is the length of
\texttt{char}) acting
on the superclass \texttt{class};
e.g., \verb+superChar[7, {1,2,3,4}, {0,2,4,6}]+.
\smallskip
  
\begin{verbatim}
superChar[n_, char_, class_] := Module[{zeta, charOrbit, dotList},
  zeta = N[Exp[2 Pi I /n]];
  charOrbit = Permute[char, SymmetricGroup[Length[char]]];
  dotList = class.# & /@ charOrbit;
  Sum[zeta^dotList[[i]], {i, 1, Length[dotList]}]]
\end{verbatim}
\bigskip

\noindent
\texttt{singleChar[n, char]} gives a list of all values of the supercharacter corresponding to the $S_d$ orbit of \texttt{char};
e.g., \verb+singleChar[7, {1,2,3,4}]+.
\smallskip
  
\begin{verbatim}
singleChar[n_, char_] := Module[{d, elems},
  d = Length[char];
  elems = 
   First /@ 
    GroupOrbits[SymmetricGroup[d], Tuples[Range[0, n - 1], d], 
     Permute];
  Chop[ParallelMap[superChar[n, char, #] &, elems]]]
\end{verbatim}
\bigskip
  
\noindent
\texttt{singleCharVectorPlot[n, char]} creates a vector graphics plot of a symmetric supercharacter;
e.g., \verb+singleCharVectorPlot[7, {1,2,3,4}]+.
\smallskip

\begin{verbatim}
singleCharVectorPlot[n_, char_] := 
  ListPlot[{Re[#], Im[#]} & /@ 
    DeleteDuplicates[Chop[singleChar[n, char]], Equal], 
    AspectRatio -> Automatic, PlotRange -> All]
\end{verbatim}
\bigskip

\noindent
\texttt{singleCharBitmapPlot[n, char, range, unitRes]} creates a bitmap image of a supercharacter, 
with a range of \texttt{[-range, range]} along both axes and a unit resolution of \texttt{unitRes}; 
e.g., 
\verb+singleCharBitmapPlot[7, {1,2,3,4}, 30, 30]+.
\smallskip

\begin{verbatim}
singleCharBitmapPlot[n_, char_, range_, unitRes_] := 
 Module[{circ3, res, charRow, array, entries},
  circ3 = {{0.3, 0.75, 0.3}, {0.75, 1, 0.75}, {0.3, 0.75, 0.3}};
  res = unitRes*range;
  charRow = singleChar[n, char];
  array = Normal[SparseArray[{}, {2 res, 2 res}, 0]];
  entries = 
   ParallelMap[
    Round[({res - unitRes*Im[#], res + unitRes*Re[#]} &)@#] &, 
    charRow];
  Do[With[{r = entries[[j, 1]], c = entries[[j, 2]]}, 
    If[1 < r < 2 res && 1 < c < 2 res, 
     array[[r - 1 ;; r + 1, c - 1 ;; c + 1]] = circ3]],
   {j, Length[charRow]}];
  Image[1 - array]]
\end{verbatim}

\bibliography{GNSG}
\bibliographystyle{plain}
\end{document}